\documentclass[reqno, letterpaper, oneside]{article}
\usepackage{amsmath,amsfonts,amssymb,amsthm} 
\usepackage[final]{graphicx}	
\usepackage[usenames,dvipsnames]{color}

\usepackage{bbm}
\usepackage{setspace}	
\usepackage{geometry}
\usepackage{enumerate}

\usepackage{tikz}
\usepackage{extarrows}
\pgfdeclarelayer{nodelayer}
\pgfdeclarelayer{edgelayer}
\pgfsetlayers{edgelayer,nodelayer,main}
\tikzstyle{black}=[fill=black, draw=black, shape=circle, scale=0.3]
\tikzstyle{none}=[]

\usepackage{algorithm, algpseudocode, algorithmicx}
\makeatletter
\renewcommand{\fnum@algorithm}{} 
\makeatother

\usepackage[colorlinks=false,,linktoc=allbookmarksopen=true,linktocpage,pdftex]{hyperref}
\usepackage[pdftex]{bookmark}

\definecolor{burgundy}{rgb}{0.5, 0.0, 0.13}

\def\bk{\backslash}
\def\e{\epsilon}
\def\s{\sigma}
\def\G{\Gamma}

\newcommand{\ov}[1]{{\overline{#1}}}

\newtheorem{theorem}{Theorem}
\newtheorem{lemma}[theorem]{Lemma}
\newtheorem{observation}[theorem]{Observation}

\newtheorem{claim}[theorem]{Claim}

\newtheorem{remark}[theorem]{Remark}

\newtheorem*{definition}{Definition}

\newtheorem*{convention}{Notational Convention}
\theoremstyle{definition}

\newcommand{\refT}[1]{Theorem~\ref{#1}}

\newcommand{\refL}[1]{Lemma~\ref{#1}}
\newcommand{\refR}[1]{Remark~\ref{#1}}
\newcommand{\refS}[1]{Section~\ref{#1}}

\newcommand{\refF}[1]{Figure~\ref{#1}}
\newcommand{\refApp}[1]{Appendix~\ref{#1}}

\newcommand{\refO}[1]{Observation~\ref{#1}}

\newcommand\E{\operatorname{\mathbb E{}}}
\renewcommand\Pr{\operatorname{\mathbb P{}}}
\renewcommand\P{\operatorname{\mathbb P{}}}

\newcommand\bigpar[1]{\bigl(#1\bigr)}
\newcommand\Bigpar[1]{\Bigl(#1\Bigr)}
\newcommand\biggpar[1]{\biggl(#1\biggr)}
\newcommand\lrpar[1]{\left(#1\right)}
\newcommand\bigsqpar[1]{\bigl[#1\bigr]}
\newcommand\Bigsqpar[1]{\Bigl[#1\Bigr]}

\newcommand\bigcpar[1]{\bigl\{#1\bigr\}}

\newcommand\biggcpar[1]{\biggl\{#1\biggr\}}

\newcommand\bigabs[1]{\bigl|#1\bigr|}

\def\rompar(#1){\textup(#1\textup)}    

\def\xexp(#1){e^{#1}}
\newcommand\ceil[1]{\lceil#1\rceil}

\newcommand\biggceil[1]{\biggl\lceil#1\biggr\rceil}

\newcommand\floor[1]{\lfloor#1\rfloor}

\newcommand{\eps}{\epsilon}

\def\D{\Delta}
\def\e{\epsilon}
\def\s{\sigma}

\newcommand{\cG}{\mathcal{G}}

\newcommand{\cL}{\mathcal{L}}

\newcommand{\cN}{\mathcal{N}}

\newcommand{\cP}{\mathcal{P}}
\newcommand{\cQ}{\mathcal{Q}}
\newcommand{\cS}{\mathcal{S}}
\newcommand{\cT}{\mathcal{T}}
\newcommand{\cU}{\mathcal{U}}

\newcommand{\cB}{\mathcal{B}}

\newcommand{\cA}{\mathcal{A}} 

\newcommand\dd{\,\mathrm{d}}

\newcommand{\hrho}{\hat{\rho}}
\newcommand{\Xs}{X_k}

\newcommand{\abxy}{\sigma_{ab,xy}}
\newcommand{\xyab}{\sigma_{xy,ab}}
\newcommand{\axby}{\sigma_{ax,by}}
\newcommand{\byax}{\sigma_{by,ax}}

\renewcommand{\emptyset}{\varnothing} 
\newcommand\mydots{\thinspace\makebox[1em][c]{.\hfil.\hfil.}\thinspace}

\newcommand{\indic}[1]{\mathbbm{1}_{\{{#1}\}}}

\newcommand{\Gdn}{G_{\mathbf{d_n}}}
\newcommand{\Gcdn}{G^C_{\mathbf{d_n}}}
\newcommand{\Gpdn}{G^P_{\mathbf{d_n}}}
\newcommand{\Gpodn}{G^{P,*}_{\mathbf{d_n}}}
\newcommand{\Hpdn}[1]{H^P_{\mathbf{d_n},{#1}}}
\newcommand{\Hpodn}[1]{H^{P,*}_{\mathbf{d_n},{#1}}}
\newcommand{\Hppdn}[1]{H^{P,+}_{\mathbf{d_n},{#1}}}

\newcommand{\bd}{\mathbf{d}}
\newcommand{\bdn}{\mathbf{d_n}}
\newcommand{\bdnp}{\mathbf{d'_n}}

\newcommand{\Xd}{X}
\newcommand{\bE}{\operatorname{\mathbb E{}}}
\newcommand{\bP}{\operatorname{\mathbb P{}}}
\newcommand{\osigma}{\overline{\sigma}}

\newcommand{\hsigma}{\hat{\sigma}}
\newcommand{\stubs}{\sum_{j \in [k]} jn_j}
\newcommand{\bG}{\operatorname{\mathbb G{}}}

\newcommand{\im}{\operatorname{\mathrm{Im}}}

\newcommand{\Ud}{U}
\newcommand{\ud}{u}
\newcommand{\xd}{x}

\newenvironment{romenumerate}[1][-5pt]{
\addtolength{\leftmargini}{#1}\begin{enumerate}
 }{\end{enumerate}}

\let\OLDthebibliography\thebibliography
\renewcommand\thebibliography[1]{
  \OLDthebibliography{#1}
  \setlength{\parskip}{0pt}
  \setlength{\itemsep}{0pt plus 0.3ex}
}

\title{The degree-restricted random process is far from uniform}
\author{Michael Molloy\thanks{Dept of Computer Science,
University of Toronto, Toronto~M5S 3G4, Canada. E-mail: {\tt molloy@cs.toronto.edu}.  Supported by NSERC Discovery Grant 2019-06522.} \ 
and Erlang Surya\thanks{Department of Mathematics, University of California, San Diego, La Jolla CA~92093, USA. E-mail: {\tt esurya@ucsd.edu}. Supported by NSF~grant DMS-2225631.} \  
and Lutz Warnke\thanks{Department of Mathematics, University of California, San Diego, La Jolla CA~92093, USA. E-mail: {\tt lwarnke@ucsd.edu}. Supported by NSF~CAREER grant~DMS-2225631 and a Sloan Research Fellowship.}}
\date{April~5, 2023; revised August~11, 2025}

\begin{document}

\maketitle

\begin{abstract}
The degree-restricted random process is a natural algorithmic model for generating graphs with degree sequence~$\bd_n=(d_1, \ldots, d_n)$:
starting with an empty $n$-vertex graph, it sequentially adds new random edges so that the degree of each vertex~$v_i$ remains at most~$d_i$. 
Wormald conjectured in~1999 that, for $d$-regular degree sequences~$\bd_n$, 
the final graph of this process is similar to a uniform random $d$-regular~graph. 

In this paper we show that, for degree sequences~$\bd_n$ that are not nearly regular, 
the final graph of the degree-restricted random process differs~substantially from a uniform random graph with degree sequence~$\bd_n$. 
The combinatorial proof~technique is our main conceptual contribution:    
we adapt the switching method to the degree-restricted process, 
demonstrating that this enumeration technique 
can also be used to analyze stochastic processes (rather than just uniform random~models, as~before). 
\end{abstract}

\section{Introduction}
Random graph processes that grow step-by-step over time are 
powerful in both theory and practice: 
they are often used to generate sophisticated combinatorial objects with surprising properties~\cite{B2009,BK2010,FPGM2013,BK2013,BW2018}, 
and are also frequently used to model complex networks arising in applications~\cite{AlbertBarabasi,DorogovtsevMendes,Dorogovtsev,Barabasi,Newman}. 
Rather few proof techniques exist for analyzing such stochastic processes, 
which is 
why even the most basic questions about these processes are often hard to answer.
In this paper we add the so-called `switching method' to the list of 
techniques for analyzing stochastic processes.

This paper concerns the perhaps simplest random graph process that attempts to generate a graph with a given graphic\footnote{As usual, a degree-sequence $\bdn$ is called \emph{graphic} if it is the degree sequence of some simple graph. 
It is well known (and easy to check) that any degree sequence $\bdn={\bigpar{d_1^{(n)},\dots, d_n^{(n)}}} \in {\{0, \ldots, \Delta\}^n}$ is graphic when the degree sum $\sum_{i \in [n]}d_i^{(n)}$ is even and larger than some sufficiently large constant~$m'_0=m'_0(\Delta)$; see~\refL{completionprob}.} degree sequence~$\bdn=\bigpar{d_1^{(n)},\dots, d_n^{(n)}}$: 
the \emph{degree-restricted random $\bdn$-process} starts with an empty $n$-vertex graph, and then step-by-step adds a new random edge subject to the constraint that the degree of each vertex~$v_i$ remains at most~$d_i$ (without creating loops or multiple edges).
The \mbox{$\bdn$-process} is so natural that it has been studied since the mid 1980s 
in chemistry~\cite{KQ1985,BQ87,BQ90}, 
combinatorics~\cite{RW,Wreg,TWZ} and statistical physics~\cite{BNK}, 
often with a focus on $d$-regular degree sequences~$\bd_n$. 
For graphic degree sequences~$\bdn$ with constant maximum degree, 
Ruci{\'n}ski and Wormald~\cite{RW} showed in~1990 that the $\bdn$-process 
is a natural algorithmic model in the sense 
that it typically generates\footnote{The $\bdn$-process does not always generate a graph with the desired degree sequence. For example, using the graphic degree sequence~$\mathbf{d_4}=(2,2,2,2)$ it is possible to get stuck after three steps if we create a triangle.} a graph with degree sequence~$\bdn$. 
Conceptually the perhaps most interesting open problem remaining~\cite{BQ1993,BQ1993A,Wreg,TWZ} 
concerns its distribution: 
is the final graph of the $\bdn$-process similar to a  uniform random graphs with degree sequence~$\bdn$?

To make this question precise, let~$\Gpdn$ be the final graph of the $\bdn$-process conditioned on having degree sequence~$\bdn$, and let~$\Gdn$ be the uniform random graph with degree sequence~$\bdn$.
Since the distributions of~$\Gpdn$ and $\Gdn$ are not\footnote{The fact that the distributions of~$\Gpdn$ and~$\Gdn$ are not identical can be seen by inspection of specific sequences on a few vertices; see \refApp{counterexample} for an example.
For $2$-regular degree sequences~$\bdn$ certain expectations and probabilities also differ~\cite{RWshort,TWZ} slightly, 
but these minor differences do not rule out contiguity (since they do not concern high probability~events).} identical,
it is natural~\cite{RobW,sj,cfmr,mrrw,kw,sj2} 
to focus on similarity with respect to their typical properties. 
More formally, the remaining key question is whether $\Gpdn$ and $\Gdn$ are \emph{contiguous}, i.e., if every property that holds with~high~probability\footnote{As usual, we say that a graph property holds~\emph{with~high~probability} if it holds with probability tending to~$1$ as~$n\to \infty$.} in one also holds with~high~probability in the other (see~\cite[Section~9.6]{JLR}).  
In~fact, Wormald conjectured in~1999 that~$\Gpdn$ and~$\Gdn$ are contiguous for $d$-regular degree sequences~$\bdn$, for any fixed~$d\geq 2$ (see~\cite[Conjecture~6.3]{Wreg}). 
Until now these contiguity questions have remained open, 
partly because we are lacking suitable proof techniques for such 
questions in random graph~processes such as~$\Gpdn$ (which are much harder to analyze\footnote{For example, a moment of reflection reveals that even estimating~$\Pr(\Gpdn=G)$ is challenging; see~\cite[Section~1]{RW}.} than uniform random graphs such~as~$\Gdn$).

In this paper we show that, for graphic degree sequences~$\bdn$ with constant maximum degree~$\Delta=O(1)$, 
the final graph~$\Gpdn$ of the $\bdn$-process is \emph{not}~contiguous to the uniform random graph~$\Gdn$, 
unless the degree sequence~$\bdn$ is nearly regular (which means that all but~$o(n)$ vertices have the same degree). 
In other words, the final graph of the degree-restricted random process differs~substantially from a uniform random graph with degree sequence~$\bd_n$. 
Here we denote the number of vertices in~$\bdn$ with degree~$j$ by 
\[n_j:=n_j(\bdn):=\bigabs{\bigcpar{v\in \{1,\dots, n\} \: : \: d^{(n)}_v=j}}. \]
%
\begin{theorem}[The degree restricted process is far from uniform]\label{final}
Fix two constants: an integer~${\Delta > 1}$ and a real~${\xi \in (0,1)}$. 
Assume that the graphic degree sequence ${\bd_n=\bigpar{d_1^{(n)},\dots, d_n^{(n)}}\in \{1,2, \ldots, \Delta \}^n}$
 satisfies the {`non-regularity'} assumption~${\max_{1\le j\le \Delta} n_j/n\le 1-\xi}$. 
Then the final graph~$\Gpdn$ of the degree restricted random \mbox{$\bdn$-process}
is not~contiguous to 
the uniform random graph~$\Gdn$ with degree sequence~$\bd_n$. 
\end{theorem}
This theorem follows from a more general result we shall present in~\refS{sec:mainresult},
which shows that a certain edge-statistic typically differs substantially between~$\Gpdn$ and~$\Gdn$.
This yields a stronger discrepancy between~$\Gpdn$ and $\Gdn$ than \mbox{non-contiguity}, 
including that the total variation distance between them is nearly~maximal, 
and also the existence of a simple 
algorithm that can distinguish between~them.

The combinatorial proof~technique used for \refT{final} is our main conceptual contribution: 
we adapt the switching method~\cite{McKay1984, GodsilMcKay1984, McKayWormald1991, Wreg, McKay2010, GM2013, GW2016} to the degree-restricted random $\bdn$-process, 
demonstrating that this 
enumeration technique can also be used to analyze stochastic processes (rather than just uniform random~models, as~before). 
Interestingly, this is not the first time that 
the \mbox{$\bdn$-process} has stimulated the development of a new proof technique for stochastic processes: 
indeed, the widely-used 
differential equation method~\mbox{\cite{RW,DEM,DEM99}} as well as the associated self-correction idea~\cite{RWdc,TWZ} were 
first developed for the random~\mbox{$\bdn$-process};   
both were crucial for later breakthroughs in Ramsey Theory~\cite{B2009,BK2010,FPGM2013,BK2013}, demonstrating the 
potential of developing new proof techniques for stochastic processes such as the random $\bdn$-process.

\subsection{Main result: discrepancy in edge-statistics}\label{sec:mainresult} 
Our main result shows that the number of certain edges differs substantially between the degree restricted random process and a uniform random graph with degree sequence~$\bd_n$.
To formalize this difference, note that after rescaling~$\xi$ to~$\xi/2$ (as we may), 
the `non-regularity' assumption of \refT{final} implies that
\begin{equation}\label{keyassumption}
\sum_{1\le j\le k} n_j/n \in [\xi,1-\xi]
\end{equation}
for some integer~$1 \le k <\D$.  
For any graph~$G$, let $\Xs(G)$ denote the number of edges in $G$ whose endpoints are of degree at most~$k$; we call such edges {\em small~edges}. 
\refT{transferred} with $\eps=\beta/2$, say, shows that~$\Xs(\Gpdn)$ and~$\Xs(\Gdn)$ differ with high probability, which implies the non-contiguity result~\refT{final}. 
Here 
\begin{equation}\label{def:mu}
\mu = \mu(\bdn,k) := \frac{(\sum_{1\le j\le k} jn_j)^2}{4m} \qquad \text{ with } \qquad 
m = m(\bdn) := \frac{\sum_{1\le j\le \Delta} jn_j}{2} , 
\end{equation}
where~$\mu$ approximates the expected number of small edges in the uniform random graph~$\Gdn$ (see~\refApp{sec:uniformtransfer}), 
and~$m$ equals the number of edges of the degree sequence~$\bdn$. 
\pagebreak[2]
\begin{theorem}[Discrepancy in number of small edges]\label{transferred} %
Fix two constants: an integer~${\Delta > 1}$ and a real~${\xi \in (0,1)}$.
Assume that the graphic degree sequence~$\bdn\in \{1,\dots,\Delta\}^n$ satisfies for some integer~$k=k(\bdn) \in {\{1, \ldots, \Delta-1\}}$ the \mbox{`non-regularity'} assumption~\eqref{keyassumption}.
Then the following holds, with~$\mu = \mu(\bdn,k)$ as defined in~\eqref{def:mu}:\vspace{-0.5em}%
\begin{romenumerate}
\itemindent 0.125em \itemsep 0.125em \parskip 0.125em 
		\item For any~$\eps=\eps(n) \gg n^{-1/2}$, the uniform random graph~$\Gdn$ with degree sequence~$\bdn$ satisfies \label{transferred:uniform}
\begin{equation}\label{eq:transferred:uniform}
\Pr\bigpar{\bigabs{\Xs\bigpar{\Gdn}-\mu} < \eps \mu} \ge 1-e^{-\Theta(\eps^2n)}.
\end{equation}
		\item There is a constant~$\beta=\beta(\xi, \Delta)>0$ such that the final graph~$\Gpdn$ of the \mbox{$\bdn$-process} satisfies \label{transferred:process}
\begin{equation}\label{eq:transferred:process}
\bP\bigpar{\bigabs{\Xs\bigpar{\Gpdn}-\mu} > \beta \mu} \ge 1-e^{-\Theta(n)}.
\end{equation}
\end{romenumerate}
\end{theorem}
Invoking again \refT{transferred} with~$\eps=\beta/2$, note that, with probability at least $1-e^{-\Theta(n)}$, 
the number of small edges satisfies~$|\Xs(\Gpdn)-\mu| > \beta \mu$ and~$|\Xs(\Gdn)-\mu| < \tfrac{1}{2}\beta \mu$.
This striking difference between~$\Gpdn$ and $\Gdn$ has a number of interesting consequences: 
(a) a polynomial algorithm can tell~$\Gpdn$ and~$\Gdn$ apart by simply counting the number of small edges, 
(b) the edit-distance\footnote{Given two graphs~$G_1, G_2$ with the same number of vertices, the \emph{edit-distance} between them is defined as the minimum number of edge-changes (addition or removal) one needs to apply to~$G_1$ in order to obtain a graph that is isomorphic to~$G_2$.}
between~$\Gpdn$ and $\Gdn$ is at least~${\tfrac{1}{2}\beta\mu = \Theta(n)}$, i.e., very large, 
and (c) the total variation distance\footnote{As usual, the \emph{total variation distance}~$\mathrm{d}_{\mathrm{TV}}(\Gpdn,\Gdn)$ is defined as~$\tfrac{1}{2}\sum_G|\Pr(\Gpdn=G)-\Pr(\Gdn=G)|$, where the sum is taken over all possible graphs~$G$ with degree sequence~$\bdn$.}
between~$\Gpdn$ and~$\Gdn$ is nearly maximal, i.e., close to~one: 
\begin{equation}
\bigabs{\mathrm{d}_{\mathrm{TV}}\bigpar{\Gpdn, \: \Gdn}-1} \le e^{-\Theta(n)}.
\end{equation}

The proof of inequality~\eqref{eq:transferred:uniform} for the uniform random graph with degree sequence~$\bdn$ is standard: 
it is based on routine configuration model~\mbox{\cite{BB1980,Wreg,FK}} arguments, and appears in \refApp{sec:uniformtransfer}.  
Our main contribution is the proof of inequality~\eqref{eq:transferred:process} for the \mbox{$\bdn$-process}: 
it is based on an intricate edge-switching argument involving the trajectories of the random $\bdn$-process, 
see the proof outline in~\refS{ssps} and the technical details in Sections~\mbox{\ref{setupandcore}--\ref{sec:counting}}.
In~\refS{sec:diffeq} we also mention a potential alternative approach 
based on the differential equation method, which so far has resisted rigorous analysis (which is why we resorted to the switching~method).

We now give some intuition as to why the number~$X_k$ of small edges differs between~$\Gdn$ and~$\Gpdn$. 
Namely, if we construct a random graph with degree sequence~$\bdn$ using the standard configuration model by adding edges step-by-step (one at a time), 
then in each step the probability of selecting a vertex~$v$ as an endpoint is proportional to the current unused degree of that vertex. 
By contrast, in the degree restricted $\bdn$-process all vertices with positive unused degree are (approximately) equally likely to be selected in each~step. 
This `preferential versus uniform' difference heuristically suggests that small edges tend to be created with higher probability in the $\bdn$-process, 
which makes the edge-discrepancies in \refT{transferred} plausible (see also~\refR{rem:lowertail}).


\subsection{Proof strategy: switching method}\label{ssps}
%
%
The basic idea of the switching method 
is to estimate ratios of closely related set-sizes via local perturbations: 
namely, by defining a suitable `switching operation' that maps objects in~$\cA$ to objects in~$\cB$, one can often estimate the set-size ratio~$|\cA|/|\cB|$ fairly precisely via a double-counting argument 
(by estimating the number of switchings from~$\cA$ to~$\cB$ and the number of inverse
switchings from~$\cB$ to~$\cA$).
This has 
been successfully used to \mbox{enumerate} many combinatorial structures and analyze uniform models thereof, 
including graphs with a given degree sequence~\mbox{\cite{McKay1984,McKayWormald1991,Wreg,GM2013,GW2016,JPRR2018}}, \mbox{$0$-$1$-matrices}~\mbox{\cite{GMX2006,McKay2010}}, regular hypergraphs~\cite{DFRS2013}, structured graph classes~\cite{NWpert1996,GS2016,CP2019},  Latin rectangles~\cite{GodsilMcKay1984} and Latin squares~\cite{KS2018,
KSS2022}.

Our proof of \refT{transferred}~\ref{transferred:process} adapts the switching method to the degree restricted random~\mbox{$\bdn$-process}, 
which requires several new ideas.
Indeed, for uniform random structures such as~$\Gdn$ the normalizing constants in~${\Pr(\Gdn \in \cA)/\Pr(\Gdn \in \cB)}={|\cA|/|\cB|}$ cancel, and the switching method directly applies. 
This simplification does not happen for stochastic processes like the \mbox{$\bdn$-process}, 
so we needed to develop a form of the switching method that allows for different probabilities:
roughly speaking, 
we (i)~apply switching operations directly to the trajectories of the random \mbox{$\bdn$-process}, 
and then (ii)~average over the probabilities of these trajectories and their corresponding~graphs; 
both ideas should also aid analysis of other stochastic~processes. 

The starting point of our proof~strategy
is the following well-known\footnote{This switching operation was introduced by Peterson~\cite{P1891,PM92} in year~1891 (to analyze structural properties 
of regular~graphs).}  
\emph{\mbox{switching}~\mbox{operation} on~graphs}:
given a graph~$G^+$ with degree sequence~$\bdn$ where the two edges $ab, xy$ satisfy the degree constraints ${\max\{\deg(a), \deg(b)\}\le k}$ and ${\min\{\deg(x),\deg(y)\}>k}$, 
we write $G^-$ for the graph obtained by replacing\footnote{See Sections~\ref{sec:configurationgraph}--\ref{sec:relaxed} for how we deal with the possibility that this replacement might create multiple edges.} the edges~$ab,xy$ with the edges~$ax,by$; see \refF{fig:switching} for an example with~$k=1$. 
Note that~$G^+$ and~$G^-$ both have the same degree sequence~$\bdn$, and, 
more importantly, that~$G^+$ has exactly one small edge more~than~$G^-$,~i.e., 
\begin{equation}\label{eq:switch}
X_k(G^+)=X_k(G^-)+1.
\end{equation}
Since our goal is to show that the $\bdn$-process prefers more small edges than the uniform model~$\Gdn$, 
in view of~${\P(\Gdn=G^+)/\P(\Gdn=G^-)}=1$ 
we thus would like to prove the probability ratio~estimate 
\begin{equation}\label{wish}
 \frac{\P\bigpar{\Gpdn=G^+}}{\P\bigpar{\Gpdn=G^-}}\ge 1+\eps
\end{equation}
for some constant~$\eps>0$,  
as this would imply the desired discrepancy in the number of small~edges. 

In \refS{simplestrategy} we first sketch a switching argument for establishing~\eqref{wish} in the special case~${k=1}$, and outline how it implies \refT{transferred}~\ref{transferred:process}.
Afterwards, in \refS{generalstrategy} 
we discuss how we modify this switching argument for the more interesting general case~${k\geq 1}$, where we need to deal with the major obstacle that~\eqref{wish} is not always true: see \refF{fig:counter} in  \refApp{counterexample} for a~counterexample with~$k=2$. 
Finally, in \refS{sec:switching:comp} we compare our approach with classical switching arguments (that do not concern stochastic~processes). 

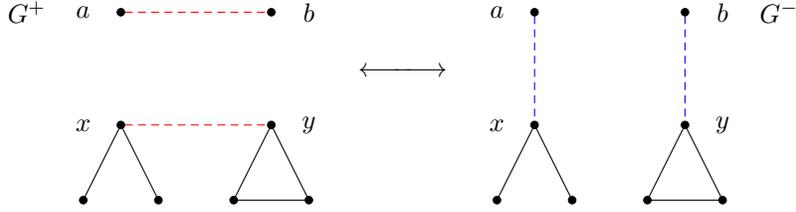
\begin{figure}
\centering
\begin{tikzpicture}
	\begin{pgfonlayer}{nodelayer}
		\node [style=black] (0) at (-2.25, 2) {};
		\node [style=black] (1) at (-0.25, 2) {};
		\node [style=black] (2) at (-2.25, 0.5) {};
		\node [style=black] (3) at (-0.25, 0.5) {};
		\node [style=none] (4) at (1.5, 1.35) {$ \xlongleftrightarrow{\phantom{iiiiiiiii}}$};
		\node [style=none] (9) at (-2.75, 2) {$a$};
		\node [style=none] (10) at (0.25, 2) {$b$};
		\node [style=none] (11) at (-2.75, 0.5) {$x$};
		\node [style=none] (12) at (0.25, 0.5) {$y$};
		\node [style=black] (13) at (-2.75, -0.5) {};
		\node [style=black] (14) at (-1.75, -0.5) {};
		\node [style=black] (15) at (-0.75, -0.5) {};
		\node [style=black] (16) at (0.25, -0.5) {};
		\node [style=black] (17) at (3.25, 2) {};
		\node [style=black] (18) at (5.25, 2) {};
		\node [style=black] (19) at (3.25, 0.5) {};
		\node [style=black] (20) at (5.25, 0.5) {};
		\node [style=none] (21) at (2.75, 2) {$a$};
		\node [style=none] (22) at (5.75, 2) {$b$};
		\node [style=none] (23) at (2.75, 0.5) {$x$};
		\node [style=none] (24) at (5.75, 0.5) {$y$};
		\node [style=black] (25) at (2.75, -0.5) {};
		\node [style=black] (26) at (3.75, -0.5) {};
		\node [style=black] (27) at (4.75, -0.5) {};
		\node [style=black] (28) at (5.75, -0.5) {};
		\node [style=none] (29) at (-3.5, 2) {$G^+$};
		\node [style=none] (30) at (6.5, 2) {$G^-$};
	\end{pgfonlayer}
	\begin{pgfonlayer}{edgelayer}
		\draw[red, densely dashed] (0) to (1);
		\draw[red, densely dashed] (2) to (3);
		\draw (2) to (13);
		\draw (2) to (14);
		\draw (3) to (15);
		\draw (15) to (16);
		\draw (3) to (16);
		\draw (19) to (25);
		\draw (19) to (26);
		\draw (20) to (27);
		\draw (27) to (28);
		\draw (20) to (28);
		\draw[blue, densely dashed] (17) to (19);
		\draw[blue, densely dashed] (18) to (20);
	\end{pgfonlayer}
\end{tikzpicture}

\caption{Switching example: the edges $ab,xy$ in~$G^+$ are replaced with the edges~$ax,by$ to obtain~$G^-$ (all other edges remain unchanged). Note that $G^+$ and $G^-$ have the same vertex degrees, and that $G^+$ has one more edge than~$G^-$ where both endpoints have degree one (which are so-called small~edges).\label{fig:switching}}%
\end{figure}

\subsubsection{Special case~$k=1$: comparison of edge-sequences of graphs}\label{simplestrategy}
To prove the probability ratio estimate~\eqref{wish} in the special case~$k=1$, we look at the trajectories of the \mbox{$\bdn$-process},
which means that we take the order of the added edges into account (this is important, since different orderings of the same edge set can occur with different probabilities). 
In particular, by summing over the set~$\Pi_G$ of all ordered edge-sequences of a given graph~$G$,
the probability that the $\bdn$-process produces~$G$~is
\begin{equation}\label{eq:Pr:expand}
\bP\bigpar{\Gpdn=G}=\sum_{\sigma\in \Pi_G} \bP(\sigma),
\end{equation}
where~$\bP(\sigma)$ denotes the probability that the $\bdn$-process produces the edge-sequence~$\sigma$.
%
%
%
Recall that, in our switching operation on graphs, the edges $ab,xy$ in~$G^+$ are replaced with the edges~$ax,by$ to obtain~$G^-$; 
see \refF{fig:switching}. 
Let~$\abxy$ be any edge-sequence of~$G^+$ where the edge~$ab$ appears earlier than the edge~$xy$. 
For our new \emph{switching operation on edge-sequences}, it is natural to consider the following two edge-sequences of $G^-$: 
(i)~the edge-sequence obtained by replacing~$ab,xy$ with~$ax,by$ in that order, which we call $\axby$, 
and (ii)~the one obtained by replacing~$ab,xy$ with~$by,ax$ in that order, which we call~$\byax$. 
To complete the symmetry, we also consider the edge-sequence of $G^+$ obtained by swapping the positions of~$ab,xy$ in~$\abxy$, which we call~$\xyab$.
By expanding~$\bP(\sigma)$ and carefully estimating the resulting formula (see~\eqref{def:Psigma}--\eqref{def:Zsigma} in \refS{sec:probabilities}), 
it turns out that one can prove the surprisingly clean
ratio~estimate
\begin{equation}\label{keywish}
\frac{\bP({\abxy})+\bP({\xyab})}{\bP({\axby})+\bP({\byax})}\ge 1,
\end{equation}
which already implies the desired probability ratio estimate~\eqref{wish} with~$\eps=0$ by noting that 
\begin{equation}\label{keywish:sum}
\bP\bigpar{\Gpdn=G^+}=\sum_{\abxy}\bigsqpar{\bP({\abxy})+\bP({\xyab})}\ge \sum_{\axby}\bigsqpar{\bP({\axby})+\bP({\byax})}= \bP\bigpar{\Gpdn=G^-}.
\end{equation}
With more work one can show that, for a positive proportion of the edge-sequences~$\abxy$, we have 
\begin{equation}\label{keywish2}
\frac{\bP({\abxy})+\bP({\xyab})}{\bP({\axby})+\bP({\byax})}\ge 1+\eps'
\end{equation}
for some constant~$\eps'>0$, 
which by similar reasoning as for inequality~\eqref{keywish:sum} then proves the desired probability ratio estimate~\eqref{wish} for an appropriately defined~$\eps>0$, which was our main goal.

For the interested reader, we now outline  how~\eqref{wish} implies \refT{transferred}~\ref{transferred:process}. 
To this end, let
\begin{equation}\label{def:Gell}
\cG_\ell :=\bigcpar{G \: : \: \text{graph with degree sequence }  \bdn \text{ and } \Xs(G)= \ell } .
\end{equation}
Note that `switching' the edges $ab,xy$ to $ax,by$ maps a graph from~$\cG_{\ell+1}$ to a graph in~$\cG_{\ell}$. 
When the number of small edges satisfies~$\ell\approx \mu$, 
then by counting the number of ways each graph in~$\cG_{\ell+1}$ can be mapped to a graph in~$\cG_{\ell}$, and vice versa, 
one can obtain the size ratio~estimate~$|\mathcal{G}_{\ell}|/|\mathcal{G}_{\ell+1}| \approx 1$ 
via a double counting argument that is standard for switching arguments. 
Combining these ideas with~\eqref{wish}, 
the crux is that by a more careful double counting argument one can
also obtain the probability ratio~estimate
\begin{equation}\label{simplecluster}
\frac{\bP\bigpar{X_k\bigpar{\Gpdn}=\ell}}{\bP\bigpar{X_k\bigpar{\Gpdn}=\ell+1}} = \frac{\sum_{F \in \cG_\ell}\bP\bigpar{\Gpdn=F}}{\sum_{H \in \cG_{\ell+1}}\bP\bigpar{\Gpdn=H}} \le 1-\tau
\end{equation}
for an appropriately defined~$\tau \in (0,1)$; 
in fact, estimate~\eqref{simplecluster} remains true as long as~${|\ell-\mu| \le \gamma \mu}$ for some sufficiently small~$\gamma>0$. 
For any~$0 \le z \le \gamma \mu$ we then obtain
the probability ratio~estimate
\begin{equation}\label{sec1:simplekey}
\frac{\bP\bigpar{\bigabs{X_k(\Gpdn)-\mu}\le z}}{\bP\bigpar{\bigabs{X_k\bigpar{\Gpdn}-\mu}\le z+1}}
 \le \frac{\sum_{\ell:|\ell-\mu|\le z}\bP\bigpar{X_k\bigpar{\Gpdn}=\ell}}{\sum_{\ell:|\ell-\mu|\le z}\bP\bigpar{X_k\bigpar{\Gpdn}=\ell+1}}  \le 1-\tau,
\end{equation}
and the proof of~inequality~\eqref{eq:transferred:process} 
follows readily for~$\beta :=\gamma/2$: by invoking~\eqref{sec1:simplekey} we have
\begin{equation}\label{simplefinishing}
\bP\bigpar{\bigabs{X_k\bigpar{\Gpdn}-\mu}\le \beta \mu} \le \prod_{0\le i\le \floor{\beta\mu}-1}\hspace{-0.125em}\frac{\bP\bigpar{ \bigabs{X_k\bigpar{\Gpdn}-\mu}\le \beta\mu+i}}{\bP\bigpar{ \bigabs{X_k\bigpar{\Gpdn}-\mu}\le \beta\mu+i+1}} \le \left(1-\tau\right)^{\floor{\beta\mu}} \le e^{-\Theta(n)},
\end{equation}  
where for the last inequality we used the fact that $\mu\geq (\xi n)^2/(2\D n) = \Theta(n)$.

\subsubsection{General case~$k \ge 1$: comparison of edge-sequences of sets of graphs}\label{generalstrategy}
With the benefit of hindsight, our switching argument for~$k=1$ can be summarized as follows.
We started with switching inequalities~\eqref{keywish} and~\eqref{keywish2} for edge-sequences, 
which we then sequentially `lifted' to the basic switching inequality~\eqref{wish} for graphs, 
then to the key   inequality~\eqref{simplecluster} for closely related sets of graphs, 
and finally to the desired inequality~\eqref{simplefinishing} for edge statistics. 
Our general argument for~$k \ge 1$ will go through a similar (albeit more complicated) lifting of~inequalities,
but there is one major obstacle: 
the basic  switching inequality~\eqref{wish} for graphs is not always true, even with~$\e=0$; see \refApp{counterexample} for a counterexample~with~${k=2}$.

Our strategy for overcoming this obstacle is based on the idea of averaging over suitable sets of graphs, called \emph{clusters} (see \refS{sec:cluster}). 
To this end, for any specific choice of our vertices $a,b,x,y$, we will partition the choices for $G^+$ (i.e., graphs containing edges $ab,xy$) into sets, which we call {\em upper clusters}.  For each upper cluster $C^+$, we also construct the corresponding set $C^-$ of the choices for $G^+$ obtained by replacing $ab,xy$ by $ax,by$ in each member of $C^+$; we call these sets {\em lower clusters}. 
We will essentially prove that~\eqref{wish} holds when averaging over graphs in the  corresponding clusters: 
namely, for every cluster pair~${C^+,C^-}$ we~obtain
\begin{equation}\label{cluster1}
\bP(C^+):=\sum_{G\in C^+} \bP(\Gpdn=G)\ge \sum_{G\in C^-} \bP(\Gpdn=G)=: \bP(C^-),
\end{equation}
and for a positive proportion of the cluster pairs~${C^+, C^-}$ we also obtain the stronger ratio estimate 
\begin{equation}\label{cluster2}
\frac{\bP(C^+)}{\bP(C^-)}\ge 1+\eps .
\end{equation}
for some constant~$\eps>0$ (see Lemmas~\ref{switch}--\ref{goodswitch} in \refS{sec:switchingcounting}).

To prove the probability ratio estimates~\eqref{cluster1}--\eqref{cluster2}, we will establish switching inequalities for edge-sequences analogous to~\eqref{keywish}--\eqref{keywish2} from the case~$k=1$, the main difference being as follows: while in~\eqref{keywish}--\eqref{keywish2} we matched two edge-sequences of~$G^+$ with two edge-sequences of $G^-$, for~\eqref{cluster1}--\eqref{cluster2} we shall match two edge-sequences of each graph in~$C^+$ with two edge-sequences of graphs in~$C^-$ (but not necessarily from the same graph; see~\refS{sec:switching}).
%

Finally, with estimates~\eqref{cluster1}--\eqref{cluster2} in hand, we can then  prove Theorem~\ref{transferred}~\ref{transferred:process} similarly to inequalities~\mbox{\eqref{simplecluster}--\eqref{simplefinishing}} from the case~$k=1$, but there is another difference 
that makes the details more complicated: 
whereas the graph switching inequality~\eqref{wish} holds for every pair of graphs~$G^+,G^-$, 
the cluster switching inequality~\eqref{cluster2} only holds for some pairs of clusters~$C^+,C^-$ (see~\refS{sec:doublecounting} for the details). 

\subsubsection{Another look at our switching argument}\label{sec:switching:comp}
We close this proof strategy subsection with a high-level discussion of how our switching argument compares to previous switching arguments applied to uniform random models, 
bearing in mind that the switching inequality~\eqref{simplecluster} is at the core of the argument leading to the telescoping product of ratios estimate~\eqref{simplefinishing}. 

In particular, if we were analyzing uniform random graphs~$\Gdn$ with degree sequence $\bdn$, then in view of~${\bP(\Gdn=F)}={\bP(\Gdn=H)}$ 
the left-hand side of the corresponding key-inequality~\eqref{simplecluster} would reduce to~$|\cG_{\ell}|/|\cG_{\ell+1}|$.
Obtaining tight asymptotic bounds on such ratios of closely related set-sizes is usually the main step in switching papers  
(see, e.g.~\cite{McKay1984, GodsilMcKay1984, McKayWormald1991, Wreg, GMX2006, GW2016}, which also use different 
switching~operations).

By contrast, in our analysis of the random  \mbox{$\bdn$-process} we nearly automatically have~$|\cG_{\ell}|/|\cG_{\ell+1}| \approx 1$, since we may focus on~$\ell \approx \mu$.
The main step in our proof of~\eqref{simplecluster} thus is to compare the average value of~$\Pr(\Gpdn=F)$ over graphs~$F\in\cG_{\ell}$ to the average of $\Pr(\Gpdn=H)$ over graphs~$H\in\cG_{\ell+1}$, 
which (together with the idea of looking at edge-sequences, i.e., trajectories of the \mbox{$\bdn$-process}) is at the heart of our use~of~switchings.

\section{Relaxed $\bdn$-process: main technical result}\label{sec:maintechnical:sec}
In this section we state our main technical result \refT{main} for a relaxed variant of the $d_n$-process, which allows for multiple edges but no loops. 
Of course, this technical result is formulated in a way that will imply our main result \refT{transferred}~\ref{transferred:process} for the standard $d_n$-process, see~\refS{sec:maintechnical} and~\refApp{sec:transfer}.  
Using the relaxed \mbox{$d_n$-process} for the proofs has two notable technical advantages: 
(a)~by allowing multiple edges, in the switching arguments we do not have to worry about switching to an edge that already exists, and 
(b)~by forbidding loops, we obtain a simpler and more tractable formula\footnote{Note that if we allowed for loops, then the degree-constraint would still forbid adding loops at vertices~$v$
whose current degree is~$\deg(v)=d^{(n)}_v-1$, which would considerably complicate the probability formulas in~\refS{sec:probabilities}.}
 for the probabilities of the process.

\subsection{Definition: configuration-graphs}\label{sec:configurationgraph}
Inspired by the configuration model~\cite{BB1980,Wreg,FK}, in this subsection we define a \textit{configuration-graph} with degree sequence~$\bdn=\bigpar{d_1^{(n)},\dots, d_n^{(n)}}$ as follows.  
For each vertex $v_i$ in a graph we imagine a bin containing $d_i^{(n)}$ labeled points. By a configuration-graph~$G$ with degree sequence~$\bdn$ we mean a perfect matching of all points where we do not allow points from the same vertex to be matched. Formally, $G$ is thus a perfect matching on the set of~points
\[ 
P(G) :=\bigcpar{v_i^j \: : \: i \in [n] \text{ and } 1\le j \le d^{(n)}_i },
\]
subject to having no matched pairs of the form~${v_i^jv_i^\ell}$ (to avoid creating loops).

As in the configuration model, a configuration-graph~$G$ naturally represents a multigraph~$\pi(G)$ with degree sequence~$\bdn$, obtained by contracting the set of points 
\[
V_i := \bigcpar{v_i^j \: : \: 1\le j\le d^{(n)}_i}
\] 
to vertex~$v_i$ for each~$i \in [n]$. 
Note that the multigraph~$\pi(G)$ may contain multiple edges but not loops. 

With the multigraph representation in mind, we say that~$P(G)$ are \emph{points} of~$G$, while the sets~$V_i$ with~${i \in [n]}$  
are \emph{vertices} of the configuration-graph~$G$. 
Furthermore, the \emph{degree of a vertex} in a configuration-graph is its size, so that vertex~$V_i$ has degree~$\deg(V_i):=|V_i|=d^{(n)}_i$. 
We also say that the \emph{degree of a point} in a configuration-graph is the size of the vertex that contains it, 
so that point~$w \in V_i$ has degree~$\deg(w):=|V_i|=d^{(n)}_i$. 
Hence, as with graphs,  $\Xs(G)$ is the number of edges whose endpoints are of degree at most $k$, meaning that the endpoints are contained in vertices of size at most~$k$.

\subsection{Definition: relaxed $\bdn$-process}\label{sec:relaxed}
In this subsection we formally define our relaxed $\bdn$-process, which attempts to randomly generate a configuration-graph with degree sequence~$\bdn=\bigpar{d_1^{(n)},\dots, d_n^{(n)}}$.  
At any point during the process, we say that a vertex~$V_i$ is \emph{unsaturated} if there is at least one point~$w \in V_i$ that is currently not matched. 
%
\vspace{-1mm}
\begin{algorithm}[H]
    \caption{Sampling random partial configuration-graph with relaxed $\bdn$-process}
    \label{euclid}
    \begin{algorithmic}[1] 
            \State $G:=\bigpar{P(G), E(G)}$ with~$E(G):=\emptyset$ 
            \While{there are at least 2 unsaturated vertices in $G$}
                \State Pick a uniformly random pair~$V_i, V_j$ of distinct unsaturated vertices %
                \State Pick uniformly random unmatched points~$v_i^{p}\in V_i$ and~$v_j^{q}\in V_j$
                \State Add the edge~$v_i^{p}v_j^{q}$ to~$E(G)$ 
            \EndWhile\label{relaxedendwhile}
            \State \textbf{return} $G$
    \end{algorithmic}
\end{algorithm}
As with the (standard) $\bdn$-process, it is possible that the relaxed $\bdn$-process will not complete, i.e., not succeed in matching every point, and so the final configuration-graph will not have the desired degree sequence~$\bdn$. 
However, it turns out that this `bad behavior' happens with sufficiently small  probability (see~\refL{completionprob} in~\refApp{sec:transfer}), 
so that conditioning on the relaxed $\bdn$-process completing does not create any technical problems.
We henceforth write~$\Gpodn$ for the final configuration-graph of the relaxed $\bdn$-process 
conditioned on having degree sequence~$\bdn$, i.e., conditioned on completing
(this is well-defined, since we only consider degree sequences~$\bdn$ that are graphic).

\subsection{Main technical result: small edges in relaxed $\bdn$-process}\label{sec:maintechnical} 
Our main technical result is the following theorem, which is a reformulation of our main result \refT{transferred}~\ref{transferred:process} for the relaxed $d_n$-process. 
It intuitively shows that the number of small edges in $\Gpodn$ is typically far from~$\mu$, i.e., the typical number of small edges in~$\Gdn$ (see part~\ref{transferred:uniform} of~\refT{transferred}). 
\begin{theorem}[Main technical result]\label{main}
Suppose that assumptions of \refT{transferred} hold.
Then there is a constant~${\alpha=\alpha(\xi,\Delta)>0}$ such that 
the final graph~$\Gpodn$ of the relaxed $\bdn$-process satisfies 
\begin{equation}\label{eq:transferred:multiprocess}
\bP\bigpar{\bigabs{\Xs\bigpar{\Gpodn}-\mu} \le \alpha \mu} \le e^{-\Theta(n)}.
\end{equation}
\end{theorem}
As the reader can guess, \refT{main} for the relaxed \mbox{$d_n$-process} implies \refT{transferred}~\ref{transferred:process} for the standard \mbox{$d_n$-process}~$\Gpdn$. 
While this kind of transfer from multigraphs to simple graphs is conceptually standard (much in the spirit of the standard argument for the configuration model), it turns out that the transfer argument for the $d_n$-process is more involved than usual; hence we defer the technical details to \refApp{sec:transfer}.
To sum up: in order to establish \refT{transferred}~\ref{transferred:process}, 
it remains to prove~\refT{main} in the following Sections~\ref{setupandcore}--\ref{sec:counting}. 
\begin{remark}\label{rem:lowertail}
While Theorem~\ref{main} is enough to prove our main result, 
by combining~\eqref{eq:transferred:multiprocess} with the lower tail result Theorem~\ref{location} from \refApp{sec:lowertail} (with~$\eps=\alpha$) we obtain 
the stronger bound  ${\bP\bigpar{\Xs\bigpar{\Gpodn} \le (1+\alpha)\mu}} \le e^{-\Theta(n)}$. 
Hence, with high probability, the $\bdn$-process indeed contains significantly more small edges than~$\Gdn$ (as suggested by the heuristics in \refS{sec:mainresult}). 
\end{remark}

\section{Core switching arguments}\label{setupandcore} 
This section is devoted to the switching-based proof of our main technical result~\refT{main}, which concerns the number of small edges in the relaxed $\bdn$-process.
Namely, in Sections~\ref{sec:probabilities}--\ref{sec:cluster} we introduce the setup and main switching definitions,
and in Section~\ref{sec:switchingcounting} then state several key switching and counting results, 
which in turn are used to prove~\refT{main} in~\refS{sec:doublecounting}.

\subsection{Preliminaries: probabilities in relaxed $\bdn$-process}\label{sec:probabilities} 
For our upcoming switching arguments we need to be able to compare the probabilities with which the relaxed $\bdn$-process produces certain configuration-graphs. 
Here our basic approach is to expand these probabilities, by considering all possible edge-sequences that produce the configuration-graphs (as mentioned in \refS{simplestrategy}).

Turning to the details, let~$G$ be any configuration-graph with degree sequence~$\bdn$, and let~$\s$ be any ordering of its edges (to clarify: the endpoints of the edges of a configuration-graph consist of points, see~\refS{sec:configurationgraph}). 
Consider building $G$ by adding its edges one at a time, ordered by~$\s$. 
Defining 
\begin{equation}\label{def:Gamma_i_sigma}
\Gamma_i(\sigma) := \mbox{number of unsaturated vertices remaining after adding the first~$i$ edges},
\end{equation}
the probability that the relaxed $\bdn$-process chooses the edges of $G$ in order $\sigma$ is easily seen to be
\begin{equation}\label{def:Psigma}
\prod_{0 \le i \le m-1} \frac{2}{\Gamma_i(\sigma)(\Gamma_i(\sigma)-1)} \cdot \prod_{1 \le j \le n} \frac{1}{d^{(n)}_j!}  .
\end{equation}
If we condition on the process completing (i.e., that the final configuration-graph has degree sequence~$\bdn$), then the probability that the relaxed $\bdn$-process chooses the edge-sequence~$\sigma$ is thus proportional~to
\begin{equation}\label{def:Zsigma}
Z(\sigma) := \prod_{0 \le i \le m-1} \frac{2}{\Gamma_i(\sigma)(\Gamma_i(\sigma)-1)}.
\end{equation}
Therefore~$\bP\bigpar{\Gpodn=G}$, the probability that the (conditional) relaxed $\bdn$-process produces~$G$, is proportional~to
\begin{equation}\label{def:ZG}
Z(G):=\sum_{\sigma \in \Pi_G}Z(\sigma) ,
\end{equation}
where we sum over the set~$\Pi_G$ of all edge-sequences of~$G$. 
When summing over a set~$S$ of configuration-graphs~$S$, we henceforth also use the convenient shorthand
\begin{equation}\label{def:ZS}
Z(S):=\sum_{G\in S} Z(G) .
\end{equation}

For later reference, we now record two simple bounds on the number~$\Gamma_i(\sigma)$ of unsaturated vertices.
\begin{observation}\label{obs:unsaturated}
We have $2(m-i)/\Delta \le \Gamma_i(\sigma) \le 2(m-i)$ for all~$0 \le i \le m$. 
\end{observation}
\begin{proof}
For any vertex~$V_j$, let~$\deg_{i}(V_j)$ denote the number of matched points in~$V_j$ after adding the first~$i$~edges 
Since~$\Gamma_i(\sigma)$ counts the number of~$j \in [n]$ with $\deg_{i}(V_j) \le {d_j^{(n)}-1}$, 
using~$0 \le \deg_{i}(V_j) \le d_j^{(n)} \le \Delta$ we infer~that
\[
\Gamma_i(\sigma) 
\le \sum_{j \in [n]} \Bigsqpar{d_j^{(n)} -\deg_{i}(V_j)} \le \Delta \cdot \Gamma_i(\sigma) .
\]
Noting that~$\sum_{j \in [n]} d_j^{(n)}=2m$ and~$\sum_{j \in [n]} \deg_{i}(V_j)=2i$ then establishes the claimed bounds. 
\end{proof}

\subsection{Definition: clusters of configuration-graphs}\label{sec:cluster}
In this subsection we introduce sets of configuration-graphs called \emph{clusters}, 
whose probabilities we shall compare in our upcoming switching arguments (as mentioned in~\refS{generalstrategy}).
To motivate their definition, in~\refS{sec:clustermotivation} we develop some intuition about which configuration-graphs are more likely to be produced by the relaxed $\bdn$-process.
The reader mainly interested in the formal definition of clusters may wish to directly skip to~\refS{sec:clusterdefinition}

\subsubsection{Intuition about configuration-graphs}\label{sec:clustermotivation}
In order to discuss our intuition about which configuration-graphs the relaxed $\bdn$-process favors, 
we first recall the switching setup from \refS{ssps}.  
We have a configuration-graph~$G^+$ with edges~$ab, xy$  that satisfy ${\deg(a),\deg(b)\le k}$ and ${\deg(x),\deg(y)>k}$ 
(recall that~${a,b,x,y}$ are points, and that the degree of a point is the size of the vertex that contains it). 
We let~${A,B,X,Y}$ denote the vertices containing points~${a,b,x,y}$ respectively, where we henceforth assume~${A\neq B}$ and~${X\neq Y}$. 
Furthermore, we write~$G^-$ for the configuration-graph obtained by replacing the edges~$ab, xy$ with~$ax, by$; see~\refF{fig:switching} for an example with~$k=1$.  

As discussed in \refS{ssps}, ideally the relaxed $\bdn$-process should produce $G^+$ with higher probability than~$G^-$, 
which by \refS{sec:probabilities} means that~$Z(G^+)\geq Z(G^-)$. 
It turns out that this is true for~$k=1$, but \emph{not} always true for~$k \ge 2$.
To gain some intuition why, we wish to understand what configuration-graphs~$G$ 
have high values of~$Z(G)=\sum_{\sigma \in \Pi_G}Z(\sigma)$. 
By carefully inspecting~\eqref{def:Zsigma}, it is not hard to see that~$Z(\s)$ tends to be big when the values~$\G_i(\s)$ tend to be small, i.e., when vertices tend to become saturated~early. 
%
Furthermore, 
the impact on~$Z(\s)$ of a vertex  becoming saturated early is amplified if other vertices are also saturated early. 

With these `early saturation' observations in mind, we now consider the case~$k=1$, i.e., where $\deg(A)=\deg(B)=1$ holds.
If the vertices~$A,B$ form the edge~$ab$, then they will both become saturated at the same step, namely the step where that edge appears in the edge-sequence~$\s$.  So in edge-sequences where~$A$ is saturated early, we automatically get the amplification of a second vertex becoming saturated at the same step.  Of course, in edge-sequences where~$A$ is saturated late then~$B$ is also saturated late.  But it turns out that the multiplicative nature of the formula for~$Z$ results in the good edge-sequences outweighing the bad ones. 
After summing over all edge-sequences, 
it thus becomes plausible that~${Z(G^+)> Z(G^-)}$. 

For $k\geq2$, the reasoning is not as clean-cut since the edge~$ab$ does not ensure that~$A,B$ both become saturated at the same step.  Nevertheless, it increases the probability that, when~$A$ is saturated then~$B$ will be saturated soon (and vice versa). Furthermore, the described increase in probability is intuitively higher for the edge~$ab$ than for the edge~$ax$, because the degree of~$B$ is lower than~$X$ (so fewer other points need to be selected until saturation). 
Hence joining~$ab$ seemingly brings a similar, albeit smaller, amplification benefit as doing the same in the case~$k=1$.
%
%
%
However, the behavior of the $\bdn$-process turns out to be more complicated:  
other neighbours of $A,B$ can also significantly impact~$Z$, and in certain situations this `neighborhood effect' can even be stronger than the benefit 
of~$A$ being adjacent to~$B$. 
For example, it could be better for~$A$ to be adjacent to a vertex~$X$ with slightly higher degree than~$B$, 
provided that  the neighbors of~$X$ have significantly lower degrees than the neighbors of~$B$ 
(since this intuitively gives the amplified benefit of several neighbors of~$X$ becoming saturated at around the same step that $A$ and $X$ are).   
\refF{fig:counter} in~\refApp{counterexample} contains an example 
where this happens: we have~$\deg(X)=\deg(B)+1$, but~$X$ is adjacent to more low degree vertices than~$B$, and we indeed have~$Z(G^+)< Z(G^-)$. 
Note that the described problematic `neighborhood effect' cannot occur for~$k=1$, since when~$A,B$ are adjacent, then they have no other neighbors 
(which partially explains why the case~$k=1$ is more~well-behaved; see also \refR{rem:k1:simpler} in~\refS{sec:switching}). 

To overcome the discussed obstacle, we have developed a more complicated form of switching. 
Instead of measuring the effect on~$Z$ of switching edges of a single configuration-graph, 
we will measure the average effect of switching on a set of configuration-graphs, which we call a {\em cluster}. 
These clusters are specifically chosen so that, 
when averaging over a cluster,
the neighbors of~$X$ will not be any better than the neighbors of~$B$, 
and so the described problematic `neighborhood effect' 
will not occur. But the amplifying effect from the edge~$ab$ still remains, and so~$Z$  really will be~larger on~average in~$G^+$ than in~$G^-$ (see~\refL{switch} in~\refS{sec:switchingcounting}).

\subsubsection{Formal definition of clusters}\label{sec:clusterdefinition}
With an eye on introducing clusters, for any configuration-graph~$G$ with a set of points~$S$, we write 
\begin{equation}\label{def:NG}
N_G(S):=\bigcpar{v\text{ point in }G \: : \: v\not\in S \text{ and there is } w\in S \text{ such that } vw \in E(G)}
\end{equation}
for the set of neighbors of~$S$.  Often~$S$ will be specified as a vertex, or a collection of vertices; this means~$S$ will be the set of points of these vertices.  So $N_G(S)$ will not be the set of vertices adjacent to~$S$, but the specific points within those vertices that are matched in~$G$ to the points of~$S$.
We are now ready to formally define clusters, which are illustrated in~\refF{clusterexample}. 
Recall that~${A,B,X,Y}$ are distinct vertices with~${\deg(A),\deg(B)\leq k}$ and~${\deg(X),\deg(Y)>k}$, 
and that~${a,b,x,y}$ are points of~${A,B,X,Y}$, respectively.
{\pagebreak[3]\begin{definition}[Clusters] \ \vspace{-0.5em}
\begin{enumerate}[(a)]
\itemindent 0.125em \itemsep 0.125em \parskip 0.125em 
\item $G_{ab,xy}$ is the set of configuration-graphs with degree sequence~$\bdn$ that contain the two edges~$ab, xy$. 
Similarly, $G_{ax,by}$ is the set of configuration-graphs with degree sequence~$\bdn$ that contain the two edges~$ax, by$. 
\item We define an equivalence relation $\sim$ on $G_{ab,xy}$ and $G_{ax,by}$,
 where~${G_1 \sim G_2}$ if~${N_{G_1}(A\cup B\cup X\cup Y)}={N_{G_2}(A\cup B\cup X\cup Y)}$ holds
 and the set of edges not incident to ${A, B, X, Y}$ are identical. 
\item The equivalence classes of $\sim$ on  $G_{ab,xy}$ are called {\em upper~clusters}. The equivalence classes of $\sim$ on $G_{ax,by}$ are called {\em lower~clusters}.
\end{enumerate}
\end{definition}}
%
%
\begin{figure}%
\centering%
\begin{tikzpicture}
\begin{pgfonlayer}{nodelayer}
\node [style=none] (99) at (1, 1.3) {};
\node [style=none] (100) at (2, 1.3) {};
\node [style=none] (101) at (0.78, 1.3) {};
\node [style=none] (102) at (0.4, 0.5) {};
\node [style=black] (103) at (0.495, 0.7) {};
\node [style=black] (104) at (1.334, 1.3) {};
\node [style=black] (105) at (0.59, 0.9) {};
\node [style=black] (106) at (0.685, 1.1) {};
\node [style=black] (107) at (1.666, 1.3) {};
\node [style=none] (108) at (2.22, 1.3) {};
\node [style=none] (109) at (2.6, 0.5) {};
\node [style=black] (110) at (2.41, 0.9) {};
\node [style=none] (111) at (1, -0.5) {};
\node [style=none] (112) at (2, -0.5) {};
\node [style=none] (113) at (0.78, -0.5) {};
\node [style=none] (114) at (0.4, 0.3) {};
\node [style=black] (115) at (0.495, 0.1) {};
\node [style=black] (116) at (1.334, -0.5) {};
\node [style=black] (117) at (0.59, -0.1) {};
\node [style=black] (118) at (0.685, -0.3) {};
\node [style=black] (119) at (1.666, -0.5) {};
\node [style=none] (120) at (2.22, -0.5) {};
\node [style=none] (121) at (2.6, 0.3) {};
\node [style=black] (122) at (2.41, -0.1) {};
\node [style=none] (131) at (4, 1.3) {};
\node [style=none] (132) at (5, 1.3) {};
\node [style=none] (133) at (3.78, 1.3) {};
\node [style=none] (134) at (3.4, 0.5) {};
\node [style=black] (135) at (3.495, 0.7) {};
\node [style=black] (136) at (4.334, 1.3) {};
\node [style=black] (137) at (3.59, 0.9) {};
\node [style=black] (138) at (3.685, 1.1) {};
\node [style=black] (139) at (4.666, 1.3) {};
\node [style=none] (140) at (5.22, 1.3) {};
\node [style=none] (141) at (5.6, 0.5) {};
\node [style=black] (142) at (5.41, 0.9) {};
\node [style=none] (143) at (4, -0.5) {};
\node [style=none] (144) at (5, -0.5) {};
\node [style=none] (145) at (3.78, -0.5) {};
\node [style=none] (146) at (3.4, 0.3) {};
\node [style=black] (147) at (3.495, 0.1) {};
\node [style=black] (148) at (4.334, -0.5) {};
\node [style=black] (149) at (3.59, -0.1) {};
\node [style=black] (150) at (3.685, -0.3) {};
\node [style=black] (151) at (4.666, -0.5) {};
\node [style=none] (152) at (5.22, -0.5) {};
\node [style=none] (153) at (5.6, 0.3) {};
\node [style=black] (154) at (5.41, -0.1) {};
\node [style=none] (163) at (7, 1.3) {};
\node [style=none] (164) at (8, 1.3) {};
\node [style=none] (165) at (6.78, 1.3) {};
\node [style=none] (166) at (6.4, 0.5) {};
\node [style=black] (167) at (6.495, 0.7) {};
\node [style=black] (168) at (7.334, 1.3) {};
\node [style=black] (169) at (6.59, 0.9) {};
\node [style=black] (170) at (6.685, 1.1) {};
\node [style=black] (171) at (7.666, 1.3) {};
\node [style=none] (172) at (8.22, 1.3) {};
\node [style=none] (173) at (8.6, 0.5) {};
\node [style=black] (174) at (8.41, 0.9) {};
\node [style=none] (175) at (7, -0.5) {};
\node [style=none] (176) at (8, -0.5) {};
\node [style=none] (177) at (6.78, -0.5) {};
\node [style=none] (178) at (6.4, 0.3) {};
\node [style=black] (179) at (6.495, 0.1) {};
\node [style=black] (180) at (7.334, -0.5) {};
\node [style=black] (181) at (6.59, -0.1) {};
\node [style=black] (182) at (6.685, -0.3) {};
\node [style=black] (183) at (7.666, -0.5) {};
\node [style=none] (184) at (8.22, -0.5) {};
\node [style=none] (185) at (8.6, 0.3) {};
\node [style=black] (186) at (8.41, -0.1) {};
\node [style=none] (195) at (1, -1.7) {};
\node [style=none] (196) at (2, -1.7) {};
\node [style=none] (197) at (0.78, -1.7) {};
\node [style=none] (198) at (0.4, -2.5) {};
\node [style=black] (199) at (0.495, -2.3) {};
\node [style=black] (200) at (1.334, -1.7) {};
\node [style=black] (201) at (0.59, -2.1) {};
\node [style=black] (202) at (0.685, -1.9) {};
\node [style=black] (203) at (1.666, -1.7) {};
\node [style=none] (204) at (2.22, -1.7) {};
\node [style=none] (205) at (2.6, -2.5) {};
\node [style=black] (206) at (2.41, -2.1) {};
\node [style=none] (207) at (1, -3.5) {};
\node [style=none] (208) at (2, -3.5) {};
\node [style=none] (209) at (0.78, -3.5) {};
\node [style=none] (210) at (0.4, -2.7) {};
\node [style=black] (211) at (0.495, -2.9) {};
\node [style=black] (212) at (1.334, -3.5) {};
\node [style=black] (213) at (0.59, -3.1) {};
\node [style=black] (214) at (0.685, -3.3) {};
\node [style=black] (215) at (1.666, -3.5) {};
\node [style=none] (216) at (2.22, -3.5) {};
\node [style=none] (217) at (2.6, -2.7) {};
\node [style=black] (218) at (2.41, -3.1) {};
\node [style=none] (227) at (7, -1.7) {};
\node [style=none] (228) at (8, -1.7) {};
\node [style=none] (229) at (6.78, -1.7) {};
\node [style=none] (230) at (6.4, -2.5) {};
\node [style=black] (231) at (6.495, -2.3) {};
\node [style=black] (232) at (7.334, -1.7) {};
\node [style=black] (233) at (6.59, -2.1) {};
\node [style=black] (234) at (6.685, -1.9) {};
\node [style=black] (235) at (7.666, -1.7) {};
\node [style=none] (236) at (8.22, -1.7) {};
\node [style=none] (237) at (8.6, -2.5) {};
\node [style=black] (238) at (8.41, -2.1) {};
\node [style=none] (239) at (7, -3.5) {};
\node [style=none] (240) at (8, -3.5) {};
\node [style=none] (241) at (6.78, -3.5) {};
\node [style=none] (242) at (6.4, -2.7) {};
\node [style=black] (243) at (6.495, -2.9) {};
\node [style=black] (244) at (7.334, -3.5) {};
\node [style=black] (245) at (6.59, -3.1) {};
\node [style=black] (246) at (6.685, -3.3) {};
\node [style=black] (247) at (7.666, -3.5) {};
\node [style=none] (248) at (8.22, -3.5) {};
\node [style=none] (249) at (8.6, -2.7) {};
\node [style=black] (250) at (8.41, -3.1) {};
\node [style=none] (259) at (4, -1.7) {};
\node [style=none] (260) at (5, -1.7) {};
\node [style=none] (261) at (3.78, -1.7) {};
\node [style=none] (262) at (3.4, -2.5) {};
\node [style=black] (263) at (3.495, -2.3) {};
\node [style=black] (264) at (4.334, -1.7) {};
\node [style=black] (265) at (3.59, -2.1) {};
\node [style=black] (266) at (3.685, -1.9) {};
\node [style=black] (267) at (4.666, -1.7) {};
\node [style=none] (268) at (5.22, -1.7) {};
\node [style=none] (269) at (5.6, -2.5) {};
\node [style=black] (270) at (5.41, -2.1) {};
\node [style=none] (271) at (4, -3.5) {};
\node [style=none] (272) at (5, -3.5) {};
\node [style=none] (273) at (3.78, -3.5) {};
\node [style=none] (274) at (3.4, -2.7) {};
\node [style=black] (275) at (3.495, -2.9) {};
\node [style=black] (276) at (4.334, -3.5) {};
\node [style=black] (277) at (3.59, -3.1) {};
\node [style=black] (278) at (3.685, -3.3) {};
\node [style=black] (279) at (4.666, -3.5) {};
\node [style=none] (280) at (5.22, -3.5) {};
\node [style=none] (281) at (5.6, -2.7) {};
\node [style=black] (282) at (5.41, -3.1) {};
\node [style=none] (291) at (1, -4.7) {};
\node [style=none] (292) at (2, -4.7) {};
\node [style=none] (293) at (0.78, -4.7) {};
\node [style=none] (294) at (0.4, -5.5) {};
\node [style=black] (295) at (0.495, -5.3) {};
\node [style=black] (296) at (1.334, -4.7) {};
\node [style=black] (297) at (0.59, -5.1) {};
\node [style=black] (298) at (0.685, -4.9) {};
\node [style=black] (299) at (1.666, -4.7) {};
\node [style=none] (300) at (2.22, -4.7) {};
\node [style=none] (301) at (2.6, -5.5) {};
\node [style=black] (302) at (2.41, -5.1) {};
\node [style=none] (303) at (1, -6.5) {};
\node [style=none] (304) at (2, -6.5) {};
\node [style=none] (305) at (0.78, -6.5) {};
\node [style=none] (306) at (0.4, -5.7) {};
\node [style=black] (307) at (0.495, -5.9) {};
\node [style=black] (308) at (1.334, -6.5) {};
\node [style=black] (309) at (0.59, -6.1) {};
\node [style=black] (310) at (0.685, -6.3) {};
\node [style=black] (311) at (1.666, -6.5) {};
\node [style=none] (312) at (2.22, -6.5) {};
\node [style=none] (313) at (2.6, -5.7) {};
\node [style=black] (314) at (2.41, -6.1) {};
\node [style=none] (323) at (4, -4.7) {};
\node [style=none] (324) at (5, -4.7) {};
\node [style=none] (325) at (3.78, -4.7) {};
\node [style=none] (326) at (3.4, -5.5) {};
\node [style=black] (327) at (3.495, -5.3) {};
\node [style=black] (328) at (4.334, -4.7) {};
\node [style=black] (329) at (3.59, -5.1) {};
\node [style=black] (330) at (3.685, -4.9) {};
\node [style=black] (331) at (4.666, -4.7) {};
\node [style=none] (332) at (5.22, -4.7) {};
\node [style=none] (333) at (5.6, -5.5) {};
\node [style=black] (334) at (5.41, -5.1) {};
\node [style=none] (335) at (4, -6.5) {};
\node [style=none] (336) at (5, -6.5) {};
\node [style=none] (337) at (3.78, -6.5) {};
\node [style=none] (338) at (3.4, -5.7) {};
\node [style=black] (339) at (3.495, -5.9) {};
\node [style=black] (340) at (4.334, -6.5) {};
\node [style=black] (341) at (3.59, -6.1) {};
\node [style=black] (342) at (3.685, -6.3) {};
\node [style=black] (343) at (4.666, -6.5) {};
\node [style=none] (344) at (5.22, -6.5) {};
\node [style=none] (345) at (5.6, -5.7) {};
\node [style=black] (346) at (5.41, -6.1) {};
\node [style=none] (355) at (7, -4.7) {};
\node [style=none] (356) at (8, -4.7) {};
\node [style=none] (357) at (6.78, -4.7) {};
\node [style=none] (358) at (6.4, -5.5) {};
\node [style=black] (359) at (6.495, -5.3) {};
\node [style=black] (360) at (7.334, -4.7) {};
\node [style=black] (361) at (6.59, -5.1) {};
\node [style=black] (362) at (6.685, -4.9) {};
\node [style=black] (363) at (7.666, -4.7) {};
\node [style=none] (364) at (8.22, -4.7) {};
\node [style=none] (365) at (8.6, -5.5) {};
\node [style=black] (366) at (8.41, -5.1) {};
\node [style=none] (367) at (7, -6.5) {};
\node [style=none] (368) at (8, -6.5) {};
\node [style=none] (369) at (6.78, -6.5) {};
\node [style=none] (370) at (6.4, -5.7) {};
\node [style=black] (371) at (6.495, -5.9) {};
\node [style=black] (372) at (7.334, -6.5) {};
\node [style=black] (373) at (6.59, -6.1) {};
\node [style=black] (374) at (6.685, -6.3) {};
\node [style=black] (375) at (7.666, -6.5) {};
\node [style=none] (376) at (8.22, -6.5) {};
\node [style=none] (377) at (8.6, -5.7) {};
\node [style=black] (378) at (8.41, -6.1) {};
\node [style=none] (385) at (3, 2) {};
\node [style=none] (386) at (6, 2) {};
\node [style=none] (387) at (9, 2) {};
\node [style=none] (399) at (9, -4) {};
\node [style=none] (434) at (0, 2) {};
\node [style=none] (435) at (0, -7) {};
\node [style=none] (436) at (9, -7) {};
\node [style=none] (437) at (0, -4) {};
\node [style=none] (438) at (0, -1) {};
\node [style=none] (439) at (9, -1) {};
\node [style=none] (440) at (3, -7) {};
\node [style=none] (441) at (6, -7) {};
\node [style=none] (442) at (-6, 2) {};
\node [style=none] (443) at (-6, -7) {};
\node [style=none] (444) at (-4, -0.7) {};
\node [style=none] (445) at (-2, -0.7) {};
\node [style=none] (446) at (-4.44, -0.7) {};
\node [style=none] (447) at (-5.2, -2.3) {};
\node [style=black] (448) at (-5.01, -1.9) {};
\node [style=black] (449) at (-3.332, -0.7) {};
\node [style=black] (450) at (-4.82, -1.5) {};
\node [style=black] (451) at (-4.63, -1.1) {};
\node [style=black] (452) at (-2.67, -0.7) {};
\node [style=none] (453) at (-1.56, -0.7) {};
\node [style=none] (454) at (-0.8, -2.3) {};
\node [style=black] (455) at (-1.18, -1.5) {};
\node [style=none] (456) at (-4, -4.3) {};
\node [style=none] (457) at (-2, -4.3) {};
\node [style=none] (458) at (-4.44, -4.3) {};
\node [style=none] (459) at (-5.2, -2.7) {};
\node [style=black] (460) at (-5.01, -3.1) {};
\node [style=black] (461) at (-3.332, -4.3) {};
\node [style=black] (462) at (-4.82, -3.5) {};
\node [style=black] (463) at (-4.63, -3.9) {};
\node [style=black] (464) at (-2.67, -4.3) {};
\node [style=none] (465) at (-1.56, -4.3) {};
\node [style=none] (466) at (-0.8, -2.7) {};
\node [style=black] (467) at (-1.18, -3.5) {};
\node [style=none] (468) at (-5.5, 0.2) {$G$};
\node [style=none] (469) at (-2.67, -0.5) {$a$};
\node [style=none] (470) at (-3, -0.05) {$A$};
\node [style=none] (471) at (-1.18, -1.8) {$b$};
\node [style=none] (472) at (-0.5, -1.05) {$B$};
\node [style=none] (473) at (-5.2, -1.9) {$x$};
\node [style=none] (474) at (-5.5, -1.05) {$X$};
\node [style=none] (475) at (-5.2, -3.1) {$y$};
\node [style=none] (476) at (-5.5, -3.8) {$Y$};
\node [style=none] (477) at (-3.325, -4.5) {$v$};
\end{pgfonlayer}
\begin{pgfonlayer}{edgelayer}
\draw [bend left=90, looseness=0.75] (99.center) to (100.center);
\draw [bend right=90, looseness=0.75] (99.center) to (100.center);
\draw [bend left=90, looseness=0.75] (101.center) to (102.center);
\draw [bend right=90, looseness=0.75] (101.center) to (102.center);
\draw [bend left=90, looseness=0.75] (108.center) to (109.center);
\draw [bend right=90, looseness=0.75] (108.center) to (109.center);
\draw [bend right=90, looseness=0.75] (111.center) to (112.center);
\draw [bend left=90, looseness=0.75] (111.center) to (112.center);
\draw [bend right=90, looseness=0.75] (113.center) to (114.center);
\draw [bend left=90, looseness=0.75] (113.center) to (114.center);
\draw [bend right=90, looseness=0.75] (120.center) to (121.center);
\draw [bend left=90, looseness=0.75] (120.center) to (121.center);
\draw (107) to (110);
\draw (119) to (122);
\draw (103) to (115);
\draw [bend left=90, looseness=0.75] (131.center) to (132.center);
\draw [bend right=90, looseness=0.75] (131.center) to (132.center);
\draw [bend left=90, looseness=0.75] (133.center) to (134.center);
\draw [bend right=90, looseness=0.75] (133.center) to (134.center);
\draw [bend left=90, looseness=0.75] (140.center) to (141.center);
\draw [bend right=90, looseness=0.75] (140.center) to (141.center);
\draw [bend right=90, looseness=0.75] (143.center) to (144.center);
\draw [bend left=90, looseness=0.75] (143.center) to (144.center);
\draw [bend right=90, looseness=0.75] (145.center) to (146.center);
\draw [bend left=90, looseness=0.75] (145.center) to (146.center);
\draw [bend right=90, looseness=0.75] (152.center) to (153.center);
\draw [bend left=90, looseness=0.75] (152.center) to (153.center);
\draw (139) to (142);
\draw (151) to (154);
\draw (135) to (147);
\draw [bend left=90, looseness=0.75] (163.center) to (164.center);
\draw [bend right=90, looseness=0.75] (163.center) to (164.center);
\draw [bend left=90, looseness=0.75] (165.center) to (166.center);
\draw [bend right=90, looseness=0.75] (165.center) to (166.center);
\draw [bend left=90, looseness=0.75] (172.center) to (173.center);
\draw [bend right=90, looseness=0.75] (172.center) to (173.center);
\draw [bend right=90, looseness=0.75] (175.center) to (176.center);
\draw [bend left=90, looseness=0.75] (175.center) to (176.center);
\draw [bend right=90, looseness=0.75] (177.center) to (178.center);
\draw [bend left=90, looseness=0.75] (177.center) to (178.center);
\draw [bend right=90, looseness=0.75] (184.center) to (185.center);
\draw [bend left=90, looseness=0.75] (184.center) to (185.center);
\draw (171) to (174);
\draw (183) to (186);
\draw (167) to (179);
\draw [bend left=90, looseness=0.75] (195.center) to (196.center);
\draw [bend right=90, looseness=0.75] (195.center) to (196.center);
\draw [bend left=90, looseness=0.75] (197.center) to (198.center);
\draw [bend right=90, looseness=0.75] (197.center) to (198.center);
\draw [bend left=90, looseness=0.75] (204.center) to (205.center);
\draw [bend right=90, looseness=0.75] (204.center) to (205.center);
\draw [bend right=90, looseness=0.75] (207.center) to (208.center);
\draw [bend left=90, looseness=0.75] (207.center) to (208.center);
\draw [bend right=90, looseness=0.75] (209.center) to (210.center);
\draw [bend left=90, looseness=0.75] (209.center) to (210.center);
\draw [bend right=90, looseness=0.75] (216.center) to (217.center);
\draw [bend left=90, looseness=0.75] (216.center) to (217.center);
\draw (203) to (206);
\draw (215) to (218);
\draw (199) to (211);
\draw [bend left=90, looseness=0.75] (227.center) to (228.center);
\draw [bend right=90, looseness=0.75] (227.center) to (228.center);
\draw [bend left=90, looseness=0.75] (229.center) to (230.center);
\draw [bend right=90, looseness=0.75] (229.center) to (230.center);
\draw [bend left=90, looseness=0.75] (236.center) to (237.center);
\draw [bend right=90, looseness=0.75] (236.center) to (237.center);
\draw [bend right=90, looseness=0.75] (239.center) to (240.center);
\draw [bend left=90, looseness=0.75] (239.center) to (240.center);
\draw [bend right=90, looseness=0.75] (241.center) to (242.center);
\draw [bend left=90, looseness=0.75] (241.center) to (242.center);
\draw [bend right=90, looseness=0.75] (248.center) to (249.center);
\draw [bend left=90, looseness=0.75] (248.center) to (249.center);
\draw (235) to (238);
\draw (247) to (250);
\draw (231) to (243);
\draw [bend left=90, looseness=0.75] (259.center) to (260.center);
\draw [bend right=90, looseness=0.75] (259.center) to (260.center);
\draw [bend left=90, looseness=0.75] (261.center) to (262.center);
\draw [bend right=90, looseness=0.75] (261.center) to (262.center);
\draw [bend left=90, looseness=0.75] (268.center) to (269.center);
\draw [bend right=90, looseness=0.75] (268.center) to (269.center);
\draw [bend right=90, looseness=0.75] (271.center) to (272.center);
\draw [bend left=90, looseness=0.75] (271.center) to (272.center);
\draw [bend right=90, looseness=0.75] (273.center) to (274.center);
\draw [bend left=90, looseness=0.75] (273.center) to (274.center);
\draw [bend right=90, looseness=0.75] (280.center) to (281.center);
\draw [bend left=90, looseness=0.75] (280.center) to (281.center);
\draw (267) to (270);
\draw (279) to (282);
\draw (263) to (275);
\draw [bend left=90, looseness=0.75] (291.center) to (292.center);
\draw [bend right=90, looseness=0.75] (291.center) to (292.center);
\draw [bend left=90, looseness=0.75] (293.center) to (294.center);
\draw [bend right=90, looseness=0.75] (293.center) to (294.center);
\draw [bend left=90, looseness=0.75] (300.center) to (301.center);
\draw [bend right=90, looseness=0.75] (300.center) to (301.center);
\draw [bend right=90, looseness=0.75] (303.center) to (304.center);
\draw [bend left=90, looseness=0.75] (303.center) to (304.center);
\draw [bend right=90, looseness=0.75] (305.center) to (306.center);
\draw [bend left=90, looseness=0.75] (305.center) to (306.center);
\draw [bend right=90, looseness=0.75] (312.center) to (313.center);
\draw [bend left=90, looseness=0.75] (312.center) to (313.center);
\draw (299) to (302);
\draw (311) to (314);
\draw (295) to (307);
\draw [bend left=90, looseness=0.75] (323.center) to (324.center);
\draw [bend right=90, looseness=0.75] (323.center) to (324.center);
\draw [bend left=90, looseness=0.75] (325.center) to (326.center);
\draw [bend right=90, looseness=0.75] (325.center) to (326.center);
\draw [bend left=90, looseness=0.75] (332.center) to (333.center);
\draw [bend right=90, looseness=0.75] (332.center) to (333.center);
\draw [bend right=90, looseness=0.75] (335.center) to (336.center);
\draw [bend left=90, looseness=0.75] (335.center) to (336.center);
\draw [bend right=90, looseness=0.75] (337.center) to (338.center);
\draw [bend left=90, looseness=0.75] (337.center) to (338.center);
\draw [bend right=90, looseness=0.75] (344.center) to (345.center);
\draw [bend left=90, looseness=0.75] (344.center) to (345.center);
\draw (331) to (334);
\draw (343) to (346);
\draw (327) to (339);
\draw [bend left=90, looseness=0.75] (355.center) to (356.center);
\draw [bend right=90, looseness=0.75] (355.center) to (356.center);
\draw [bend left=90, looseness=0.75] (357.center) to (358.center);
\draw [bend right=90, looseness=0.75] (357.center) to (358.center);
\draw [bend left=90, looseness=0.75] (364.center) to (365.center);
\draw [bend right=90, looseness=0.75] (364.center) to (365.center);
\draw [bend right=90, looseness=0.75] (367.center) to (368.center);
\draw [bend left=90, looseness=0.75] (367.center) to (368.center);
\draw [bend right=90, looseness=0.75] (369.center) to (370.center);
\draw [bend left=90, looseness=0.75] (369.center) to (370.center);
\draw [bend right=90, looseness=0.75] (376.center) to (377.center);
\draw [bend left=90, looseness=0.75] (376.center) to (377.center);
\draw (363) to (366);
\draw (375) to (378);
\draw (359) to (371);
\draw (104) to (116);
\draw (118) to (105);
\draw (117) to (106);
\draw (148) to (150);
\draw (149) to (137);
\draw (138) to (136);
\draw (180) to (182);
\draw (181) to (170);
\draw (169) to (168);
\draw (212) to (213);
\draw (202) to (200);
\draw (308) to (297);
\draw (298) to (310);
\draw (296) to (309);
\draw (340) to (330);
\draw (329) to (341);
\draw (342) to (328);
\draw (372) to (362);
\draw (361) to (374);
\draw (373) to (360);
\draw (214) to (201);
\draw (434.center) to (435.center);
\draw (435.center) to (436.center);
\draw (436.center) to (387.center);
\draw (439.center) to (438.center);
\draw (385.center) to (440.center);
\draw (441.center) to (386.center);
\draw [bend left=90, looseness=0.75] (444.center) to (445.center);
\draw [bend right=90, looseness=0.75] (444.center) to (445.center);
\draw [bend left=90, looseness=0.75] (446.center) to (447.center);
\draw [bend right=90, looseness=0.75] (446.center) to (447.center);
\draw [bend left=90, looseness=0.75] (453.center) to (454.center);
\draw [bend right=90, looseness=0.75] (453.center) to (454.center);
\draw [bend right=90, looseness=0.75] (456.center) to (457.center);
\draw [bend left=90, looseness=0.75] (456.center) to (457.center);
\draw [bend right=90, looseness=0.75] (458.center) to (459.center);
\draw [bend left=90, looseness=0.75] (458.center) to (459.center);
\draw [bend right=90, looseness=0.75] (465.center) to (466.center);
\draw [bend left=90, looseness=0.75] (465.center) to (466.center);
\draw (452) to (455);
\draw (464) to (467);
\draw (448) to (460);
\draw (461) to (449);
\draw (451) to (463);
\draw (462) to (450);
\draw (266) to (278);
\draw (264) to (265);
\draw (276) to (277);
\draw (246) to (232);
\draw (244) to (233);
\draw (234) to (245);
\draw (437.center) to (399.center);
\draw (434.center) to (387.center);
\end{pgfonlayer}
\end{tikzpicture}%
\caption{Example of an upper cluster ${C^+=C^+(G,ab,xy)}$ containing ten configuration-graphs. 
In this case~${N_G(A\cup B\cup X\cup Y)=\{v\}}$, 
and three edges are fixed: $ab, xy$ and the only edge not incident to~${A,B,X,Y}$. 
If we start with only these three edges, one can construct any configuration-graph in the cluster~$C^+$ by 
(i)~matching~$v$ to an unmatched point in~${A, B, X, Y}$, and then 
(ii)~matching the remaining four unmatched points in~${A, B, X, Y}$. 
One can easily check that there are five ways to do step~(i), 
and in each case there are exactly two ways to complete step~(ii), 
leading to all ten depicted configuration-graphs.
\label{clusterexample}}%
\end{figure}
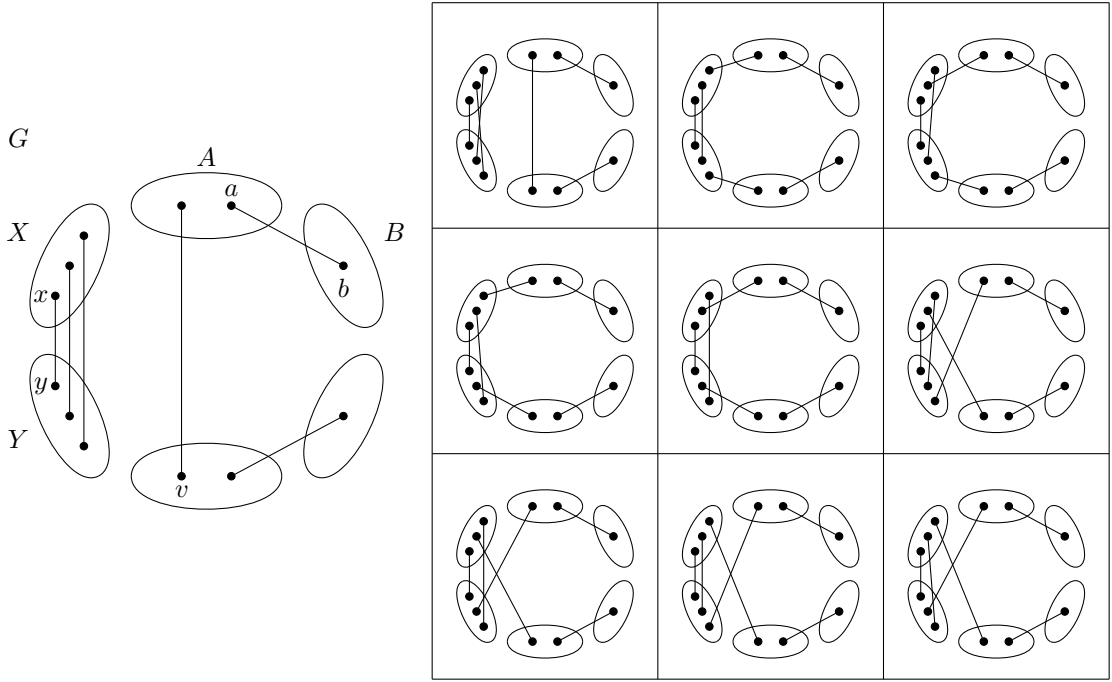
For a configuration-graph~$G\in  G_{ab,xy}$, we denote by $C^+(G,ab,xy)$ the upper cluster containing $G$ that corresponds to an equivalence class in $G_{ab,xy}$. Similarly for a configuration-graph~$G\in  G_{ax,by}$, we denote by $C^-(G,ax,by)$ the lower cluster containing $G$ that corresponds to an equivalence class in $G_{ax,by}$. Thus, if~$H$~is another configuration-graph in  $C^+(G,ab,xy)$ then  $C^+(G,ab,xy)= C^+(H,ab,xy)$.

We can think of an upper cluster as follows.  Start with any configuration-graph~$G$ with edges $ab, xy$. Remove all edges of $G$ with at least one endpoint in~${A,B,X,Y}$, except for~$ab, xy$.  Then replace them with any other edges, each having at least one endpoint in~${A,B,X,Y}$,  subject to the final degree-sequence being~$\bdn$.  So if a vertex $Z$ (distinct from $A,B,X,Y$) had~$\ell$ neighboring points in~${A,B,X,Y}$, then~$Z$ will still have~$\ell$ neighboring points in those vertices.  The set of configuration-graphs that can be obtained is the upper cluster; cf.~\refF{clusterexample}. We can think of lower clusters in the same manner, replacing~$ab, xy$ with~$ax, by$.
\begin{definition}[Switching-partners] \label{oclusters} 
Given $G^+\in G_{ab,xy}$, let $G^-\in G_{ax,by}$ be the graph obtained by replacing the edges $ab, xy$ in $G^+$ with $ax, by$. 
There is a natural bijection between configuration-graphs in $C^+=C^+(G^+,ab,xy)$ and $C^-=C^-(G^-,ax,by)$. 
Namely, for every configuration-graph in $C^+$, replace the edges $ab, xy$ with $ax, by$. 
The resulting set of configuration-graphs is the lower cluster~$C^-$.  
We call this pair of clusters~$C^+, C^-$ {\em switching-partners}.
\end{definition}
We remark that comparing the values of $Z$ averaged over the configuration-graphs of two switching-partners (see~\refL{switch} below) is the more complicated form of switching that we alluded to in~\refS{sec:clustermotivation}. 
The crux is that in this average, the vertices in ${A, B, X, Y}$ each have the same potential neighbors 
outside of~$\{A, B, X, Y\}$, 
which eliminates the problematic `neighborhood effect' described in \refS{sec:clustermotivation}. 

\subsection{Key lemmas: switching and counting results for clusters}\label{sec:switchingcounting}
In this subsection we state our main switching results for clusters, along with relevant counting results.  

We start with the following switching-type result between upper and lower clusters, whose proof we defer to~\refS{sec:goodswitch}. 
Recall that~$Z(C)=\sum_{G\in C} Z(G) = \sum_{G \in C}\sum_{\sigma \in \Pi_G}Z(\sigma)$ 
is proportional to the probability that the 
relaxed $\bdn$-process produces a graph from the cluster~$C$, 
see~\eqref{def:ZG}--\eqref{def:ZS}. 
So \refL{switch} intuitively states that, up to some local averaging\footnote{Since~$C^+$ and~$C^-$ have the same cardinality, \refL{switch} shows that $\sum_{G\in C^+} Z(G)/|C^+|\ge \sum_{G\in C^-} Z(G)/|C^-|$, which in concrete words means that~$Z(G)$ averaged over~$G \in C^{+}$ is at least at large as~$Z(G)$ averaged over~$G \in C^-$.}, 
the relaxed $\bdn$-process prefers the edges~$ab, xy$ over the edges~$ax, by$, 
i.e., prefers configuration-graphs with more small edges 
(which is consistent with~\refR{rem:lowertail} and the heuristics from~\refS{sec:mainresult}). 
\begin{lemma}[Switching of clusters]\label{switch}
For all switching-partners~$C^+,C^-$, 
we have~${Z(C^+)\ge Z(C^-)}$. 
\end{lemma}
%
%
We would like to strengthen the conclusion of \refL{switch} to ${Z(C^+)\ge (1+\epsilon)Z(C^-)}$ for some $\epsilon>0$, 
but for technical reasons we only prove it for clusters with special properties, 
defined below as \emph{good clusters} (which suffices, since a significant proportion of clusters turns out to be good). 
The main difficulty in proving it for all clusters is that when the vertex~$A$ or~$B$ become saturated during the final few steps of the relaxed $\bdn$-process, 
the number of unsaturated vertices~$\G_i$ is so small that a small (additive) change to it could have a significant effect on~$Z(\s)$;  
cf.~equation~\eqref{def:Zsigma} and Observation~\ref{obs:unsaturated}. 
The definition of good clusters will intuitively ensure that such `late saturation effects' have limited impact on our arguments. 
Recall that~$G_{ab,xy}$ is the set of configuration-graphs with degree sequence~$\bdn$ that contain the edges~$ab, xy$, see~\refS{sec:clusterdefinition}. 
\begin{definition}[Good edge-sequences and clusters]\label{def:good}
An edge-sequence $\s$ of $G\in G_{ab,xy}$ is called \emph{good with respect to $a,b,x,y$} if the following two conditions hold:%
\vspace{-0.25em}\begin{enumerate}
\itemindent 0.125em \itemsep 0.125em \parskip 0.125em
    \item [(i)] no points of~${X, Y}$ are adjacent to any points of ${A, B}$ in $G$, and
    \item[(ii)] the vertices~$A$ and~$B$ are both saturated in the first $(1-\zeta)m$ steps, 
where 
\begin{equation}\label{def:zeta}
\zeta 
 := \frac{\xi^2}{16\Delta^3} .
\end{equation}
\end{enumerate}
An upper cluster $C^+=C^+(G,ab,xy)$ is called \textit{good} if at least a sixteenth of its total $Z$-value comes from good edge-sequences $\s$, i.e., if
\begin{equation}\label{goodclusterdef}
\sum_{G\in C^+}
\sum_{\substack{\sigma\in \Pi_G:\\ (\sigma,ab,xy)\in \bG}} Z(\sigma) \; \ge \; \frac{1}{16}Z(C^+),
\end{equation}
where
\begin{equation}\label{def:bG}
\bG:=\bigcpar{(\sigma, ab, xy) \: : \: \text{$\sigma$ is good with respect to $a,b,x,y$}}.
\end{equation}%
\end{definition}
The next lemma, whose proof we defer to \refS{sec:goodswitch}, 
shows that for good clusters we can indeed strengthen the switching-type estimate of \refL{switch} 
by an extra factor of~$1+\epsilon$.  
\begin{lemma}[Switching of good clusters]\label{goodswitch}
There are~$\epsilon,n_0>0$ 
such that the following holds for all~${n \ge n_0}$: 
if~${C^+,C^-}$ are switching-partners and ${C^+=C^+(G,ab,xy)}$ is good, then we have
\[Z(C^+)\ge (1+\epsilon) Z(C^-).\] 
\end{lemma}

Recall that $\Xs(G)$~is the number of small edges in a configuration-graph~$G$, 
and that $\mu$~is approximately the expected number of small edges in a uniform configuration-graph with degree sequence~$\bdn$. 
Our main technical result \refT{main} bounds the probability that, in the graph chosen by the relaxed $\bdn$-process, 
the number of small edges has a small (but significant) distance from $\mu$.
We thus define 
\begin{equation}\label{def:Nz}
\cN_z :=\bigcpar{G \: : \: \text{configuration-graph with degree sequence }  \bdn \text{ and } |\Xs(G) -\mu| \le z } .
\end{equation}
and henceforth focus on the configuration-graphs $G\in \cN_{\gamma\mu}$ where $\gamma>0$ is a constant.
For such~$G$, different choices for $a,b,x,y$ can yield different upper clusters $C^+(G,ab,xy)$, and similarly for lower clusters. 
Note that while switching-partners form a bijection upon specifying $a,b,x,y$, each configuration-graph can have  different switching-partners for different choices of those four vertices.  One particularly noteworthy example arises when we exchange $x,y$. 
Note that $C^+(G,ab,xy) = C^+(G,ab,yx)$, but $C^-(G,ax, by)$ and $C^-(G,ay,bx)$ are disjoint. 
 Similarly when we exchange $a,b$. 
This could potentially create a problem with our definition of switching-partners, since the switching-partner of $C^+(G^+,ab,xy)$ is defined to be $C^-(G^-,ax,by)$, while the switching-partner of $C^+(G^+,ab,yx)$ is defined to be $C^-(G^-,ay,bx)$.  
To account for this problem we specify the following convention, 
which will be key for the counting result Lemma~\ref{coefficients} below (as well as Observation~\ref{obs:convention} in Section~\ref{sec:counting}).
\begin{convention}\label{notationalconvention}
$C^+(G,ab,xy), C^+(G,ab,yx), C^+(G,ba,xy), C^+(G,ba,yx)$ are considered to be four different clusters, despite the fact that they are identical. Similarly $C^-(G,ax,by)$, $C^-(G,by, ax)$ are considered to be two different clusters and so are $C^-(G,ay,bx)$, $C^-(G,bx, ay)$. On the other hand, $C^-(G,ax,by)$ and $C^-(G,xa,by)$ are considered the same; 
note that this does not lead to any problems in our definition of switching-partners, since we  list the edge as ``$ax$'' rather than ``$xa$'' in our definitions of lower clusters and switching-partners.
\end{convention}

When we discuss the number of clusters containing $G$, we mean the number of choices of $a,b,x,y$ such that $G$ lies in a cluster $C^+(G,ab,xy)$ or $C^-(G,ax,by)$; i.e., the number of choices of $a,b,x,y$ such that~$E(G)$ contains $ab, xy$ or $ax, by$. From the notational convention above, we see that clusters obtained by permuting the order of $a,b$ or $x,y$ are considered to be distinct, even if they contain the same set of configuration-graphs (e.g., $C^+(G,ab,xy)$ and $C^+(G,ab,yx)$ are counted as two different clusters).
With this convention in mind, for any configuration-graph~$G$, we define~$U_G$ to be the number of upper clusters that contain~$G$. 
We similarly define~$L_G$ to be the number of lower clusters that contain~$G$.

The next lemma, whose proof we defer to \refS{sec:coefficients}, shows that for any configuration-graph with approximately $\mu$ small edges, 
the number of upper and lower clusters containing it are roughly~equal 
for which it is intuitively important that switching-partners form a bijection, as discussed above). 
%
\begin{lemma}[Counting clusters]\label{coefficients}
There are $\gamma_0,n_0,D>0$ such that the following holds for all~${0<\gamma<\gamma_0}$ and~${n\ge n_0}$: 
we have $|{L_G}/{U_H}-1| \le D\gamma$ 
and~$L_G,U_H \ge 1$ 
for all configuration-graphs~$G,H\in \cN_{\gamma\mu}$.
\end{lemma}
%

For technical reasons we shall also need the following lemma, whose proof we defer to \refS{sec:goodpermutation}. 
It intuitively says that, for any edge-sequence $\s$ of a configuration-graph~$G$ with approximately~$\mu$ small edges, 
many upper clusters~$C^+=C^+(G,ab,xy)$ containing~$G$ yield a good edge-sequence~$(\sigma,ab,xy) \in \bG$, 
where~$\bG$ is defined as in~\eqref{def:bG}. 
%
\begin{lemma}[Counting upper clusters yielding good edge-sequences]\label{goodpermutation}
There are~$\gamma,n_0>0$ such that the following holds for all~$n \ge n_0$: 
for all configuration-graphs~$G \in \cN_{\gamma\mu}$ and edge-sequences~$\sigma\in \Pi_G$ we have 
\begin{equation}\label{eq:goodpermutation}
\big|\bigcpar{C^+=C^+(G,ab,xy) \: :\: (\sigma, ab, xy)\in \mathbb{G}}\big| \; \ge \; \frac{1}{8} U_G. 
\end{equation}
\end{lemma}

We close this subsection with an observation about the number of small edges in switching-partners. 
\begin{observation}[Small edges in switching-partners]\label{obs4D}
For any pair of switching-partners~$C^+,C^-$ the following holds: 
if~$C^-$ contains any configuration-graph~$G\in \cN_{z}$, then~$C^+\subseteq \cN_{z+4\D}$.
\end{observation}
\begin{proof}
By construction, any two configuration-graphs in~$C^+\cup C^-$ can differ in at most~$4\D$ edges.
Therefore the number of small edges of any configuration-graph in~$C^+$ can differ by at most~$4\D$ from the number of small edges of~${G \in \cN_z}$, 
which in turn establishes~${C^+ \subseteq \cN_{z+4\D}}$. 
\end{proof}

\subsection{Double-counting argument: proof of \refT{main}}\label{sec:doublecounting}
This subsection is devoted to the proof of our main technical result~\refT{main}, and we start by outlining our proof strategy. 
Our main step will be to prove that, when $z$ is sufficiently small, decreasing $z$ by an additive constant will decrease $Z(\cN_z)$ be a multiplicative constant. 
Specifically, we will prove that the following key inequality holds (for all sufficiently large~$n$): 
there is a constant~$\tau \in (0,1)$ such that 
\begin{equation}\label{eq:main:goal}
Z(\cN_z) \le (1-\tau) {Z(\cN_{z+5\Delta})} \qquad \text{for all~$0 \le z \le \gamma \mu$,}
\end{equation}
where $\gamma$ comes from Lemmas~\ref{coefficients} and~\ref{goodpermutation} 
(and the usage of the additive constant $5\Delta$ is closely related to Observation~\ref{obs4D}, as we shall see). 
Iterating~\eqref{eq:main:goal} implies that decreasing~$z$ by~$\Theta(n)$ will decrease $Z(\cN_z)$ by a factor that is exponentially small in~$\Theta(n)$.
In particular, setting $\alpha:=\gamma/2$ and noting that $\mu\geq (\xi n)^2/(2\D n) = \Theta(n)$, we thus readily obtain~that
\[Z(\cN_{\alpha \mu})\leq e^{-\Theta(n)} \cdot Z(\cN_{2\alpha \mu}).\]
Since~$\bP\bigpar{\Gpodn\in\cN_z}$ is proportional to~$Z(\cN_z)$, see~\refS{sec:probabilities}, this immediately implies \refT{main}.

Our strategy for proving the key inequality~\eqref{eq:main:goal} focuses on clusters.
Recall that each configuration-graph~$G$ lies in $L_G$~lower clusters and $U_G$~upper clusters. 
Lemma~\ref{coefficients} tells us that the values of $L_G, U_G$ are very close over all configuration-graphs~$G \in \cN_{\gamma\mu}$; 
close enough that~\eqref{eq:main:goal} will follow from proving 
\begin{equation}\label{eq:main:approx}
\sum_{G\in\cN_{z}} L_G Z(G) \le (1-2\tau)\sum_{G\in\cN_{z+5\Delta}} U_G Z(G),
\end{equation}
and taking~$\gamma$ to be sufficiently small in terms of~$\tau$. 
For technical reasons, we then define a suitable set~$\cL$ of lower clusters~$C^-$ with the property that~${\cN_{z} \subseteq \bigcup_{C^-\in \cL} C^-}$. 
Furthermore, we let~$\cU$ contain all upper clusters~$C^+$ that are switching-partners of the lower clusters~$C^-$ in~$\cL$, which will turn out to satisfy~${\bigcup_{C^+\in \cU} C^+ \subseteq \cN_{z+5\Delta}}$ by Observation~\ref{obs4D}.  
Hence inequality~\eqref{eq:main:approx} would be implied by 
\begin{equation}\label{eq:main:approx2}
\sum_{C^-\in\cL} Z(C^-) \le (1-2\tau)\sum_{C^+\in\cU} Z(C^+).
\end{equation}
Since $\cU$ consists of the switching-partners of the lower clusters in $\cL$, \refL{switch} implies the following weaker bound: ${\sum_{C^-\in\cL} Z(C^-)}\leq {\sum_{C^+\in\cU} Z(C^+)}$. 
If a positive proportion of the clusters in $\cU$ are good, then \refL{goodswitch} allows us to improve this bound to obtain~\eqref{eq:main:approx2} and thus inequality~\eqref{eq:main:approx}. 
In the remaining case where rather few clusters of~$\cU$ are good, inequality~\eqref{eq:main:approx2} might not hold. But
it turns out that, using \refL{goodpermutation}, we are able to compare~$\sum_{C^+\in\cU} Z(C^+)$ with~$\sum_{G\in\cN_{z}} L_G Z(G)$ via the contributions of good edge-sequences 
(and here the careful definition of~$\cL$ in~\eqref{def:cL} below with respect to~$\cN_{z+1}$ rather than~$\cN_z$ will matter).  
This will allow us to complete the proof of inequality~\eqref{eq:main:approx} by a technical case~analysis; see~\eqref{bound:firstcase}--\eqref{secondgap}~below for the details. 
\begin{proof}[Proof of \refT{main}]
As mentioned in the above proof outline, our main goal is to prove inequality~\eqref{eq:main:goal}. 
Indeed, by iterating this key inequality, in view of \refS{sec:probabilities} and~$\mu=\Theta(n)$ it then follows that
\begin{equation}\label{eq:main:goal:cor}
\begin{split}
\bP\bigpar{\bigabs{\Xs\bigpar{\Gpodn}-\mu} \le \gamma\mu/2}
& \le  
\frac{\bP\bigpar{\bigabs{\Xs\bigpar{\Gpodn}-\mu} \le \gamma\mu/2}}{\bP\bigpar{\bigabs{\Xs\bigpar{\Gpodn}-\mu} \le \gamma\mu/2 + 5\Delta \cdot \floor{\gamma \mu/(10\Delta)}}}\\
& = \frac{Z(\cN_{\gamma\mu/2})}{Z(\cN_{\gamma\mu/2+5\Delta \cdot \floor{\gamma \mu/(10\Delta)}})} 
\le \bigpar{1-\tau}^{\floor{\gamma \mu/(10\Delta)}}
\le e^{-\Theta(n)} ,
\end{split}
\end{equation}
which completes the proof with~$\alpha := \gamma/2$. 

It thus remains to prove inequality~\eqref{eq:main:goal}.
To this end we will henceforth tacitly assume that~$0 \le z\le \gamma\mu$, and also that~$\gamma>0$ is sufficiently small (whenever necessary).
We shall first switch the focus of inequality~\eqref{eq:main:goal} from configuration-graphs to clusters, leveraging that any~$G \in \cN_{\gamma\mu}$ is contained in roughly the same number of clusters. 
Indeed, recalling the definition~\eqref{def:ZS} of~$Z(\cN_\ell)$, \refL{coefficients} implies~that
\begin{equation}\label{firstgap}
    \frac{Z(\cN_z)}{Z(\cN_{z+5\Delta})}= \frac{\sum_{G\in\cN_{z}} Z(G)}{\sum_{G\in\cN_{z+5\Delta}} Z(G)} \le (1+O(\gamma)) \cdot \frac{\sum_{G\in\cN_{z}} L_GZ(G)}{\sum_{G\in\cN_{z+5\Delta}} U_GZ(G)},    
\end{equation}
where the implicit constant in $O(\gamma)$ is uniform i.e., does not depend on $\gamma$.
The next step is to further rewrite the right-hand side of inequality~\eqref{firstgap} in terms of clusters, 
for which we~define 
\begin{equation}\label{def:cL}
\begin{split}
\cL &:= \bigcpar{C^- : \: \text{lower cluster with $C^-\cap \cN_{z+1}\neq \emptyset$}} ,\\
\cU &:= \bigcpar{C^+ : \: \text{upper cluster that is switching partner of $C^- \in \cL$}} .
\end{split}
\end{equation}
Since every~$G \in \cN_{z+1}$ lies in at least one lower cluster by \refL{coefficients}, we have~$\cN_{z} \subseteq {\cN_{z+1} \subseteq {\bigcup_{C^-\in \cL} C^-}}$.
Observation~\ref{obs4D} (applied with~$z$ replaced by~$z+1$) implies~${\bigcup_{C^+\in \cU} C^+ \subseteq \cN_{z+1+4\Delta}} \subseteq \cN_{z+5\Delta}$.
It follows~that
\begin{equation}\label{eq:ZL:ZU:relevance}
\begin{split}
\sum_{G\in\cN_{z}} L_GZ(G) & \le \sum_{C^- \in \cL} Z(C^-) =: Z(\cL) , \\
\sum_{G\in\cN_{z+5\Delta}} U_GZ(G) & \ge \sum_{C^+ \in \cU} Z(C^+) =: Z(\cU).
\end{split}
\end{equation}

With the cluster based inequalities from~\eqref{eq:ZL:ZU:relevance} in hand, 
we are now in a position to estimate the right-hand side of~\eqref{firstgap} from above. 
Let~$\cG$ denote the set of good clusters, where good clusters are defined around~\eqref{goodclusterdef}.
Decomposing~$Z(\cL)$ into two parts,
by invoking Lemmas~\ref{switch} and~\ref{goodswitch} separately for~${C^+ \in \cG^c}$ and ${C^+ \in \cG}$ 
it follows~that 
\begin{equation}\label{bound:ZLZU}
\begin{split}
Z(\cL) & 	=\sum_{C^-: C^+\in \cU\cap \cG^c} Z(C^-) + \sum_{C^-: C^+\in \cU\cap \cG} Z(C^-)\\
        &\le \sum_{C^+\in \cU\cap \cG^c} Z(C^+) + \frac{1}{1+\epsilon}\sum_{C^+\in \cU\cap \cG} Z(C^+)
				\\ &
				=Z(\cU)-\frac{\epsilon}{1+\epsilon}\sum_{C^+\in \cU\cap \cG} Z(C^+),
\end{split}
\end{equation}
where in the summations $C^-$ denotes the switching-partner of $C^+$ (defined by the bijection discussed earlier). 

We now distinguish two cases.
If~$\sum_{C^+\in \cU\cap \cG} Z(C^+) \ge  Z(\cU)/32$, then inequality~\eqref{bound:ZLZU} implies 
\begin{align}\label{bound:firstcase}
Z(\cL) \le (1-\nu) Z(\cU)   
\end{align}
for~$\nu:= \epsilon/[32(1+\epsilon)] \in (0,1)$, say.

Otherwise~$\sum_{C^+\in \cU\cap \cG} Z(C^+) \le Z(\cU)/32$ holds, 
in which case we shall focus on the contribution of good edge-sequences 
to~$Z(\cU)$.
Using the definition of good upper clusters (see page~\pageref{def:good}), it follows that the contribution of good edge-sequences is at most 
\begin{equation}\label{bound:secondcase:1}
\begin{split}
		 \sum_{C^+\in \cU} \sum_{G \in C^+}\sum_{\substack{\sigma \in \Pi_{G}:\\(\sigma, ab, xy)\in \mathbb{G}}} Z(\sigma)
& \le \sum_{C^+\in \cU} \Bigpar{\indic{C^+\in \cG}Z(C^+) + \indic{C^+\not\in \cG}\tfrac{1}{16}Z(C^+)}\\
& = \sum_{C^+\in \cU\cap \cG} Z(C^+)+\frac{1}{16}\sum_{C^+\in \cU\cap \cG^c} Z(C^+) \le \frac{3}{32} Z(\cU) ,
\end{split}
\end{equation}
where the edges $ab,xy$ in the third summation are determined by the corresponding upper cluster~$C^+=C^+(G,ab,xy)$.
To bound the contribution of good edge-sequences from below, we next analyze the left-hand side of~\eqref{bound:secondcase:1} more carefully. 
\refL{coefficients} implies that every~${G^+ \in \cN_z}$ is contained in at least one upper cluster.
We now claim that any upper cluster~$C^+$ containing some~${G^+  \in \cN_z}$ satisfies~${C^+ \in \cU}$.
To see this, note that by construction its switching-partner~${G^- \in C^-}$ (see page~\pageref{oclusters}) contains~$X_k(G^-)={X_k(G^+)-1}$ small edges, which in view of~\eqref{def:Nz} implies~$G^- \in \cN_{z+1}$.
Recalling the definition~\eqref{def:cL} of~$\cL$ it follows that~${C^- \in \cL}$, which by construction gives~${C^+ \in \cU}$, as claimed. 
This in particular implies~$\cN_{z}\subseteq {\bigcup_{C^+\in \cU} C^+}$, and using \refL{goodpermutation} it follows~that 
\begin{equation}\label{bound:secondcase:2}
\begin{split}
\sum_{C^+\in \cU} \sum_{G \in C^+}\sum_{\substack{\sigma \in \Pi_{G}:\\(\sigma, ab, xy)\in \mathbb{G}}} Z(\sigma)
& \ge \sum_{G\in \cN_z} \sum_{\sigma \in \Pi_{G}} Z(\sigma) \cdot \big|\bigcpar{C^+=C^+(G,ab,xy) \: :\: (\sigma, ab, xy)\in \mathbb{G}}\big|\\
& \ge \sum_{G\in \cN_z} Z(G) \cdot \frac{1}{8}U_G.
\end{split}
\end{equation}
Combining~\eqref{bound:secondcase:2} and~\eqref{bound:secondcase:1}, using \refL{coefficients} we infer that 
\begin{equation}\label{bound:secondcase}
    Z(\cU)\ge \frac{4}{3} \sum_{G\in\cN_z} U_GZ(G) \ge (1-O(\gamma)) \cdot \frac{4}{3} \sum_{G\in\cN_z} L_GZ(G) .
\end{equation}

Finally, by combining the two above-discussed cases with~\eqref{eq:ZL:ZU:relevance}, 
using inequalities~\eqref{bound:firstcase} and~\eqref{bound:secondcase} it follows that we always~have 
\begin{equation}\label{secondgap}
    \frac{\sum_{G\in\cN_{z}} L_GZ(G)}{\sum_{G\in\cN_{z+5\Delta}} U_GZ(G)}\le 
		\min\biggcpar{\frac{Z(\cL)}{Z(\cU)}, \: \frac{\sum_{G\in\cN_{z}} L_GZ(G)}{Z(\cU)}} 
		\le \max\biggcpar{1-\nu, \: \frac{1}{4/3 - O(\gamma)}},
\end{equation}
which together with estimate~\eqref{firstgap} establishes inequality~\eqref{eq:main:goal} for sufficiently small~$\gamma>0$. 
\end{proof}

\section{Switching results: deferred proofs}\label{sec:switching}
This section is devoted to the deferred proofs of the `cluster switching' results \refL{switch} and~\ref{goodswitch} from \refS{sec:switchingcounting}. 
To avoid clutter, we henceforth always tacitly assume that~$n \ge n_0(\Delta,\xi)$ is sufficiently~large (whenever~necessary).
Furthermore, recalling that there is a natural bijection between the upper clusters and the lower clusters, 
for the rest of the section we fix an upper cluster ${C^+=C^+(G^+,ab,xy)}$ and its switching-partner lower cluster ${C^-=C^-(G^-,ax,by)}$, as defined in \refS{oclusters}.  
We also define
\begin{equation}\label{def:P}
\cP^-:=\bigcup_{G\in C^-} \Pi_G 
\qquad \text{ and } \qquad 
\cP^+:=\bigcup_{G\in C^+} \Pi_G
\end{equation} 
to be the set of edge-sequences of configuration-graphs that lie in $C^-$ and $C^+$, respectively. 
Thus, if $\s\in\cP^+$ then $\s$ contains the edges $ab,xy$, while if $\s\in\cP^-$ then $\s$ contains the edges $ax,by$.

In our upcoming switching based proofs it will be important to understand when certain points are chosen. 
More specifically, at any point in our proof, when we are comparing two edge-sequences the step at which any vertex beside $A, B, X, Y$ saturates 
(to clarify: we say that vertex~$A$ saturates at step~$i$, if $A$ is saturated for the first time after the first~$i$ edges are added) 
will be the same in both edge-sequences  
(where $A, B, X, Y$ denote the vertices that contain the points $a,b,x,y$, as usual). 
This allows us to focus on the step where points of~${A, B, X, Y}$ are chosen. 
For any edge-sequence $\sigma=(e_1,\dots, e_m)$ of the edges of a configuration-graph~$G$, 
we~let 
\begin{equation}\label{def:se}
\s(e)\in \{1,\dots,m\}
\end{equation} 
be the unique edge-number such that~$e=e_{\s(e)}$. Similarly, for any point~$v$ of~$G$, we let $\s(v)\in\{1,\dots,m\}$ be the unique edge-number such that~$v\in e_{\sigma(v)}$.
For our analysis, we will need to focus  on the last step at which a point of the vertex~$A$ besides~$a$ is chosen in~$\sigma$, which we formally denote as
\begin{equation}\label{def:ta}
t_a=t_a(\sigma):=
\begin{cases}
\max\bigcpar{\sigma(a') \: : \: a'\in A \setminus \{a\}} \quad &\text{if $|A|>1$,}\\
0\quad &\text{otherwise.}
\end{cases}
\end{equation}
The crux will be that~$\sigma(a)$ and~$t_a$ together determine the step~$\max\{\s(a), t_a\}$ at which vertex~$A$ saturates in~$\sigma$, 
which will be important for understanding the contributions of the vertex~$A$ to~$Z(\sigma)$.
We define~$t_b, t_x, t_y$ analogously to~$t_a$.

With the notation above, \refL{switch} can be stated as
\begin{equation}\label{el7}
\sum_{\sigma\in\cP^+}Z(\s)\geq \sum_{\sigma\in\cP^-}Z(\s).
\end{equation}
Our proof strategy for inequality~\eqref{el7} is as follows: 
in \refS{sec:twins} we will construct pairs of edge-sequences ${(\s_1,\s_2) \in \cP^+ \times \cP^-}$, called {\em twins}, which will satisfy $Z(\s_1)=Z(\s_2)$. 
Hence~\eqref{el7} holds with equality when summing only over those paired edge-sequences.
The unpaired $\sigma\in \cP^+\cup \cP^-$ will have the very useful property that $t_a\leq t_y$ and $t_b\leq t_x$. 
In such case the edge $ab$, especially if located early in the sequence, will help to saturate $A$ and $B$ early on. 
This suggests that $Z(\sigma)$ should typically have higher values when~$ab$ is in~$\sigma$, i.e., $\sigma\in \cP^+$. 
Using this property, we will see that~\eqref{el7} also holds when summing only over those unpaired members.  
Combining these estimates will eventually establish~\eqref{el7} and thus \refL{switch}; see~\refS{sec:goodswitch}.

Recall that two configuration-graphs~$G^+\in C^+$ and  $G^-\in C^-$~are switching-partners if~$G^-$ is obtained from~$G^+$ by replacing edges~$ab, xy$ with~$ax, by$, see~\refS{sec:clusterdefinition}. 
This notion extends naturally to edge-sequences of configuration-graphs that lie in $C^+$ and $C^-$, respectively.
\begin{definition}[Counterparts]\label{dcounterparts2} 
We say that the edge-sequences~$\s\in\cP^+$ and~$\s'\in\cP^-$ are {\em counterparts} if~$\s'$ is obtained from~$\s$ by replacing~$ab$ with~$ax$ and~$xy$ with~$by$.
\end{definition}
Note that two counterparts have the same values of $t_a,t_b,t_x$ and $t_y$; this is one of the motivations for defining those values.  
Note also that the vertices $A,Y$ become saturated at the same step of two counterparts, but that $B$ and $X$ may become saturated at different steps. 
This means that, in general, $Z(\s)$ differs from~$Z(\s')$. 
When defining pairs of twins in \refS{sec:twins}, we will account for this by rearranging some of the edges involving $b,x$ so that $t_b,t_x$ exchange values, while $t_a,t_y$ remain the same. 
As a result, every vertex except $B$ and $X$ becomes saturated at the same step in $\s$ and in its twin. Furthermore, the step at which $B$ and $X$ saturates will swap between twins.
Therefore, the twins will both have the same $Z$-value, as desired; see \refS{sec:twins} for the~details.

Besides replacing edges as in counterparts, we also need an operation that changes the relative order of appearance of the switching-relevant edges. 
First, for $\s\in \cP^+$, we define $\ov{\s}$ to be the edge-sequence obtained from $\s$ by swapping the positions of~$ab$ and~$xy$. 
Similarly, for $\s\in \cP^-$ we define $\ov{\s}$ to be the edge-sequence obtained from $\s$ by swapping the positions of~$ax$ and~$by$.
For later reference, we now record some basic properties of these natural operations on edge-sequences (whose routine verification we omit).
\begin{observation}\label{fact:switch}
For all counterparts~$\s\in\cP^+$ and $\s'\in\cP^-$ the following properties hold:\vspace{-0.5em}%
\begin{romenumerate}
\parskip 0em  \partopsep=0pt \parsep 0em \itemsep 0.0675em
\item\label{13:1} The edge-sequences~$\ov{\s}$ and~$\ov{\s'}$ are counterparts.
\item\label{13:2} The values $t_a,t_b,t_x,t_y$ are the same for $\s, \s', \ov{\s}, \ov{\s'}$.
\item\label{13:3}  The map $\sigma\mapsto \osigma$ is a bijection between~$\cP^+$ and~$\cP^+$, and a bijection between~$\cP^-$ and~$\cP^-$.
\end{romenumerate}
\end{observation}

In the following lemma, $\cT$ denotes the set of edge-sequences which have a twin; see \refS{sec:summingtwins}.
As indicated above, an edge-sequence $\sigma\in \cT$ and its twin will have the same $Z$-value, which allows us to obtain part~\ref{lem:twins:sum} in the following lemma 
(the technical property from part~\ref{lem:twins:techn} was also discussed above). 
\begin{lemma}[Properties of twins]\label{twins}
There exists~$\cT\subseteq\cP^+\cup\cP^-$ satisfying the following properties:\vspace{-0.5em}%
\begin{romenumerate}
\parskip 0em  \partopsep=0pt \parsep 0em \itemsep 0.0675em
\item\label{lem:twins:sum} $\sum_{\sigma \in \cT \cap \cP^+}Z(\sigma) = \sum_{\sigma \in \cT \cap \cP^-}Z(\sigma)$. 
\item\label{lem:twins:equiv} For all counterparts $\ov{\s}, \ov{\s'}$ we have: $\s\in\cT$ iff $\ov{\s}\in\cT$ iff $\s'\in\cT$ iff $\ov{\s'}\in\cT$.
\item\label{lem:twins:techn}
If $\s\notin\cT$ then $t_a\le t_y$ and $t_b\le t_x$ in $\s$. 
\item\label{lem:twins:techn2} If in $\sigma\in \cP^+\cup\cP^-$ all edges incident to $X,Y$ appears after all edges incident to $A,B$, then $\sigma\not\in \cT$.
\end{romenumerate}
\end{lemma}

Recall that in order to prove~\refL{switch}, our strategy is to first sum over edge-sequences with a twin, for which we shall invoke part~\ref{lem:twins:sum} of \refL{twins}. 
When summing over the remaining edge-sequences with no twin, which we call \emph{non-twins}, 
in~\refS{sec:goodswitch} we will invoke part~\ref{completeswitch:1} of the following lemma (which applies due to part \ref{lem:twins:techn} of \refL{twins}). 
For~\refL{goodswitch} we will use a variant of this strategy, 
exploiting the stronger conclusion of part~\ref{completeswitch:2} of \refL{completeswitch} 
in order to obtain the desired extra $1+\eps$~factor.  
\begin{lemma}[Switching of edge-sequences]\label{completeswitch}
There is~${\epsilon'=\epsilon'(\zeta)>0}$ such that, 
for all counterparts~$\s\in\cP^+$ and $\s'\in\cP^-$ satisfying~$t_a\le t_y$ and~$t_b\le t_x$,
the following holds:\vspace{-0.5em}%
\begin{romenumerate}
\parskip 0em  \partopsep=0pt \parsep 0em \itemsep 0.0675em
\item \label{completeswitch:1} 
${Z(\sigma)+Z(\ov{\sigma})}\ge {Z(\sigma')+Z(\ov{\sigma'})}$. 
\item \label{completeswitch:2}
If~${\min\{\s(xy), t_x, t_y\}}-{\max\{\s(ab), t_a, t_b\}} \ge {\zeta m/3}$, then ${Z(\sigma)+Z(\ov{\sigma})}\ge (1+\epsilon')[Z(\sigma')+Z(\ov{\sigma'})]$.
\end{romenumerate}%
\end{lemma}

Before giving the proofs of the above two auxiliary lemmas, 
 we shall first use both to prove the `cluster switching' results Lemmas~\ref{switch} and~\ref{goodswitch} in~\refS{sec:goodswitch}.
Afterwards, in~\refS{sec:completeswitch} we use edge-switching arguments to prove \refL{completeswitch}, 
and in~\refS{sec:twins} we then define~twins and prove~\refL{twins}.

\begin{remark}[Special case~$k=1$]\label{rem:k1:simpler}
In the case~$k=1$ we have~${t_a=t_b=0}$, and so every pair of counterparts~${\s,\s'}$ satisfies~${t_a\le t_y}$ and~${t_b\le t_x}$, and thus \refL{completeswitch}  applies. Hence there is no need for \refL{twins} or even the concept of twins. 
This is one place where the case~$k=1$ is much simpler than the general case.
\end{remark}

\subsection{Switching of clusters: proofs of Lemmas~\ref{switch} and~\ref{goodswitch}}\label{sec:goodswitch}
To prove the basic cluster switching result \refL{switch}, 
we shall invoke \refL{twins} for twins~$\sigma$ and \refL{completeswitch} for non-twins~$\sigma$, as discussed. 
\begin{proof}[Proof of \refL{switch}]
Recalling \refO{fact:switch}~\ref{13:3} as well as twins properties \ref{lem:twins:equiv} and \ref{lem:twins:techn} from \refL{twins},
using part~\ref{completeswitch:1} of \refL{completeswitch} it follows that non-twins satisfy
\[\sum_{\sigma\in  \cP^+\bk\cT} Z(\sigma)=\frac12 \sum_{\sigma\in  \cP^+\bk\cT}\bigsqpar{ Z(\sigma)+Z(\osigma)}\ge \frac12 \sum_{\sigma\in  \cP^-\bk\cT} \bigsqpar{Z(\sigma)+Z(\osigma)}=\sum_{\sigma\in  \cP^-\bk\cT} Z(\sigma).\]
Combining this inequality with part~\ref{lem:twins:sum} of \refL{twins} for twins, we infer that 
\[Z(C^+)=\sum_{\sigma\in\cP^+\cap\cT} Z(\sigma)+\sum_{\sigma\in  \cP^+\bk\cT} Z(\sigma)\ge \sum_{\sigma\in  \cP^-\cap\cT} Z(\sigma)+\sum_{\sigma\in  \cP^-\bk\cT} Z(\sigma)=Z(C^-),\]
as claimed. 
\end{proof}

Turning to the proof of the stronger switching result \refL{goodswitch} for good clusters~$C^+=C^+(G,ab,xy)$, 
to work with good edge-sequences in $G\in C^+$ it is convenient to define the shorthand 
\begin{equation}\label{def:cPC_Gabxy}
\bG_{ab,xy} :=\bigcpar{\sigma: \: (\sigma,ab,xy)\in \bG}, 
\end{equation}
where~$\bG$ is defined as in~\eqref{def:bG}. 
Our strategy will be to construct a `super nice' subset~$\cS \subseteq \cP^+ \cap \bG_{ab,xy}\cap \cT^c$ of non-twins that satisfy the `gap' hypothesis of~\refL{completeswitch}~\ref{completeswitch:2}. 
In comparison with the above proof of \refL{switch}, this eventually allows us to gain the desired extra $1+\eps$ factor by invoking the stronger conclusion of \refL{completeswitch}~\ref{completeswitch:2} for non-twins~$\sigma \in \cS$. 
\begin{proof}[Proof of \refL{goodswitch}]
Let~$\cS\subseteq \cP^+$ denote the set of edge-sequences~$\sigma \in \cP^+$ with the following two properties:
(a)~that ${\min\{\sigma(xy), t_x, t_y\}}-{\max\{\sigma(ab), t_a, t_b\}}\ge {\zeta m/3}$ holds, 
and (b)~that all edges incident to $X,Y$ appear after all edges incident to~$A,B$. 
By property~\ref{lem:twins:techn2} of twins (see \refL{twins}) it follows that~$S\subseteq \cT^c$. 
Recalling that the map $\sigma \mapsto \sigma'$ is a bijection between~$\cP^+$ and~$\cP^-$ (see Observation~\ref{fact:switch}), 
we shall now decompose~$Z(C^-) = \sum_{\sigma \in C^+}Z(\sigma')$ into several parts, 
and apply different inequalities to each part. 
Recalling Observation~\ref{fact:switch} and twins property~\ref{lem:twins:equiv}, 
by invoking parts~\ref{completeswitch:1} and~\ref{completeswitch:2} of \refL{completeswitch} separately for~$\sigma\in \cS^c$ and~$\sigma \in \cS$, 
together with \refL{twins} it follows that there is~$\epsilon'=\epsilon'(\zeta)>0$ such that 
\begin{equation}\label{eq:goodswitch:decomposition}
\begin{split}
Z(C^-) & = \frac{1}{2}\sum_{\sigma\in \cS} \bigsqpar{Z(\sigma')+Z(\osigma')}+\frac{1}{2}\sum_{\sigma\in \cP^+\cap \cT^c\cap \cS^c} \bigsqpar{Z(\sigma')+Z(\osigma')}+\sum_{\sigma\in \cP^+\cap \cT} Z(\sigma')\\
& \le \frac{1}{(1+\epsilon')}\frac{1}{2}\sum_{\sigma\in  \cS} \bigsqpar{Z(\sigma)+Z(\osigma)}+\frac{1}{2}\sum_{\sigma\in \cP^+\cap \cT^c\cap \cS^c} \bigsqpar{Z(\sigma)+Z(\osigma)}+\sum_{\sigma\in \cP^+\cap \cT} Z(\sigma) \\
& = Z(C^+)-\frac{\epsilon'}{(1+\epsilon')}\frac{1}{2}\sum_{\sigma\in  \cS} \bigsqpar{Z(\sigma)+Z(\osigma)} .	
\end{split}
\end{equation}
It thus suffices to prove that~$\sum_{\sigma\in \cS} Z(\sigma) \ge \varrho Z(C^+)$ for some $\varrho=\varrho(\zeta,\Delta)>0$. 
Indeed, for~${\eps:=\eps'/(1+\epsilon') \cdot \varrho/2}$ we then infer the desired result~$Z(C^+) \ge (1+\epsilon)Z(C^-)$ by noting that 
\begin{equation}\label{eq:goodswitch:goal}
Z(C^-) \le (1-\epsilon)  Z(C^+) \le Z(C^+)/(1+\epsilon).
\end{equation}

In the remaining proof of~$\sum_{\sigma\in \cS} Z(\sigma) \ge \varrho Z(C^+)$, we shall work with a fixed lexicographic ordering of all possible edges. 
The following mapping ${f:\cP^+\cap \mathbb{G}_{ab,xy} \to \cS}$ is at the heart of our argument. 
Given an edge-sequence $\sigma\in\cP^+\cap \mathbb{G}_{ab,xy}$, we write $e_1,\ldots,e_\ell$ for an enumeration of all edges incident to~$X,Y$ (including~$xy$) according to our lexicographic ordering. 
For~$i=1,\ldots,\ell$ we now sequentially proceed as follows: 
if the edge~$e_i$ is not in the last $\frac{2}{3}\zeta m-2\Delta$ steps of our current edge-sequence, then we repeatedly shift it back by exactly $\lfloor\frac{1}{3}\zeta m\rfloor$ steps until it is.
Writing~$f(\sigma)$ for the edge-sequence resulting from this shifting process, we~define 
\begin{equation}
\cS_f := \bigcpar{ f(\sigma): \: \sigma \in \cP^+\cap \mathbb{G}_{ab,xy} }. 
\end{equation}

We claim that~$\cS_f \subseteq \cS$. 
Fix~$\sigma^* \in \cS_f$, where by construction~$\sigma^* =f(\sigma)$ for some $\sigma \in \cP^+\cap \mathbb{G}_{ab,xy}$. 
Note that $\sigma\in \bG_{ab,xy}$ implies~$\sigma^*\in \bG_{ab,xy}$, 
and that the edges~$e_1,\ldots,e_\ell$ all appear in the last $\frac{2}{3}\zeta m$ steps of the new edge-sequence~$\sigma^*$ (since~$\ell \le 2\Delta$).  
Furthermore, since~$A,B$ are not adjacent to~$X,Y$ in~$\sigma \in \cP^+\cap \mathbb{G}_{ab,xy}$ by definition of good clusters, 
it follows that~$A,B$ are saturated in the first $(1-\zeta)m$ steps in~$\sigma$ and thus~$\sigma^*$ (since we do not shift edges incident to~$A,B$). 
Putting things together, we infer that~$\sigma^*$ satisfies both properties~(a) and~(b). 
Therefore~$\sigma^*\in \cS$, establishing the claim that~$\cS_f \subseteq \cS$.

We now bound the preimage size $\bigabs{f^{-1}(\sigma^*)}$ for~$\sigma^* \in \cS_f$. 
Note that $\sigma$ and $\sigma^*=f(\sigma)$ contain the same sets of edges, 
so from $\sigma^*$ we can uniquely recover the lexicographic enumeration of the edges~$e_1,e_2,\dots,e_\ell$ incident to~$X,Y$. 
The crux is that if we are given~$\sigma^*$ and the number of times each edge~$e_i$ is shifted backward (which is at most $\ceil{{m}/{\lfloor\frac{1}{3}\zeta m\rfloor}}$ by construction), 
then we can uniquely recover~$\sigma$. 
Since~$\ell\le 2\Delta$, it follows~that 
\begin{equation}\label{eq:goodswitch:preimage}
\bigabs{f^{-1}(\sigma^*)} \le \left(\biggceil{\frac{m}{\lfloor\frac{1}{3}\zeta m\rfloor}}\right)^{\ell} \le \lrpar{\frac{6}{\zeta}}^{2\Delta} .
\end{equation}

We next bound the ratio~$Z(\sigma^*)/Z(\sigma)$ for~$\sigma^*=f(\sigma)$. 
Since the shifted edges are shifted back by $\floor{\frac{1}{3}\zeta m}$ steps until they are in the last $\frac{2}{3} \zeta m -2\Delta$ steps, 
it follows that none of them are shifted back into the last $\frac{1}{3}m\zeta -2\Delta$ steps. As a consequence we have $\Gamma_i(\sigma) = \Gamma_i(\sigma^*)$ for $i>\ceil{(1-\zeta/3)m} +2\Delta$, 
since $\sigma,\sigma^*$ are identical in the last $\frac{1}{3}m\zeta -2\Delta$ steps.
Note that the mapping~$f$ shifts back edges whose endpoints belong to at most $2\Delta +2$ distinct vertices ($X,Y$ and their neighbours). These are the only vertices whose saturation can be delayed by~$f$, therefore $\Gamma_i(\sigma)\ge {\Gamma_i(\sigma^*)-(2\Delta+2)}$.
Recalling the definition~\eqref{def:Zsigma} of~$Z(\sigma)$ and that Observation~\ref{obs:unsaturated} gives $\Gamma_{i}(\sigma^*) \ge {2(m-i)/\Delta}$, 
by combining the aforementioned bounds with~$m=\Theta(n)$ and~$1 \le \Delta = O(1)$ as well as the estimate~${1-x \ge e^{-2x}}$ for~$x \in [0,1/2]$,
it follows~that, say, 
\begin{equation}\label{eq:goodswitch:massratio}
\begin{split}
\frac{Z(\sigma^*)}{Z(\sigma)} 
&= \prod_{0 \le i \le \ceil{(1-\zeta/3)m} +2\Delta} \frac{\Gamma_i(\sigma)(\Gamma_i(\sigma)-1)}{\Gamma_i(\sigma^*)(\Gamma_i(\sigma^*)-1)}\\
&\ge \prod_{0 \le i \le \ceil{(1-\zeta/3)m} +2\Delta} \left(1-\frac{2\Delta+2}{\Gamma_i(\sigma^*)}\right)\left(1-\frac{2\Delta+2}{\Gamma_i(\sigma^*)-1}\right)\\
& \ge \left(1-\frac{2\Delta+2}{\frac{\zeta m}{3\Delta}}\right)^m\left(1-\frac{2\Delta+2}{\frac{\zeta m}{3\Delta}}\right)^m\ge e^{-48\Delta^2/\zeta}.
\end{split}
\end{equation}

Finally, combining~$\cS_f \subseteq \cS$ with the definition of good clusters and~\eqref{eq:goodswitch:preimage}--\eqref{eq:goodswitch:massratio}, we infer that 
\begin{align*}
       \frac{\sum_{\sigma^*\in \cS} Z(\sigma^*)}{Z(C^+)}& \ge \frac{\sum_{\sigma^*\in \cS_f} Z(\sigma^*)}{\sum_{\sigma\in \cP^+\cap \mathbb{G}_{ab,xy}} Z(\sigma)} \cdot \frac{{\sum_{ \sigma\in \cP^+\cap \mathbb{G}_{ab,xy}} Z(\sigma)}}{Z(C^+)} \\
       &\ge \frac{\sum_{\sigma^*\in \cS_f} Z(\sigma^*)}{\sum_{\sigma^*\in \cS_f} \sum_{\sigma\in f^{-1}(\sigma^*)} Z(\sigma)} \cdot \frac{1}{16} 
       \ge \frac{e^{-48\Delta^2/\zeta}}{(6/\zeta)^{2\Delta}} \cdot \frac{1}{16} \cdot =: \varrho  ,
\end{align*}
which completes the proof of \refL{goodswitch}, as discussed. 
\end{proof}

\subsection{Switching of edge-sequences: proof of \refL{completeswitch}}\label{sec:completeswitch}
For an edge-sequence~$\sigma \in \Pi_G$ of~$G$, 
recall that the definition~\eqref{def:Zsigma} of the probability parameter~$Z(\sigma)$ is
\begin{equation*}
Z(\sigma)= \prod_{0 \le i \le m-1} \frac{2}{\Gamma_i(\sigma)(\Gamma_i(\sigma)-1)},
\end{equation*}
where $\Gamma_i=\Gamma_i(\sigma)$ denotes the number of unsaturated vertices after adding the first~$i$ edges of the edge-sequence~$\sigma$.
To prove the switching result \refL{completeswitch} (which is at the heart of our approach), we will analyze how the denominators~$\Gamma_i(\Gamma_i-1)$ differ between the edge-sequences~$\sigma, \osigma, \sigma', \osigma'$. 
A close inspection of these transformations reveals that $A, B, X, Y$ are the only vertices whose saturation steps might differ between~$\sigma, \osigma, \sigma', \osigma'$, 
so our analysis reduces to tracking when~$A, B, X, Y$ saturate in these~edge-sequences (recall that $A, B, X, Y$ denote the vertices that contain the points $a,b,x,y$). 
\begin{proof}[Proof  of \refL{completeswitch}]
Recall from \refO{fact:switch}~\ref{13:2} that the steps~$t_a, t_b, t_x, t_y$ defined around~\eqref{def:ta} are the same for all the four edge-sequences~$\sigma, \osigma, \sigma', \osigma'$.  
Note that, without loss of generality, we may henceforth assume that ${\sigma(ab) < \sigma(xy)}$ holds: 
indeed, part~\ref{completeswitch:2} is otherwise vacuous, and for part~\ref{completeswitch:1} we can otherwise exchange the role of $\sigma$ and $\ov{\sigma}$ (which swaps the order of $\sigma(ab)$ and $\sigma(xy)$, and also exchanges the roles of $\sigma'$ and $\ov{\sigma'}$ by Observation~\ref{fact:switch}). 
Our upcoming estimates formally use a case distinction depending on whether $t_a \le t_b$ or $t_b \le t_a$ holds. 
Both cases can be handled by the same argument (with routine notational changes), so we shall henceforth only consider the case where $t_a \le t_b$~holds. 

Recall that $\G_i(\s)$ is defined to be the number of unsaturated vertices {\em after} the addition of edge $i$. Thus, if a vertex $v$ becomes saturated with the addition of edge $uv$ then $v$ is {\em not} counted in $\G_i(\s)$ for $i=\s(uv)$.  For example, if $\sigma(ab)<t_a$ then $a$ is not counted in $\G_{t_a}(\s)$.

As mentioned above, a careful inspection reveals that~${A, B, X, Y}$ are the only vertices whose saturation steps might differ between~${\sigma, \osigma, \sigma', \osigma'}$. 
Based on the way the edge-replacements and order-swaps of the edge-sequences~$\sigma, \osigma, \sigma', \osigma'$ (see Definition~\ref{dcounterparts2} and above Observation~\ref{fact:switch}) alter only the order of the edges $\{ab,xy\}$ or $\{ax,by\}$ added in steps~$\sigma(ab)$ and~$\sigma(xy)$,  
it thus follows that for $i\not\in{[\sigma(ab),\sigma(xy)-1]}$ the number ${\Gamma_i(\pi)}$ of unsaturated vertices is the same for all edge-sequences~${\pi \in \{\sigma, \osigma, \sigma', \osigma'\}}$ (in fact, the same set of vertices are saturated in all four~sequences). 
With foresight, we now~introduce
\begin{equation}\label{def:ia}
i_a :=
\begin{cases}
\sigma(ab), & \text{if }  \sigma(ab)\ge t_a,\\
t_a, & \text{if } \sigma(ab)<t_a<\sigma(xy),\\
\sigma(xy), & \text{if } t_a\ge \sigma(xy),
\end{cases}
\end{equation}
and define~$i_b, i_x, i_y$ similarly. 
Intuitively, $i_a$ locates~$t_a$ with respect to the interval~$[\sigma(ab),\sigma(xy)]$. 
Note that the assumptions~${t_a \le t_x}$ and~${t_b \le t_y}$ of \refL{completeswitch} ensure together with our extra assumption~${t_a \le t_b}$ that ${t_a=\min\{t_a,t_b,t_x,t_y\}}$, 
which in turn implies ${i_a=\min\{i_a,i_b,i_x,i_y\}}$ by inspecting~\eqref{def:ia}. 
The following 
property of the number of unsaturated vertices in the four edge-sequences~${\sigma, \osigma, \sigma', \osigma'}$ is crucial for our argument. 
\begin{claim}\label{cl:unsaturated}
For each~${i\not\in [i_a, \max\{i_x,i_y\}-1]}$, the number ${\Gamma_i(\pi)}$ of unsaturated vertices is the same for all edge-sequences~${\pi \in \{\sigma, \osigma, \sigma', \osigma'\}}$.
\end{claim}
\begin{proof}
The claim about the number $\Gamma_i(\pi)$ of unsaturated vertices follows for~$i < \sigma(ab)$ and~$i \ge \sigma(xy)$ by the discussion above~\eqref{def:ia}, 
and follows for the remaining~$i$ by the case-by-case analysis below (keeping in mind that~${A, B, X, Y}$ are the only vertices whose saturation steps might differ).
%


(i)~If $\sigma(ab) \le i < i_a$, then by inspecting~\eqref{def:ia} it follows that $i < i_a \le t_a = \min\{t_a,t_b,t_x,t_y\}$, 
and so none of the vertices~${A, B, X, Y}$ are saturated in any of the four edge-sequences. 

(ii)~If~$\max\{i_x,i_y\} \le i < \sigma(xy)$, then exactly two of the vertices $A,B,X,Y$ are unsaturated in each of the four edge-sequences. 
To see this, note that in each sequence the two vertices chosen in step $\sigma(xy)$ are not saturated by the end of step~$i$.
Furthermore, using the assumptions $t_a \le t_x$ and $t_b \le t_y$ of \refL{completeswitch} we infer that $\max\{t_x,t_y\} \ge \max\{t_a,t_b\}$, 
which in turn implies~$\max\{i_x,i_y\} \ge \max\{i_a,i_b,\sigma(ab)\}$ by inspecting~\eqref{def:ia}. 
Hence in each sequence the two vertices chosen in step~$\sigma(ab)$ will be saturated by the end of step~$i$.
\end{proof}

From Claim~\ref{cl:unsaturated}, 
in view of the definition~\eqref{def:Zsigma} of~$Z(\sigma)$ we readily infer, defining
\begin{equation} \label{dZ'}
Z'(\pi) := \hspace{-0.125em}\prod_{i_a \le i < \max\{i_x,i_y\}}\hspace{-0.075em} \frac{2}{\Gamma_i(\pi)(\Gamma_i(\pi)-1)},
\end{equation}
that the following ratios are all equal to the same value:
\begin{equation}\label{def:Z'}
\frac{Z'(\sigma)}{Z(\sigma)}=\frac{Z'(\osigma)}{Z(\osigma)}=\frac{Z'(\sigma')}{Z(\sigma')}=\frac{Z'(\osigma')}{Z(\osigma')}.
\end{equation}
It follows that inequality~${Z(\sigma)+Z(\osigma)}\ge {Z(\sigma')+Z(\osigma')}$ from \refL{completeswitch} part~\ref{completeswitch:1} is equivalent to 
\begin{equation}\label{eq:completeswitch:equiv}
Z'(\sigma)+Z'(\osigma)\ge Z'(\sigma')+Z'(\osigma') ,
\end{equation}
and that the inequality from \refL{completeswitch} part~\ref{completeswitch:2} is equivalent to ${Z(\sigma)+Z(\osigma)}\ge {(1+\epsilon')[Z(\sigma')+Z(\osigma')]}$. 
Using several case distinctions, 
we shall now prove these equivalent inequalities. 

To avoid clutter in the upcoming estimates, for brevity
 we define $\Gamma^*_i$ to be $\Gamma_i(\s)$ plus the number of saturated vertices among $A, B, X, Y$ at the end of step~$i$ of $\s$.  Thus, $\Gamma^*_i$ is the size of the set of all unsaturated vertices at the end of step~$i$, unioned with $\{A,B,X,Y\}$, and this is true for step $i$ of any of the four edge-sequences $\sigma, \osigma, \sigma', \osigma'$.  This allows us to express $\Gamma_i$, for any of those sequences, as $\Gamma_i^*$ minus the number of $\{A,B,X,Y\}$ that are saturated at the end of step~$i$.

Recall that 
we have $t_a\leq t_b\leq t_x$ and $t_a\leq t_y$ (as discussed above).  
In the proof of the above-mentioned inequalities, we shall consider the following three possible cases for the location of $t_y$. 

\textit{\bf Case~$t_a\le t_b\le t_x\le t_y$:} 
Note that this implies $\sigma(ab) \le i_a\le i_b\le i_x\le i_y \le \sigma(xy)$ by~\eqref{def:ia}.
Recalling~\eqref{eq:completeswitch:equiv}, we thus only need to consider $i_a\leq i< \max\{i_x,i_y\}=i_y$ in~\eqref{dZ'}. 
By tracking when the vertices~$A, B, X, Y$ saturate, it follows~that 
\begin{equation*}
    \begin{split}
Z'(\sigma)&= 
\underbrace{\prod_{i_a\le i < i_b}\frac{2}{(\Gamma^*_i-1)(\Gamma^*_i-2)}}_{=:B_1}\cdot 
\underbrace{\prod_{i_b\le i < i_x}\frac{2}{(\Gamma^*_i-2)(\Gamma^*_i-3)}}_{=:C_1} \cdot 
\underbrace{\prod_{i_x\le i < i_y}\frac{2}{(\Gamma^*_i-2)(\Gamma^*_i-3)}}_{=:D_1} ,\\
Z'(\osigma)&=
\underbrace{\prod_{i_a\le i < i_b}\frac{2}{\Gamma^*_i(\Gamma^*_i-1)}}_{=:B_2}\cdot 
\underbrace{\prod_{i_b\le i < i_x}\frac{2}{\Gamma^*_i(\Gamma^*_i-1)}}_{=:C_2}\cdot 
\underbrace{\prod_{i_x\le i < i_y}\frac{2}{(\Gamma^*_i-1)(\Gamma^*_i-2)}}_{=:D_2},
    \end{split}
\end{equation*}
as well as 
\begin{equation*}
    \begin{split}
Z'(\sigma')&=
\underbrace{\prod_{i_a\le i < i_b}\frac{2}{(\Gamma^*_i-1)(\Gamma^*_i-2)}}_{= B_1}\cdot 
\underbrace{\prod_{i_b\le i < i_x}\frac{2}{(\Gamma^*_i-1)(\Gamma^*_i-2)}}_{=:C_3}\cdot 
\underbrace{\prod_{i_x\le i < i_y}\frac{2}{(\Gamma^*_i-2)(\Gamma^*_i-3)}}_{=D_1},\\
Z'(\osigma')&=
\underbrace{\prod_{i_a\le i < i_b}\frac{2}{\Gamma^*_i(\Gamma^*_i-1)}}_{=B_2}\cdot
\underbrace{\prod_{i_b\le i < i_x}\frac{2}{(\Gamma^*_i-1)(\Gamma^*_i-2)}}_{=C_3} \cdot 
\underbrace{\prod_{i_x\le i < i_y}\frac{2}{(\Gamma^*_i-1)(\Gamma^*_i-2)}}_{=D_2} .
    \end{split}
\end{equation*}
Note that $B_1\ge B_2$, $C_3\ge C_2$, $D_1\ge D_2$ and $C_1C_2\ge C_3^2$. 
Using elementary estimates\footnote{In equation~\eqref{eq:case1:bound} we exploit that for all~$\alpha,\beta\ge 1$ have 
$\alpha^2(\beta-1) \ge \beta -1$, which implies~$2(\alpha^2\beta+1) \ge (\alpha^2\beta +1)+(\alpha^2+\beta)=(\alpha^2+1)(\beta+1)$ and thus $(\alpha^2\beta+1)/[\alpha(\beta+1)] \ge (\alpha^2+1)/(2\alpha)$.\label{fn:elementary2}}
we then obtain that 
\begin{equation}\label{eq:case1:bound}
\frac{Z'(\sigma)+Z'(\osigma)}{Z'(\sigma')+Z'(\osigma')}
= \frac{\frac{B_1C_1D_1}{B_2C_2D_2} + 1}{\frac{B_1C_3D_1}{B_2C_2D_2}+\frac{C_3}{C_2}}
\ge \frac{\frac{C_3^2}{C_2^2}\frac{B_1}{B_2}\frac{D_1}{D_2}+1}{\frac{C_3}{C_2}(\frac{B_1}{B_2}\frac{D_1}{D_2}+1)}
\ge\frac{(\frac{C_3}{C_2})^2+1}{2\frac{C_3}{C_2}}\ge 1 ,
\end{equation}
establishing inequality~\eqref{eq:completeswitch:equiv}. 

For part~\ref{completeswitch:2} of Lemma~\ref{completeswitch}, 
note that for $i_b\le i < i_x$ we used ${\Gamma^*_i-2=\Gamma_i(\sigma)}$ above (cf.~$C_1$ above). 
Furthermore, we have ${\Gamma_i(\sigma)} \le {2(m-i) \le 2m}$ by Observation~\ref{obs:unsaturated}. 
Combined with the fact that ${i_x-i_b}\ge {{\min\{\sigma(xy), t_x\}}-{\max\{\sigma(ab), t_b\}}}$ satisfies ${i_x-i_b} \ge {\zeta m/3}$ by the `gap' hypothesis of part~\ref{completeswitch:2}, 
it follows~that 
\begin{equation}\label{eq:case1:bound:better}
\frac{C_3}{C_2}=\prod_{i_b\le i < i_x}\left(1+\frac{2}{\Gamma^*_i-2}\right)\ge \lrpar{1+\frac{1}{m}}^{i_x-i_b} \ge 1+\frac{i_x-i_b}{m} \ge 1+\zeta/3,
\end{equation}
which enables us to improve the final inequality of~\eqref{eq:case1:bound} to $\ge 1+\epsilon'$ for suitable~$\epsilon'=\epsilon'(\zeta)>0$.

\textit{\bf Case $t_a\le t_b\le t_y\le t_x$:} 
Note that this implies $\sigma(ab) \le i_a \le i_b \le i_y \le i_x \le \sigma(xy)$ by~\eqref{def:ia}, 
so this time we only need to consider $i_a\leq i< \max\{i_x,i_y\}=i_x$ in~\eqref{dZ'}. 
Proceeding analogously to the previous case, by tracking when the vertices~$A, B, X, Y$ saturate it follows~that 
\begin{equation*}
    \begin{split}
Z'(\sigma)&= 
\underbrace{\prod_{i_a\le i < i_b}\frac{2}{(\Gamma^*_i-1)(\Gamma^*_i-2)}}_{=:B_1}\cdot 
\underbrace{\prod_{i_b\le i < i_y}\frac{2}{(\Gamma^*_i-2)(\Gamma^*_i-3)}}_{=:C_1} \cdot 
\underbrace{\prod_{i_y\le i < i_x}\frac{2}{(\Gamma^*_i-2)(\Gamma^*_i-3)}}_{=:D_1} , \\
Z'(\osigma)&=
\underbrace{\prod_{i_a\le i < i_b}\frac{2}{\Gamma^*_i(\Gamma^*_i-1)}}_{=:B_2} \cdot 
\underbrace{\prod_{i_b\le i < i_y}\frac{2}{\Gamma^*_i(\Gamma^*_i-1)}}_{=:C_2} \cdot 
\underbrace{\prod_{i_y\le i < i_x}\frac{2}{(\Gamma^*_i-1)(\Gamma^*_i-2)}}_{=:D_2}, 
    \end{split}
\end{equation*}
as well as 
\begin{equation*}
    \begin{split}
Z'(\sigma')&=
\underbrace{\prod_{i_a\le i < i_b}\frac{2}{(\Gamma^*_i-1)(\Gamma^*_i-2)}}_{=B_1}\cdot 
\underbrace{\prod_{i_b\le i < i_y}\frac{2}{(\Gamma^*_i-1)(\Gamma^*_i-2)}}_{=:C_3}\cdot 
\underbrace{\prod_{i_y\le i < i_x}\frac{2}{(\Gamma^*_i-1)(\Gamma^*_i-2)}}_{=D_2}, \\
Z'(\osigma')&=
\underbrace{\prod_{i_a\le i < i_b}\frac{2}{\Gamma^*_i(\Gamma^*_i-1)}}_{=B_2}\cdot  
\underbrace{\prod_{i_b\le i < i_y}\frac{2}{(\Gamma^*_i-1)(\Gamma^*_i-2)}}_{=C_3}\cdot 
\underbrace{\prod_{i_y\le i < i_x}\frac{2}{(\Gamma^*_i-2)(\Gamma^*_i-3)}}_{=D_1} .
    \end{split}
\end{equation*}
Note that $B_1\ge B_2$, $C_1C_2\ge C_3^2$, $C_3\ge C_2$, and $D_1\ge D_2$. 
Similarly to~\eqref{eq:case1:bound}, 
using elementary estimates\footnote{Similarly to Footnote~\ref{fn:elementary2}, in equation~\eqref{eq:case2:bound} we exploit that for all~$\alpha,\beta,\gamma\ge 1$ we have $\alpha^2\beta(\gamma-1) \ge \gamma-1$ and $\alpha^2\gamma(\beta-1) \ge \beta-1$, 
which implies $2(\alpha^2\beta\gamma+1) \ge (\alpha^2\beta + \gamma) + (\alpha^2\gamma+\beta) = (\alpha^2+1)(\beta+\gamma)$ and thus $(\alpha^2\beta\gamma+1)/[\alpha(\beta+\gamma)] \ge (\alpha^2+1)/(2\alpha)$.} 
we then obtain that
\begin{equation}\label{eq:case2:bound}
\frac{Z'(\sigma)+Z'(\osigma)}{Z'(\sigma')+Z'(\osigma')}
\ge \frac{\frac{C_3^2}{C_2^2}\frac{B_1}{B_2}\frac{D_1}{D_2}+1}{\frac{C_3}{C_2}(\frac{B_1}{B_2}+\frac{D_1}{D_2})}\ge \frac{(\frac{C_3}{C_2})^2+1}{2\frac{C_3}{C_2}}\ge 1,
\end{equation}
establishing~\eqref{eq:completeswitch:equiv}. 
For part~\ref{completeswitch:2}, using~$i_y-i_b\ge \zeta m/3$ we obtain analogously to~\eqref{eq:case1:bound:better} that 
\begin{equation*}
\frac{C_3}{C_2}=\prod_{i_b\le i < i_x}\left(1+\frac{2}{\Gamma^*_i-2}\right) \ge \lrpar{1+\frac{1}{m}}^{i_x-i_b} \ge 1+\frac{i_x-i_b}{m} \ge 1+\zeta/3,
\end{equation*}
which again enables us to improve the final inequality of~\eqref{eq:case2:bound} to~$\ge 1+\epsilon'$. 

\textit{\bf Case $t_a\le t_y\le t_b\le t_x$:} 
Note that this implies $\sigma(ab) \le i_a \le i_y \le i_b \le i_x \le \sigma(xy)$ by~\eqref{def:ia}, 
so we again only need to consider $i_a\leq i< \max\{i_x,i_y\}=i_x$ in~\eqref{dZ'}. 
Proceeding analogously to the previous cases, it follows~that 
\begin{equation*}
    \begin{split}
Z'(\sigma)&=
\underbrace{\prod_{i_a\le i < i_y}\frac{2}{(\Gamma^*_i-1)(\Gamma^*_i-2)}}_{=:B_1}\cdot 
\underbrace{\prod_{i_y\le i < i_b}\frac{2}{(\Gamma^*_i-1)(\Gamma^*_i-2)}}_{=:C_1} \cdot 
\underbrace{\prod_{i_b\le i < i_x}\frac{2}{(\Gamma^*_i-2)(\Gamma^*_i-3)}}_{=:D_1} ,\\
Z'(\osigma)&=
\underbrace{\prod_{i_a\le i < i_y}\frac{2}{\Gamma^*_i(\Gamma^*_i-1)}}_{=:B_2}\cdot 
\underbrace{\prod_{i_y\le i < i_b}\frac{2}{(\Gamma^*_i-1)(\Gamma^*_i-2)}}_{=C_1} \cdot 
\underbrace{\prod_{i_b\le i < i_x}\frac{2}{(\Gamma^*_i-1)(\Gamma^*_i-2)}}_{=:D_2} ,
    \end{split}
\end{equation*}
as well as 
\begin{equation*}
    \begin{split}
Z'(\sigma')&=
\underbrace{\prod_{i_a\le i < i_y}\frac{2}{(\Gamma^*_i-1)(\Gamma^*_i-2)}}_{=B_1}\cdot 
\underbrace{\prod_{i_y\le i < i_b}\frac{2}{(\Gamma^*_i-1)(\Gamma^*_i-2)}}_{=C_1} \cdot 
\underbrace{\prod_{i_b\le i < i_x}\frac{2}{(\Gamma^*_i-1)(\Gamma^*_i-2)}}_{=D_2}  \\
Z'(\osigma')&=
\underbrace{\prod_{i_a\le i < i_y}\frac{2}{\Gamma^*_i(\Gamma^*_i-1)}}_{=B_2}\cdot  
\underbrace{\prod_{i_y\le i < i_b}\frac{2}{(\Gamma^*_i-1)(\Gamma^*_i-2)}}_{=C_1} \cdot 
\underbrace{\prod_{i_b\le i < i_x}\frac{2}{(\Gamma^*_i-2)(\Gamma^*_i-3)}}_{=D_1} .
    \end{split}
\end{equation*}
Noting that $B_1\ge B_2$ and $D_1\ge D_2$ imply $(B_1-B_2)(D_1-D_2)\ge 0$, 
it follows that 
\begin{equation*}
\frac{Z'(\sigma)+Z'(\osigma)}{Z'(\sigma')+Z'(\osigma')}
=\frac{B_1D_1+B_2D_2}{ B_1D_2+B_2D_1} \ge 1.
\end{equation*}
For part~\ref{completeswitch:2} there is nothing to show in this case, 
since~$t_y\le t_b$ implies ${\min\{\sigma(xy), t_x, t_y\}}-{\max\{\sigma(ab), t_a, t_b\}}\le 0<{\zeta m/3}$, contradicting the `gap' hypothesis of part~\ref{completeswitch:2}. 
%
%
\end{proof}

\subsection{Pairing of edge-sequences: twins and proof of \refL{twins}}\label{sec:twins}
Expanding on the beginning of \refS{sec:switching}, the majority of this subsection is devoted to the rigorous construction of twins, 
i.e., pairs of edge-sequences ${(\s_1,\s_2) \in (\cP^+ \times \cP^-)\cup (\cP^-\times \cP^+)}$ that 
satisfy~$Z(\s_1)=Z(\s_2)$ and a number of additional technical properties 
that are useful for the proof of \refL{twins}.

Before giving the formal details of our construction, we start with a simple yet illustrative example for the case $\s_1\in \cP^+$ with $t_x(\s_1)<t_b(\s_1)$. 
Suppose that~$B$ and~$X$ are not adjacent, with~$\deg(B)=4$ and~$\deg(X)=6$, 
and that the ten edges incident to $B$ and $X$ appear in the edge-sequence~$\s_1$ in the following~manner:
\[
\mydots, qx_1, \mydots, rb_1, \mydots, ab, \mydots, sx_2, \mydots, tx_3, \mydots, ux_4, \mydots, vx_5, \mydots, xy, \mydots, wb_2, \mydots, zb_3, \mydots 
\]
We then replace the edges $ab, xy$ with $ax, by$ (which can be thought of as swapping the points~$b$~and~$x$), 
and further swap some of the points in~$B$ with some of the points in~$X$ (say, swapping~$b_2$~and~$x_3$, and swapping~$b_3$~and~$x_5$) 
to~obtain the edge-sequence~$\s_2 \in \cP^-$ (clusters are defined to accommodate these operations; in particular, swapping of~$b$ and~$x$ ensures~$\s_2\in \cP^-$):
\[
\mydots, qx_1, \mydots, rb_1, \mydots, ax, \mydots, sx_2, \mydots, tb_2, \mydots, ux_3, \mydots, vb_3, \mydots, by, \mydots, wx_4, \mydots, zx_5, \mydots
\]
There are two important properties to observe in this example:\vspace{-2mm}
\begin{romenumerate}
	\parskip 0em  \partopsep=0pt \parsep 0em \itemsep 0.0675em
	\item Only points of~$B$ and~$X$ move to different positions, 
meaning that any vertex $V\notin \{B,X\}$ becomes saturated during the same step of~$\s_2$ as it does in~$\s_1$.
	\item Inspecting the definition~\eqref{def:se}--\eqref{def:ta}, 
we have~$t_x(\s_2)=t_b(\s_1)$ and $t_b(\s_2)=t_x(\s_1)$, as well as~$\s_2(ax)=\s_1(ab)$ and~$\s_2(by)=\s_1(xy)$. 
This implies that the steps at which~$B$ and~$X$ become saturated are swapped between~$\sigma_1$ and~$\sigma_2$.
\end{romenumerate}	
From these properties, it follows that the number of unsaturated vertices satisfies $\G_i(\sigma_1)=\G_i(\sigma_2)$ for all steps~$1\le i\le m$, which in turn leads to~$Z(\s_1)=Z(\s_2)$. 
In this case, we will call~$\s_2$ the \emph{$BX$-twin} of~$\s_1$, and vice versa; see \refS{sec:bxtwins}.
In general, we will carefully swap some points in~$B$ and~$X$ but not in a $BX$ edge such that the two properties stated above holds 
(to clarify: we will not swap the endpoints of any edges between $B$ and $X$, since that could create loops).

With that example in mind, we now turn to the formal details of our construction, 
where the following auxiliary observation will be used to swap points.
\begin{observation}\label{oswap}
Consider any integers satisfying $0<d_b<d_x\le t<d_x+d_b$. 
Let~$\Theta_1$ denote the set of sequences consisting of $d_x$ many X's and $d_b$ many B's such that:
(i)~the last character is B, (ii)~the  $t^{\rm th}$ character is X, and (iii)~every character after the $t^{\rm th}$ is~B. 
Let~$\Theta_2$ denote the set of sequences consisting of $d_x$ many X's and $d_b$ many B's such that:
(i)~the last character is X, (ii)~the $t^{\rm th}$ character is B, and (iii)~every character after the $t^{\rm th}$ is~X.
Then~$|\Theta_1|\leq |\Theta_2|$.
\end{observation}
\begin{proof} 
The main idea is that restriction~(iii) is stricter for~$\Theta_1$ than for~$\Theta_2$ because of~$d_x>d_b$. 
More formally, by considering the first $t-1$ character in the sequence, it follows that $|\Theta_1|= \tbinom{t-1}{d_x-1}\le \tbinom{t-1}{d_b-1}=|\Theta_2|$, where the inequality holds since~$t\le d_x+d_b-1$.
\end{proof}
Observation~\ref{oswap} implies that there is an injection $I_{t,d_b,d_x}$ mapping $\Theta_1$ into $\Theta_2$. 
There are many possible injections, but we henceforth fix one arbitrarily (and use it in what follows).

\subsubsection{$BX$-twins}\label{sec:bxtwins}
Consider any edge-sequence $\s_1\in\cP^+$ with $t_x(\s_1)<t_b(\s_1)$. 
We will construct its \emph{$BX$-twin} $\s_2\in\cP^-$ by generalizing the example above, 
where we swapped $b$ and $x$, but did not swap points of~$B$ and~$X$ in a $BX$ edge.
To this end we define $\lambda_{BX}$ as the number of edges between~$B$ and~$X$ (which we will call $BX$ edge), 
and introduce the parameters 
\begin{equation}\label{def:degp}
d_b:=\deg(B)-1-\lambda_{BX} \qquad \text{ and } \qquad d_x:=\deg(X)-1-\lambda_{BX}, 
\end{equation}
which corresponds to the number of $B$ and $X$ points that we might swap. 
Using $\deg(B)<\deg(X)$, we infer that~$d_b<d_x$. 
With a view towards applying \refO{oswap}, we consider the $d_b+d_x$ edges in~$\s_1$ 
incident to~$B$ or~$X$ but is not ~$ab$ or~$xy$ and not a~$BX$-edge.
%
By considering these edges in the order in which they appear in $\s_1$, 
we naturally obtain a sequence containing $d_b$ many $B$'s and $d_x$ many X's, where the $i^{\rm th}$ entry is $B$ (resp.~$X$) if the $i^{\rm th}$ edge is has $B$ (resp.~$X$) as an endpoint. 
The resulting sequence~$S_1$ is well-defined,  since none of these edges has both~$B$ and~$X$ as endpoints. 
Since $t_x(\s_1)<t_b(\s_1)$, it follows that the last entry of $S_1$ is a $B$.  
Let $t$ be the position in $S_1$ of the last $X$.  Thus $t\geq d_x$.
So, using the notation from~\refO{oswap}, we have $S_1\in\Theta_1$ and 
we then set~$S_2:=I_{t,d_b,d_x}(S_1)$ using the injection we fixed above. 
(As an illustration, in the example above we have~$d_b=3$ and~$d_x=5$ as well as~$S_1=X,B,X,X,X,X,B,B$ and $S_2=X,B,X,B,X,B,X,X$.)
With~$S_1$ and~$S_2$ in hand, we are now ready to construct~$\s_2$ by modifying~$\s_1$ via the following steps:
\vspace{-0.5em}
\begin{romenumerate}	
\parskip 0em  \partopsep=0pt \parsep 0em \itemsep 0.0675em
\item\label{step1} In $\s_1$, replace~$ab$ with~$ax$ and~$xy$ with~$by$.
\item\label{step2} We temporarily ignore the label of points in $X$. For every edge $e$ counted by~$d_x$, i.e., edges incident to~$X$ that is not $xy$ and not a $BX$ edge, 
let $i_e$ denote the position in $S_1$ corresponding to $e$; thus~$S_1$ has an $X$ in position~$i_e$.  
If~$S_2$ has a $B$ in position~$i_e$, then we modify $e$ by changing endpoint $X$ to $B$; 
keep this edge in the same position in $\s_2$ as $e$ is in~$\s_1$. 
(If~$S_2$ has an $X$ in position~$i_e$, then we do nothing.)
\item\label{step3} We temporarily ignore the label of points in $B$. For every edge $e$ counted by~$d_b$, i.e., edges incident to~$B$ that is not $ab$ and not a $BX$ edge,  
let $i_e$ denote the position in $S_1$ corresponding to $e$; thus~$S_1$ has a $B$ in position~$i_e$.  
If~$S_2$ has a $X$ in position~$i_e$, then we modify $e$ by changing endpoint $B$ to $X$;
keep this edge in the same position in $\s_2$ as $e$ is in~$\s_1$. 
(If~$S_2$ has an $B$ in position~$i_e$, then we do nothing.)
\item\label{step4} Label the points of $B$ not equal to $b$ in $\s_2$ with the same relative ordering as (labeled) points of $B$ not equal to $b$ in $\s_1$.
Similarly, label the points of $X$ not equal to $x$ in $\s_2$ with the same relative ordering as (labeled) points of $X$ not equal to $x$ in $\s_1$.
\end{romenumerate}
The resulting edge-sequence~$\s_2$ is called a $BX$-twin of~$\s_1$, and from the definition of clusters it follows that~${\s_2\in\cP^-}$ 
(because we did not change the set of edges not incident to~$B$ and $X$, and we did not change the union of neighboring points of~$B$ and~$X$). 
The following lemma records some additional~properties. 
\begin{lemma}\label{ltbtx}
The $BX$-twin $\s_2\in\cP^-$ constructed above satisfies the following properties:\vspace{-0.5em}%
\begin{enumerate}[(a)]
	\parskip 0em  \partopsep=0pt \parsep 0em \itemsep 0.0675em
\item\label{16:1}  $Z(\s_1)=Z(\s_2)$.
\item\label{16:2} $\s_2\in\cP^-$ and $t_b(\s_2)<t_x(\s_2)$.
\item\label{16:3} $t_a(\s_2)=t_a(\s_1)$, $t_y(\s_2)=t_y(\s_1)$, $t_b(\s_2)=t_x(\s_1)$, $t_x(\s_2)=t_b(\s_1)$.
\item\label{16:4} Every $\s\in\cP^-$ with $t_b(\s)<t_x(\s)$ has at most one $BX$-twin.
\end{enumerate}
\end{lemma}
\begin{proof}
We first argue that~$t_b(\s_1)=t_x(\s_2)$. 
Since $t_x(\s_1)<t_b(\s_1)$, the last edge in $\s_1$ which has $B$ as an endpoint cannot be a $BX$ edge.  So $t_b(\s_1)$ is the index corresponding to the last element of $S_1$.  That edge is modified to have endpoint $X$ in $\s_2$, and no edges following it in $\s_1$, except possibly $ab$ and $xy$, are modified. Also, no edges following it, except possibly $xy$, have $X$ as an endpoint, since $t_x(\s_1)<t_b(\s_1)$.  So that index becomes $t_x(\s_2)$, establishing~$t_b(\s_1)=t_x(\s_2)$. 

Similar reasoning shows that~$t_b(\s_2)=t_x(\s_1)$.  The two key additional insights are: (i) every character after that point in $S_2$ is an $X$. Thus (ii) any  edges after that point in $\s_2$ that have $B$ as an endpoint must be $BX$ edges or $ab$. The index of the latest such edge that is not $ab$, if any exist, would be $t_x(\s_1)$ and $t_b(\s_2)$.

Now, since the positions of all copies of~$a,y$ are the same in~$\s_2$ as in~$\s_1$, 
it follows that the findings of the first two paragraphs imply part~\eqref{16:3}. 

Part~\eqref{16:3} immediately implies part~\eqref{16:2}.

For part~\eqref{16:1} note that, since $ab,xy$ in $\s_1$ become $ax,by$ in $\s_2$, the step at which $B$ is saturated in $\s_1$ is equal to the step at which $X$ is saturated in $\s_2$, and vice versa. 
Furthermore, every other vertex becomes saturated at the same step of $\s_1$ and $\s_2$ (since they all appear on the same edges and in the same position in both). 
Hence~${\G_i(\s_1)=\G_i(\s_2)}$ for all~$1\le i\le m$, which implies ${Z(\s_1)=Z(\s_2)}$, establishing part~\eqref{16:1}.

Part~\eqref{16:4} follows because $I_{t,d_b,d_x}$ is an injection. 
Indeed, given $\s_2\in\cP^-$ with $t_b(\s_2)<t_x(\s_2)$, we look at the~$d_b$ edges incident to~$B$ and~$d_x$ edges incident to~$X$, other than $ax, by$ and any $BX$ edges. 
This specifies the sequence $S_2\in \Theta_2$ and the latest $B$ in that sequence establishes the value~$t$. 
Now we obtain $S_1=I^{-1}_{t,d_b,d_x}(S_2)$, which then allows us to uniquely reconstruct the $BX$-twin of~$\s_2$, concluding the proof. 
\end{proof}

Next consider any $\s_1\in\cP^-$ with $t_x(\s_1)<t_b(\s_1)$. We can construct a $BX$-twin $\s_2\in\cP^+$ in the same manner, with only the following change: In $\sigma_1$, replace $ax$ with $ab$ and $by$ with $xy$. The obvious analogue of \refL{ltbtx} holds, by the same proof. This will be encompassed by the more general \refL{ltwins} below.

\subsubsection{$AY$-twins and properties of twins}
Consider any $\s_1\in\cP^+\cup \cP^-$ with $t_y(\s_1)<t_a(\s_1)$. 
We construct its {\em $AY$-twin} $\s_2\in\cP^-$ using the same construction as for $BX$-twins, 
where we this time swap the points of $A$ and $Y$ not in an $AY$ edge (instead of the points of $B$ and $X$ not in an $BX$ edge). 
Besides notational changes, this really just corresponds to interchanging the roles of the vertices $B$ and~$A$ as well as the vertices $X$ and $Y$ in the construction from \refS{sec:bxtwins}. 
In particular, we here use a sequence of $d_a:=\deg(A)-1-\lambda_{AY}$ many~$A$'s and $d_y:=\deg(Y)-1-\lambda_{AY}$ many~$Y$'s. 
We replace step~\ref{step1} with the following: (a)~if $\s_1\in\cP^+$, then,  in $\s_1$, replace $ab$ with $by$ and $xy$ with $ax$, and (b)~if $\s_1\in\cP^-$ then, in $\s_1$, replace $ax$ with $xy$ and $by$ with $ab$.
In steps~\ref{step2} and \ref{step3} we use the injection $I_{t,d_a,d_y}$, 
where~$t$ is defined analogously as in the $BX$-twin case  (the assumptions of \refS{oswap} hold, since $\deg(A)<\deg(Y)$ implies $d_a<d_y$), 
and modify step~(iv) in the obvious way. 

We henceforth use the term {\em twins} to refer to both $BX$-twins and $AY$-twins. 
Then the obvious analogue of \refL{ltbtx} holds (by the same proof), 
and we now record these key properties for later reference. 
\begin{lemma}[Key properties of twins]\label{ltwins}
Defining~$\cP:=\cP^-\cup \cP^+$, the following holds for twins:\vspace{-0.5em}%
\begin{enumerate}[(a)]
	\parskip 0em  \partopsep=0pt \parsep 0em \itemsep 0.0675em
\item\label{17:1} If $(\s_1,\s_2)$ are twins, then $Z(\s_1)=Z(\s_2)$.
\item\label{17:2} Every $\s\in\cP$ with $t_x(\s)<t_b(\s)$ has exactly one $BX$-twin, and every $\s\in\cP$ with $t_y(\s)<t_a(\s)$ has exactly one $AY$-twin.
\item\label{17:3} Every $\s\in\cP$ with $t_b(\s)<t_x(\s)$ has at most one $BX$-twin, and every $\s\in\cP$ with $t_a(\s)<t_y(\s)$ has at most one $AY$-twin.
\item\label{17:4} If $\s$ and $\s'$ are counterparts then either (i)~they both do not have $BX$-twins, or (ii)~their $BX$-twins are counterparts of each other.  The same is true for $AY$-twins.
\end{enumerate}
\end{lemma}
\begin{proof} Parts~\eqref{17:1} and~\eqref{17:3} are simply restatements of parts~\eqref{16:1} and~\eqref{16:4}  of \refL{ltbtx} and its analogues. 
Part~\eqref{17:2} holds by construction of twins. 
Part~\eqref{17:4} follows from the fact that, as counterparts, both $\s$ and $\s'$ will have the same values of $d_b, d_x, t_b,t_x$ and the same ordering of edges incident to $B$ and $X$ except $ab, xy, ax, by$. So if one of them is a $BX$-twin then either $t_x<t_b$ or we can reconstruct their $BX$-twin using $I^{-1}_{t,d_b,d_x}$ as in the proof of \refL{ltbtx}\eqref{16:4}. In both cases both $\s$ and $\s'$ are $BX$-twins.
\end{proof}

\subsubsection{Quadruplets}
Consider $\s_1\in\cP^+\cup \cP^-$ with $t_x(\s_1)<t_b(\s_1)$ and $t_y(\s_1)<t_a(\s_1)$. 
Then~$\s_1$ has both a $BX$-twin~$\s_2$ and an  $AY$-twin~$\s_3$. 
By \refL{ltbtx} and its analogues, $t_y(\s_2)<t_a(\s_2)$ and $t_x(\s_3)<t_b(\s_3)$.  
Therefore, $\s_2$ has an $AY$-twin and~$\s_3$ has a $BX$-twin.  
In the proof of~\refL{twins}, the following result will eventually allow us to uniquely pair up edge-sequences with their twins 
(this is not obvious, because there could exist edge-sequences $\s_1, \s_2$ and $\s_3$ that have two twins). 
\begin{lemma}\label{lquads}
The  $AY$-twin of $\s_2$ is equal to the $BX$-twin of $\s_3$.
\end{lemma}
\begin{proof} When we construct a $BX$-twin, we do not alter any endpoints that are $A$ or $Y$, and the edges containing those endpoints remain in the same position in the edge-sequences.  It follows that when we construct the $AY$-twin of $\s_2$ and the $AY$-twin of $\s_1$, we alter the exact same sequence $S_1$ of $A$'s and $Y$'s and, since we always use the same injection $I_{t,d_a,d_y}$, we alter it in the same way.  The only difference in the two constructions is that, for  one twin we change $(ab,xy)$ to $(by,ax)$ and for the other we change $(ax,by)$ to $(xy,ab)$. 
Similarly for when we construct the $BX$-twins of $\s_1$ and $\s_3$.

Thus, the sequence of two modifications:  $\s_1\rightarrow \s_2\rightarrow AY$-twin of $\s_2$ makes the same alterations as the sequence $\s_1\rightarrow \s_3\rightarrow BX$-twin of $\s_3$, with the following exception: In the first sequence, if $\s_1\in\cP^+$, we have $(ab,xy)\rightarrow (ax,by)\rightarrow (xy,ab)$, and in the other sequence, $(ab,xy)\rightarrow (by,ax)\rightarrow (xy,ab)$.  If $\s_1\in\cP^-$ then we have $(ax,by)\rightarrow (ab,xy)\rightarrow (by,ax)$ and $(ax,by)\rightarrow (xy,ab)\rightarrow (by,ax)$.  So in both sequences, we end up with the same pair, concluding the proof. 
\end{proof}

We henceforth refer to  $\s_1, \s_2, \s_3$ and the $AY$-twin of $\s_2$ (i.e., the $BX$-twin of $\s_3$) as {\em quadruplets}.

\subsubsection{Summing over twins and quadruplets: proof of of \refL{twins}}\label{sec:summingtwins}
In the following proof of \refL{twins} we construct the desired pairs of edge-sequences~${(\sigma_2,\sigma_1)}$, 
defining~$\cT{\subseteq\cP^+\cup\cP^-}$ as the set of all edge-sequences~$\s$ which have a twin.
\begin{proof}[Proof of \refL{twins}]
The goal of part~\ref{lem:twins:sum} is to show that
\begin{equation}\label{eq:lemmatwins}
\sum_{\sigma\in\cT\cap\cP^+}Z(\s)= \sum_{\sigma\in\cT\cap\cP^-}Z(\s).
\end{equation}
To prove this equation, we apply Lemmas~\ref{ltwins} and~\ref{lquads} to pair terms from the left-hand side of~\eqref{eq:lemmatwins} with terms from the right-hand side of~\eqref{eq:lemmatwins}. 
More precisely, for each set of quadruplets, we arbitrarily choose one of the two ways to split the four into two pairs, where each pair~${(\sigma_2,\sigma_1) \in (\cP^+ \times \cP^-)\cup (\cP^-\times \cP^+)}$ is a set of twins.
For each remaining edge-sequence~$\sigma_2 \in \cP^+$ that has a twin (and is not part of a quadruplet), 
we then form the pair~$(\sigma_2,\sigma_1)$, where~$\sigma_1 \in \cP^-$ is the unique twin of~$\sigma_2$. 
By \refL{ltwins}\eqref{17:1}, all paired edge-sequences each have the same $Z$-value. 
So in~\eqref{eq:lemmatwins} all terms cancel out, establishing that the left-hand side of~\eqref{eq:lemmatwins} is equal to the right-hand side of~\eqref{eq:lemmatwins}, concluding the proof of part~\ref{lem:twins:sum}. 

Part~\ref{lem:twins:equiv} follows from \refL{ltwins}\eqref{17:4} and the fact that the position of the edges $ab,xy,ax,by$ in $\s$ plays no role in the construction of twins.  

Part~\ref{lem:twins:techn} follows from \refL{ltwins}\eqref{17:2}.  

For part \ref{lem:twins:techn2}, for the sake of contradiction suppose $\s_1\in \cP$ is a $BX$-twin of $\sigma$. Then $t_x(\s_1)<t_b(\s_1)$ and no points of~$X,Y$ are adjacent to any points of~$A,B$ in~$\s_1$. Then by \refO{oswap} (ii), $\sigma$ satisfies the following: there is some $t\ge d_x>d_b$ such that the $t^{\rm th}$ edge in the sequence of edges incident to~$B$ and~$X$ is incident to~$B$.
In particular we see that not all edges incident to~$X$ appear after all edges incident to~$B$, a contradiction. 
Therefore~$\sigma$ is not a $BX$-twin, and similarly $\sigma$ is not an $AY$-twin, which completes the~proof.
\end{proof}

\section{Counting results: deferred proofs}\label{sec:counting}
This section is devoted to the deferred proofs of the counting results \refL{coefficients} and \ref{goodpermutation} from \refS{sec:switchingcounting}, 
which give bounds on the number of clusters containing a given configuration-graph and the number of good edge-sequences. 
Both proofs are based on counting arguments and the following two observations.  
For brevity, we henceforth always tacitly assume that~$\gamma>0$ is sufficiently small 
and that~$n \ge n_0(\gamma,\Delta,\xi)$ is sufficiently~large (whenever~necessary). 
\begin{observation}\label{technical}
There are constants $c=c(\xi,\Delta) > 1/2$ and $\nu=\nu(\xi,\Delta)>0$ 
such that~${m \ge \max\bigcpar{c \stubs,\mu}}$ and~${\mu+m-\stubs \ge \nu \max\{m,\mu\}}$. 
\end{observation}
\begin{proof}
Since by assumption~$\sum_{k < j \le \Delta}n_i = n - \sum_{j \in [k]} n_j\ge \xi n$, it follows that $m \ge c \stubs$ for some constant ${c=c(\xi,\Delta) > 1/2}$. 
Hence
\begin{equation*}
        \mu+m-{\textstyle\stubs} 
				= \left(1-\frac{\stubs}{2m}\right)^2 \cdot m \ge \left(1-\frac{1}{2c}\right)^2 \cdot m	,
\end{equation*}
which completes the proof  by noting that~$m^2 \ge (\stubs)^2/4$ implies~$m \ge \mu$.
\end{proof}
%
Recall that an edge is {\em small} if both endpoints have degree at most~$k$ (see~\refS{sec:mainresult}).  
We say an edge is {\em large} if both endpoints have degree greater than~$k$, 
and {\em mixed} if one endpoint's degree is at most~$k$ and the other is greater than~$k$. 
With these definitions in hand, we now record the following consequence of our notational convention from Section~\ref{sec:switchingcounting} (see page~\pageref{notationalconvention}). 

\begin{observation}\label{obs:convention}
For all configuration-graphs~$G$ the following holds:
\vspace{-0.5em}%
\begin{romenumerate}
\parskip 0em  \partopsep=0pt \parsep 0em \itemsep 0.0675em
\item\label{obs:convention:UG}
The number~$U_G$ of upper clusters containing~$G$ equals four times the number of pairs consisting of small and large edges in~$G$.
\item\label{obs:convention:LG}
The number~$L_G$ of lower clusters containing~$G$ equals twice the number of pairs consisting of mixed edges in~$G$ whose endpoints are contained in four distinct~vertices.
\end{romenumerate}
\end{observation}
\begin{proof}
For part~\ref{obs:convention:UG}, note that for every small edge~$e$ and large edge~$e'$ in a graph $G$, there are four upper clusters of the form~$C^+(G,e,e')$. For part~\ref{obs:convention:LG}, note that for every pair of edges~$e,e'$ in a graph~$G$, each with one endpoint of degree at most $k$ and another of degree greater than $k$, there are two lower clusters of the form~$C^-(G,e,e')$ 
if the endpoints of $e$ and $e'$  are contained in four distinct vertices (note that the distinct vertices condition is baked into the definition of clusters, as we insist on $A,B,X,Y$ being distinct~vertices). 
\end{proof}

\subsection{Counting clusters: proof of \refL{coefficients}}\label{sec:coefficients}
Recall that~$U_G$ and~$L_G$ count the number of upper and lower clusters that contain the configuration-graph~$G$ (see~\refS{sec:switchingcounting}).   
%
%
To prove \refL{coefficients}, we shall combine Observation~\ref{obs:convention} with counting arguments to show that~$L_G$ and~$U_H$ are approximately equal when~$G$ and~$H$ contain similar numbers of small edges.
%
%
\begin{proof}[Proof of \refL{coefficients}]
We first estimate~$L_G$ and $U_G$ for any configuration-graph~$G$ with $\ell=\Xs(G)$ small edges. 
Recall that an edge is called mixed if one endpoint's degree is at most~$k$ and the other is greater than~$k$. 
By Observation~\ref{obs:convention}, the number~$L_G$ of lower clusters containing $G$ equals to twice the number of pairs~$(e,e')$ of mixed edges in~$G$ 
where the endpoints of $e,e'$ lie on four different vertices. 
Since there are $2 \cdot \Xs(G) = 2\ell$ small points whose neighbor is also small, we infer that there are~$\stubs-2\ell$ mixed edges. It follows~that 
\begin{equation}\label{eq:LG:1}
\binom{\stubs-2\ell}{2}-\bigpar{{\textstyle\stubs}-2\ell} \cdot 2\Delta \: \le \: \frac{L_G}{2} \: \le \: \binom{\stubs-2\ell}{2} .
\end{equation}
By Observation~\ref{obs:convention}, the number~$U_G$ of upper clusters containing~$G$ equals to four times the number of pairs of edges~$(e,e')$ where $e$ is small and $e'$ is large in~$G$.
Since there are $\ell$ small edges and ${m-\ell-(\stubs-2\ell)}={m+\ell-\stubs}$ large edges, it follows~that 
\begin{equation}\label{eq:UG}
U_G
=4\ell\bigpar{\ell+m-{\textstyle\stubs}}. 
\end{equation}

We next simplify inequality~\eqref{eq:LG:1} for~$L_G$ when~$G$ has~$\ell\le {(1+\gamma)\mu}$ small edges.
Using \refO{technical} we infer~${2\ell} \le {(1+\gamma)\stubs/(2c)}$ for some~$c>1/2$, so that~${\stubs-2\ell} = \Omega(\stubs) = \omega(1)$ for all sufficiently small~$\gamma>0$.
Hence inequality~\eqref{eq:LG:1} implies
\begin{equation}\label{eq:LG}
L_G = (1+o(1)) \cdot \bigpar{{\textstyle\stubs}-2\ell}^2 = \omega(1).
\end{equation}

Finally, combining~\eqref{eq:UG}--\eqref{eq:LG} with~$4m\mu=(\stubs)^2$ and \refO{technical}, 
for all sufficiently small~$\gamma>0$ it follows that for all configuration-graphs~$G, H\in \cN_{\gamma\mu}$ we have  
\begin{equation}\label{LUcoefficient}
    \begin{split}
        \frac{L_G}{U_H}&
				= \frac{(1+o(1)) \cdot \bigpar{\stubs- 2 [1+O(\gamma)]\mu}^2}{(1+O(\gamma)) \cdot  4\mu\bigpar{[1+O(\gamma)]\mu+m-\stubs}}\\
				&=(1+O(\gamma)) \cdot \biggpar{1+\frac{(\stubs)^2-4m\mu + O\bigpar{\gamma\mu\bigpar{\stubs+\mu}}}{4\mu\bigpar{[1+O(\gamma)]\mu+m-\stubs}}}\\
				&= (1+O(\gamma)) \cdot \biggpar{1+\frac{O\bigpar{\gamma m}}{\nu \max\{m,\mu\} - O(\gamma \mu)}} =1+O(\gamma) ,
    \end{split}
\end{equation}
which together with~\eqref{eq:LG} also implies $U_H =\omega(1)$, completing the proof. 
\end{proof}

\subsection{Counting good edge-sequences: proof of \refL{goodpermutation}}\label{sec:goodpermutation}
Recall that in any upper cluster~$C^+=C^+(G,ab,xy)$ the edge~$ab$ is small and the edge~$xy$ is large (see \refS{sec:clusterdefinition}). 
By Observation~\ref{obs:convention} we know that $U_G$ equals four times the number of pairs consisting of small and large edges, 
and so it follows (with room to spare) 
that the desired estimate~\eqref{eq:goodpermutation} of \refL{goodpermutation} would be implied~by
\begin{equation}\label{eq:goodpermutation:goal}
\Lambda_{G,\sigma} :=\frac{|\{ab,xy\in E(G)\: : \: ab \text{ small, } xy \text{ large, } (\sigma, ab,xy)\in \mathbb{G}\}|}{|\{ab,xy\in E(G)\: :\: ab \text{ small, } xy \text{ large}\}|}  \; \ge \; \frac{1}{2} .
\end{equation}
Intuitively \eqref{eq:goodpermutation:goal} says that at least half of the choices of small edges~$ab \in E(G)$ and large edges~$xy \in E(G)$ yield a good edge-sequence~$(\sigma,ab,xy)$. 
To prove this estimate, 
we shall first use counting arguments to show that many small edges have endpoints lying in vertices that do not become saturated during the last~$\zeta m$ steps of the edge-sequence $\s$. 
Then, for each such small edge~$ab$, we shall use counting arguments to show that almost all large edges~$xy$ yield a good edge-sequence~$(\sigma, ab, xy) \in \mathbb{G}$, i.e., $X$ and~$Y$ are not adjacent to~$A$ and~$B$ (recall that $A, B, X, Y$ denote the vertices that contain the points $a,b,x,y$). 
As we shall see, this together implies that many edge-sequences are~good, 
which eventually establishes~\eqref{eq:goodpermutation:goal} and thus~\eqref{eq:goodpermutation}. 
%
\begin{proof}[Proof of \refL{goodpermutation}]
We henceforth fix a configuration-graph~$G \in \cN_{\gamma\mu}$ and an edge-sequence $\sigma \in \Pi_G$ of the edges of~$E(G)$. 
A small edge $ab\in E(G)$ is called early if both vertices~$A,B$ 
saturate during the first~$(1-\zeta)m$ steps.
Inspecting~$\Lambda_{G,\sigma}$ from~\eqref{eq:goodpermutation:goal}, 
by definition of good edge-sequences~$\bG$ (see~\refS{sec:switchingcounting}) we~have 
\begin{equation}\label{eq:goodpermutation:first}
\begin{split} 
\Lambda_{G,\sigma} \; \ge \; 
\frac{|\{ab \in E(G) \: : \: \text{$ab$ small, } \text{$ab$ early}\}|}{|\{ab \in E(G) \: : \: \text{$ab$ small}\}|} 
\cdot 
\min_{\substack{ab\in E(G):\\\text{$ab$ small,}\\ \text{$ab$ early}}}\frac{|\{xy \in E(G) \: : \: \text{$xy$ large, } (\sigma, ab,xy)\in \mathbb{G}\}|}{|\{xy \in E(G) \: : \: \text{$xy$ large}\}|}.
\end{split}
\end{equation}

It remains to bound the two fractions in the right-hand side of~\eqref{eq:goodpermutation:first} from below, 
and we start with the first fraction. 
As before, we write~$\ell=\Xs(G)$ for the number of small edges of~$G$, where~$\ell \ge (1-\gamma)\mu$ by assumption. 
During the final~$\zeta m$ steps of the process, at most $2 \cdot \zeta m \le \zeta n\Delta$ small vertices become saturated. 
Noting that $2m \le n\Delta$ and $\gamma \le 1/2$ imply $\ell \ge (1-\gamma)\mu\ge (\stubs)^2/8m 
\ge \xi^2 n/(4\Delta)$, 
using the definition~\eqref{def:zeta} of $\zeta = \xi^2/(16\Delta^3)$  
it follows that 
\begin{equation}\label{eq:goodpermutation:ell}
\frac{|\{ab \in E(G) \: : \: \text{$ab$ small, } \text{$ab$ early}\}|}{|\{ab \in E(G) \: : \: \text{$ab$ small}\}|} 
\ge \frac{\ell - \zeta n\Delta \cdot \Delta}{\ell} 
= 1 - \frac{\xi^2n}{16 \Delta\ell} \ge \frac{3}{4}.
\end{equation}

Turning to the second fraction in the right-hand side of~\eqref{eq:goodpermutation:first}, let $ab\in E(G)$ be a small and early edge. 
There are~$\ell+m-\stubs$ large edges~$xy \in E(G)$, as established above~\eqref{eq:UG}. 
Furthermore, there are at most $2k \cdot \Delta \le 2\Delta^2$ edges~$xy \in E(G)$ with the property that at least one point of~$X,Y$ is adjacent to a point in~$A,B$ (recall that~$A,B,X,Y$ denote the vertices containing the points~$a,b,x,y$).
For sufficiently small~$\gamma>0$ \refO{technical} implies that~$\mu+m-\stubs\ge 2\gamma\mu$, 
so using~$\ell \ge (1-\gamma)\mu$ and~$\mu= \Theta(n)$ it follows~that 
\begin{equation*}
\frac{|\{xy \in E(G) \: : \: \text{$xy$ large, } (\sigma, ab,xy)\in \mathbb{G}\}|}{|\{xy \in E(G) \: : \: \text{$xy$ large}\}|}
\ge \frac{\ell+m-\stubs-2\Delta^2}{\ell+m-\stubs} 
\ge 1-\frac{2\Delta^2}{\gamma\mu}\ge \frac{2}{3} ,
\end{equation*}
which together with inequalities~\eqref{eq:goodpermutation:first}--\eqref{eq:goodpermutation:ell} establishes~\eqref{eq:goodpermutation:goal} and thus the desired estimate~\eqref{eq:goodpermutation}. 
\end{proof}

\section{Epilogue: alternative approach}\label{sec:diffeq}
In this final section we briefly discuss an alternative approach to Theorems~\ref{final} and~\ref{transferred}, 
which instead of switching is based on the differential equation method~\cite{DEM,DEM99,Warnke}. 
Our primary interest is the intriguing possibility that this 
approach can be made rigorous, 
which would lead to a conceptually simpler proof of our main~result.

The basic idea is straightforward: 
standard combinatorial methods (see \refApp{sec:uniformtransfer}) show that in the uniform model the number of small edges~$\Xs(\Gdn)$ is concentrated around~${\mu=(\stubs)^2/(4m)}$, 
and a routine application of the differential equation method shows that in the standard $\bdn$-process the number of small edges~$\Xs(\Gpdn)$ is concentrated around~$\rho_k n$, where the quantity~$\rho_k$ has an analytic description in terms of differential equations. 
To prove a discrepancy it thus suffices to show that the two rescaled means~$\mu/n$ and~$\rho_k$ differ in the limit as~$n \to \infty$.  
While this discrepancy is easy to verify for any concrete degree sequence~$\bdn$ covered by \refT{transferred} 
(by numerically solving the relevant differential equations), 
we are lacking an analytic proof technique that can prove this for all such~$\bdn$ based only on the relevant differential equations.

To stimulate further research into such analytic proof techniques 
(which would most likely enhance the power of the differential equation~method), 
as a proof of concept we now outline an analytic 
argument that works for degree sequences~$\bdn$ that are sufficiently irregular, 
i.e., under much stronger assumptions than \refT{transferred}. 
To this end we shall assume that in~$\bdn$ the number of vertices of degree~$j$ satisfies~$n_j/n\to r_j$ as~${n\to\infty}$ for all $j\in [\Delta]$, 
which in view of the (standard) combinatorial 
part~\ref{transferred:uniform} of~\refT{transferred} implies that in the uniform model we typically~have 
\[
\frac{\Xs(\Gdn)}{n} \sim \frac{\mu}{n} 
\sim \frac{(\sum_{j \in [k]}j r_j)^2}{2\sum_{j \in [\Delta]}j r_j} =: \hrho_k
\qquad \text{ and } \qquad 
\frac{m}{n} 
 \sim \frac{1}{2} \sum_{j \in [\Delta]} jr_j =: T ,
\]
where $a_n \sim b_n$ is the usual shorthand for $a_n=(1+o(1))b_n$ as~$n \to \infty$.
To analyze the number of small edges~$\Xs(\Gpdn)$ in the standard $\bdn$-process, we shall track the following random variables after each step~$i$: 
$\Ud(i)$ counts the number of unsaturated vertices, 
$\Ud_a(i)$ with~$a \in [\Delta]$ counts the number of unsaturated vertices~$v$ with degree~$d_v^{(n)}=a$ in~$\bdn$, 
and $\Xd_{a,b}(i)$ with $a,b \in [\Delta]$ counts the number of edges whose two endvertices have degrees~$a$ and~$b$ in~$\bdn$. 
A standard application of the differential equation method shows that typically the following holds: 
for all steps~${0 \le i \le m}$ we have ${\Ud(i)/n = u(t)+o(1)}$, ${\Ud_j(i)/n = u_j(t)+o(1)}$, ${\Xd_{a,b}(i)/n = x_{a,b}(t)+o(1)}$ with~$t=i/n$, 
where the functions~$u(t)$, $u_j(t)$ and $x_{a,b}(t)$ are the solutions to suitable differential equations that depend only on~$r_1, \ldots, r_{\Delta}$. 
In the standard $\bdn$-process we thus typically~have 
\[
\frac{\Xs(\Gpdn)}{n} = \sum_{1 \le a \le b \le k} \frac{X_{a,b}(m)}{n} = \rho_k + o(1) \qquad \text{ for } \qquad \rho_k:=\sum_{1 \le a \le b \le k} x_{a,b}(T)  .
\]
To prove a discrepancy in the number of small edges in both models, 
by our above discussion it thus remains to analytically show that~$\rho_k \neq \hrho_k$  
(which we know to be true by our switching based \refT{transferred}). 
The key to proving this discrepancy is to prove suitable bounds on the values of $x_{a,b}(T)$. 
We know how to do this when the degree sequence~$\bdn$ is sufficiently far from regular, in a sense specified below.  
But new ideas are needed to do so for all the degree sequences~$\bdn$ covered by \refT{transferred}. 
We close this paper with further elaboration of this analytic approach, and a description of how it works on some degree sequences~$\bdn$. 
The relevant differential equations satisfy $u_j(0)=r_j$ and $u_j'(t) \ge -2u_j(t)/u(t)$, 
from which it follows\footnote{With more work one can show that~$u_j(t) =r_j e^{-\lambda(t)} \sum_{0 \le \ell < j} \lambda^\ell(t)/\ell!$ and~$u(t)=\sum_{j \in [\Delta]}u_j(t)$ hold.} that 
\[
u_j(t) \ge r_j e^{-\lambda(t)} \qquad \text{ with } \qquad \lambda(t):=\int_0^t \frac{2}{\ud(x)} \dd x .
\]
%
The simple degree-counting argument underlying Observation~\ref{obs:unsaturated} shows that~$U(i) \ge 2(m-i)/\Delta$, 
and one can similarly show that~$u(t) \ge 2(T-t)/\Delta$, 
which by integration implies~$\lambda(t) \le \Delta \ln\bigpar{T/(T-t)}$. 
Inserting these estimates and~$u(t) \le 1$ into the relevant differential equations $x_{a,b}(0)=0$ and $\xd_{a,b}'(t)=(1+\indic{a \neq b})u_a(t)u_b(t)/\ud^2(t)$, 
we obtain~that 
\[ 
\xd_{a,b}(T)=\int_0^T \xd_{a,b}'(t) \dd t \ge (1+\indic{a\neq b}) r_ar_b \int_0^T \left(\frac{T-t}{T}\right)^{2\Delta} \dd t =  (1+\indic{a\neq b}) r_ar_b\frac{T}{2\Delta+1} ,
\]
which in turn implies that, say, 
\[
\rho_k =\sum_{1 \le a \le b \le k} x_{a,b}(T) \ge \frac{T ( \sum_{j \in [k]} r_j)^2}{2\Delta+1}  \ge \frac{T (\sum_{j \in [k]}j r_j)^2}{(2\Delta+1)k^2}.
\]
Comparing~$\hrho_k$ with~$\rho_k$, we see that the desired discrepancy in the number of small edges is 
implied by ${T/(2\Delta+1) > k^2/(4T)}$, or equivalently ${\sum_{j \in [\Delta]} jr_j > k \sqrt{2\Delta+1}}$.
While this sufficient condition is weaker than the one from \refT{transferred}, it does apply to sufficiently irregular degree sequences~$\bdn$. 
For example, this condition establishes non-contiguity for~$\Delta \ge 7$ 
when half of the vertices of~$\bdn$ have degree one and half have degree~$\Delta$  
(i.e., where~$r_1=r_{\Delta}=1/2$ and~$k=1$). 
We leave it as an open problem to find an improved analytic argument that 
recovers Theorems~\ref{final} and~\ref{transferred} based on analysis of the relevant differential~equations.

\bigskip{\noindent\bf Acknowledgements.} 
We are grateful to the referees for a careful reading and helpful suggestions. 
Lutz Warnke would also like to thank Nick Wormald for many stimulating discussions about the behavior of the degree-restricted random $\bdn$-process.

\begin{appendix}
\section{Appendix}

\subsection{Counterexample to simple switching heuristic}\label{counterexample}
In Sections~\ref{ssps} and~\ref{sec:clustermotivation} we discussed a simple switching heuristic which suggested that, 
when the degrees of a (configuration) graph~$G^+$ satisfy $\max\{\deg(a),\deg(b)\}<\min\{\deg(x),\deg(y)\}$, 
then the switching operation which replaces the edges~$ab, xy$ with~$ax, by$ 
should decrease the probability of the resulting (configuration) graph~$G^-$, 
i.e., $G^+$ should be more likely than~$G^-$.
\begin{figure}
	\centering%
\begin{tikzpicture}
\begin{pgfonlayer}{nodelayer}
\node [style=none] (4) at (-2.875, 1) {};
\node [style=none] (5) at (-3.875, 1) {};
\node [style=black] (6) at (-3.625, 1) {};
\node [style=black] (7) at (-3.125, 1) {};
\node [style=none] (12) at (-1.375, 1) {};
\node [style=none] (13) at (-2.175, 1) {};
\node [style=black] (14) at (-1.975, 1) {};
\node [style=black] (15) at (-1.775, 1) {};
\node [style=black] (16) at (-1.575, 1) {};
\node [style=none] (17) at (-4.625, 1) {};
\node [style=none] (18) at (-5.625, 1) {};
\node [style=black] (19) at (-5.375, 1) {};
\node [style=black] (20) at (-4.875, 1) {};
\node [style=none] (21) at (-4.625, -0.98) {};
\node [style=none] (22) at (-5.625, -0.98) {};
\node [style=black] (23) at (-5.375, -0.98) {};
\node [style=black] (24) at (-4.875, -0.98) {};
\node [style=none] (25) at (-2.775, -1) {};
\node [style=none] (26) at (-3.975, -1) {};
\node [style=black] (27) at (-3.775, -1) {};
\node [style=black] (28) at (-3.525, -1) {};
\node [style=black] (29) at (-2.975, -1) {};
\node [style=none] (30) at (-2.125, 0.5) {};
\node [style=none] (31) at (-3.125, 0.5) {};
\node [style=black] (32) at (-2.8, 0.5) {};
\node [style=black] (33) at (-2.4, 0.5) {};
\node [style=none] (34) at (-2.125, -0.5) {};
\node [style=none] (35) at (-3.125, -0.5) {};
\node [style=black] (36) at (-2.875, -0.5) {};
\node [style=black] (37) at (-2.375, -0.5) {};
\node [style=none] (38) at (-2.175, -1) {};
\node [style=none] (39) at (-1.375, -1) {};
\node [style=black] (40) at (-1.575, -1) {};
\node [style=black] (41) at (-1.775, -1) {};
\node [style=black] (42) at (-1.975, -1) {};
\node [style=black] (43) at (-3.225, -1) {};
\node [style=none] (44) at (-5.375, 1.4) {$a$};
\node [style=none] (45) at (-5.375, -1.4) {$b$};
\node [style=none] (46) at (-1.525, 1.4) {$x$};
\node [style=none] (47) at (-1.525, -1.4) {$y$};
\node [style=none] (48) at (3.625, 1) {};
\node [style=none] (49) at (2.625, 1) {};
\node [style=black] (50) at (2.875, 1) {};
\node [style=black] (51) at (3.375, 1) {};
\node [style=none] (52) at (5.125, 1) {};
\node [style=none] (53) at (4.325, 1) {};
\node [style=black] (54) at (4.525, 1) {};
\node [style=black] (55) at (4.725, 1) {};
\node [style=black] (56) at (4.925, 1) {};
\node [style=none] (57) at (1.875, 1) {};
\node [style=none] (58) at (0.875, 1) {};
\node [style=black] (59) at (1.125, 1) {};
\node [style=black] (60) at (1.625, 1) {};
\node [style=none] (61) at (1.875, -0.98) {};
\node [style=none] (62) at (0.875, -0.98) {};
\node [style=black] (63) at (1.125, -0.98) {};
\node [style=black] (64) at (1.625, -0.98) {};
\node [style=none] (65) at (3.725, -1) {};
\node [style=none] (66) at (2.525, -1) {};
\node [style=black] (67) at (2.725, -1) {};
\node [style=black] (68) at (2.975, -1) {};
\node [style=black] (69) at (3.525, -1) {};
\node [style=none] (70) at (4.375, 0.5) {};
\node [style=none] (71) at (3.375, 0.5) {};
\node [style=black] (72) at (3.7, 0.5) {};
\node [style=black] (73) at (4.1, 0.5) {};
\node [style=none] (74) at (4.375, -0.5) {};
\node [style=none] (75) at (3.375, -0.5) {};
\node [style=black] (76) at (3.625, -0.5) {};
\node [style=black] (77) at (4.125, -0.5) {};
\node [style=none] (78) at (4.325, -1) {};
\node [style=none] (79) at (5.125, -1) {};
\node [style=black] (80) at (4.925, -1) {};
\node [style=black] (81) at (4.725, -1) {};
\node [style=black] (82) at (4.525, -1) {};
\node [style=black] (83) at (3.275, -1) {};
\node [style=none] (84) at (1.125, 1.4) {$a$};
\node [style=none] (85) at (1.125, -1.4) {$b$};
\node [style=none] (86) at (4.975, 1.4) {$x$};
\node [style=none] (87) at (4.975, -1.4) {$y$};
\node [style=none] (88) at (5.575, 1) {$G^-$};
\node [style=none] (89) at (-6.075, 1) {$G^+$};
\end{pgfonlayer}
\begin{pgfonlayer}{edgelayer}
\draw [bend left=90, looseness=0.75] (4.center) to (5.center);
\draw [bend right=90, looseness=0.75] (4.center) to (5.center);
\draw [bend left=90, looseness=0.75] (12.center) to (13.center);
\draw [bend right=90, looseness=0.75] (12.center) to (13.center);
\draw [bend left=90, looseness=0.75] (17.center) to (18.center);
\draw [bend right=90, looseness=0.75] (17.center) to (18.center);
\draw [bend left=90, looseness=0.75] (21.center) to (22.center);
\draw [bend right=90, looseness=0.75] (21.center) to (22.center);
\draw [bend left=90, looseness=0.75] (25.center) to (26.center);
\draw [bend right=90, looseness=0.75] (25.center) to (26.center);
\draw [bend left=90, looseness=0.75] (30.center) to (31.center);
\draw [bend right=90, looseness=0.75] (30.center) to (31.center);
\draw [bend left=90, looseness=0.75] (34.center) to (35.center);
\draw [bend right=90, looseness=0.75] (34.center) to (35.center);
\draw [bend right=90, looseness=0.75] (38.center) to (39.center);
\draw [bend left=90, looseness=0.75] (38.center) to (39.center);
\draw (20) to (6);
\draw[red,densely dashed] (19) to (23);
\draw (24) to (27);
\draw (32) to (28);
\draw (33) to (15);
\draw (14) to (7);
\draw[red,densely dashed] (16) to (40);
\draw (42) to (29);
\draw (41) to (37);
\draw (43) to (36);
\draw [bend left=90, looseness=0.75] (48.center) to (49.center);
\draw [bend right=90, looseness=0.75] (48.center) to (49.center);
\draw [bend left=90, looseness=0.75] (52.center) to (53.center);
\draw [bend right=90, looseness=0.75] (52.center) to (53.center);
\draw [bend left=90, looseness=0.75] (57.center) to (58.center);
\draw [bend right=90, looseness=0.75] (57.center) to (58.center);
\draw [bend left=90, looseness=0.75] (61.center) to (62.center);
\draw [bend right=90, looseness=0.75] (61.center) to (62.center);
\draw [bend left=90, looseness=0.75] (65.center) to (66.center);
\draw [bend right=90, looseness=0.75] (65.center) to (66.center);
\draw [bend left=90, looseness=0.75] (70.center) to (71.center);
\draw [bend right=90, looseness=0.75] (70.center) to (71.center);
\draw [bend left=90, looseness=0.75] (74.center) to (75.center);
\draw [bend right=90, looseness=0.75] (74.center) to (75.center);
\draw [bend right=90, looseness=0.75] (78.center) to (79.center);
\draw [bend left=90, looseness=0.75] (78.center) to (79.center);
\draw (60) to (50);
\draw (64) to (67);
\draw (72) to (68);
\draw (73) to (55);
\draw (54) to (51);
\draw (82) to (69);
\draw (81) to (77);
\draw (83) to (76);
\draw [blue,bend left,densely dashed] (59) to (56);
\draw [blue,bend right,densely dashed] (63) to (80);
\end{pgfonlayer}
\end{tikzpicture}%
\caption{Counterexample to the simple switching heuristic from Sections~\ref{ssps} and~\ref{sec:clustermotivation} 
(we depict configuration-graphs as defined in \refS{sec:configurationgraph}, 
where each vertex~$v$ is represented by an ellipse containing~$\deg(v)$ many points): 
the graph~$G^-$ is obtained from~$G^+$ by the natural switching operation which replaces the edges~$ab, xy$ with~$ax, by$, 
but despite the degree difference $\max\{\deg(a),\deg(b)\}<\min\{\deg(x),\deg(y)\}$ the resulting graph~$G^-$ is more likely than~$G^+$; 
see~\refApp{counterexample} for more~details.\label{fig:counter}}%
\end{figure}
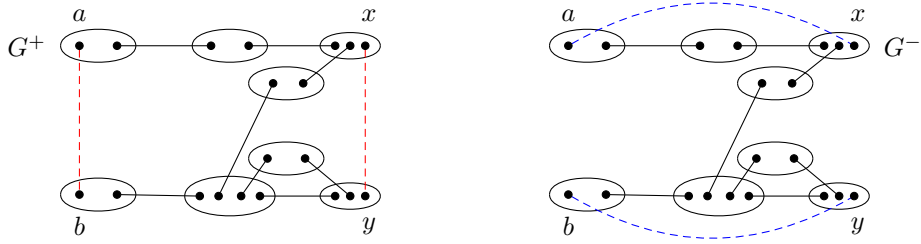
\refF{fig:counter} depicts an example 
where this heuristic fails: 
with a computer program one easily checks that in the relaxed $\bdn$-process (see~\refS{sec:relaxed}) 
we~have 
\[ \frac{\bP\bigpar{\Gpodn=G^+}}{\bP\bigpar{\Gpodn=G^-}}=\frac{Z(G^+)}{Z(G^-)}\approx 0.95652 <1.\]
For the reader worried that this counterexample might be an artifact of using the relaxed $\bdn$-process, 
we remark that if we interpret~$G^+$ and~$G^-$ as normal graphs 
(by contracting all points in an ellipse to one vertex), 
then with a computer program one easily checks that in the standard $\bdn$-process we also~have 
\[ \frac{\bP\bigpar{\Gpdn=G^+}}{\bP\bigpar{\Gpdn=G^-}}\approx 0.82164 < 1.\]

For switching arguments, the example from \refF{fig:counter} thus suggests that
very local consideration on the level of pairs of graphs do not suffice, 
i.e., that looking at larger sets of graphs at once seems to be necessary 
(in this paper we implement this by considering clusters, see~\refS{sec:cluster}).

\subsection{Proof of \refT{transferred}~\ref{transferred:uniform}: number of small edges~$\Xs$ in uniform model}\label{sec:uniformtransfer}
In this appendix we, for completeness, prove \refT{transferred}~\ref{transferred:uniform} using routine configuration model arguments. 
\begin{proof}[Proof of \refT{transferred}~\ref{transferred:uniform}]
Let $\Gcdn$ be the random graph obtained from the configuration model~\cite{BB1980,Wreg,FK} on the degree sequence~$\bdn$. As usual, $\Xs(G)$ denotes the number of small edges in~$G$. 
Since the maximum degree is~${\Delta=O(1)}$, it follows from well-known transfer results (see, e.g., \cite[Theorem~11.3]{FK}) that
\begin{equation}\label{eq:CM:transfer}
\bP\bigpar{|\Xs(\Gdn)- \mu|\ge \eps \mu}
 \le  O(1) 
\cdot \bP\bigpar{|\Xs(\Gcdn)- \mu|\ge \eps \mu} ,
\end{equation} 
so that we may henceforth focus on the number~$\Xs:=\Xs(\Gcdn)$ of small edges (which include loops) in the configuration model.
To this end, let $E_k$ denote the set of all possible small edges. Note that~${|E_k|= \binom{\stubs}{2}}$.
Writing $I_e$ is the indicator variable for the event that the edge~$e$ is in~$\Gcdn$, 
using~\eqref{keyassumption} and~$\eps \gg n^{-1/2}$ it follows that 
$\min\{\stubs,m\} = \Omega(n)$ as well~as
\begin{equation*}
\E \Xs = \sum_{e\in E_k} \E I_e
= \frac{|E_k|}{2m-1} 
= \bigpar{1+O\bigpar{n^{-1}}} \frac{\bigpar{\stubs}^2}{4m} = (1+o(\eps)) \mu. 
\end{equation*}
Gearing up towards an application of the bounded differences inequality for the configuration model 
(which follows from a simple modification of~\cite[Theorem~2.19]{Wreg} for regular~$\bdn$; see also~\cite[Section~1.1.4]{W16}), 
note that a switching operation can change the number~$\Xs$ of small edges by at most two. 
In view of~${\eps \gg n^{-1/2}}$, ${\Delta=\Theta(1)}$ and~${\mu=\Omega(n)}$, see~\eqref{keyassumption}, 
now a routine application of the bounded differences inequality to~$\Xs=\Xs(\Gcdn)$~yields 
\[ 
\bP\bigpar{|X_k-\mu|\ge \eps \mu} \le \bP\bigpar{|X_k-\bE X_k|\ge \eps \mu/2} \le 2\exp\left( \frac{-\Theta(\eps^2 \mu^2)}{n\Delta \cdot \Theta(1)}\right) 
\le \exp\Bigpar{-\Theta(\eps^2 n)} ,
\]
which together with~$\Xs=\Xs(\Gcdn)$ and inequality~\eqref{eq:CM:transfer} completes the proof of the desired estimate~\eqref{eq:transferred:uniform}. 
\end{proof}

\subsection{Transfer of \refT{main} to \refT{transferred}~\ref{transferred:process}}\label{sec:transfer}
In this appendix we show that our main technical result \refT{main} for the relaxed~\mbox{$\bdn$-process} 
transfers to the standard \mbox{$\bdn$-process}, i.e., implies our main result \refT{transferred}~\ref{transferred:process}.
This kind of transfer argument is conceptually standard, but the technical details are more involved than usual. 
Recall that, in both the final graph~$\Gpdn$ of the standard \mbox{$\bdn$-process} and the final graph~$\Gpodn$ of the relaxed~\mbox{$\bdn$-process}, 
we each time condition on having degree sequence~$\bdn$, i.e., we condition on the underlying process completing. 
The following lemma shows that the probability of completing is always at least a positive constant 
(stronger bounds can be deduced from~\cite{RW} and the arguments therein, but we do not need~this). 
\begin{lemma}[Completion probability]\label{completionprob}
For any integer~$\Delta \ge 1$ there are~$m_0,c>0$ such that, 
for any degree sequence $\bdn={\bigpar{d_1^{(n)},\dots, d_n^{(n)}}} \in {\{0, \ldots, \Delta\}^n}$ 
with even degree sum $\sum_{i \in [n]}d_i^{(n)}$ and at least~$m=m(\bdn)={\tfrac{1}{2}\sum_{i \in [n]}d_i^{(n)}} \ge m_0$ edges, 
the \mbox{(unconditional)} standard and relaxed $\bdn$-process each complete with probability at least~$c$.
\end{lemma}
As we shall see, this lemma (whose elementary proof we defer to the end of this appendix section) and~\refT{main} 
effectively reduce the proof of \refT{transferred}~\ref{transferred:process} to a comparison argument between the unconditional models. 
To this end we below write $\Hpdn{i}$ for the graph of the (unconditional) standard $\bdn$-process after~step~$i$, 
and also write $\Hpodn{i}$ for the (unconditional) relaxed $\bdn$-process after~step~$i$.
\begin{proof}[Proof of \refT{transferred}~\ref{transferred:process}]
Fix~$\alpha>0$ as in \refT{main}. 
The main idea is to compare the (unconditional) standard and relaxed $\bdn$-process after~$m_0 := {m-\floor{\min\{\alpha\mu/4,m/2\}}}$ steps,
exploiting that during the final ${m-m_0 \le \alpha\mu/4}$ steps 
the number of small edges can only change by at most most~$\alpha\mu/4$. 
Since the graph~$\Gpdn$ has the same distribution as~$\Hpdn{m}$ conditioned on completing,
by invoking \refL{completionprob} it follows that 
\begin{equation}\label{eq:transferred:anchor1}
\begin{split}
\bP\bigpar{\bigabs{\Xs\bigpar{\Gpdn}-\mu} \le \alpha \mu/4} 
& \le \frac{\bP\bigpar{\bigabs{\Xs\bigpar{\Hpdn{m}}-\mu} \le \alpha \mu/4}}{\bP\bigpar{\text{$\Hpdn{m}$ completes}}} \\
& \le O(1) \cdot \bP\bigpar{\bigabs{\Xs\bigpar{\Hpdn{m_0}}-\mu} \le \alpha \mu/2} .
\end{split}
\end{equation} 
Suppose that we condition on the graph~$\Hpodn{m_0}$ of the (unconditional) relaxed $\bdn$-process after~$m_0$ steps. 
Then the remaining~$m-m_0 = \Theta(\min\{\mu,n\}) \gg 1$ steps of the relaxed $\bdn$-process 
again correspond to a relaxed $\bdnp$-process with a modified degree sequence~$\bdnp$ (where the vertex degrees after~$m_0$ steps are subtracted off) with even degree sum, so by \refL{completionprob} the conditional probability of completing is at least~$c=\Omega(1)$. 
Summing over all graphs~$H$ satisfying $|\Xs(H)-\mu| \le \alpha \mu/2$ that can be attained by~$\Hpodn{m_0}$, 
it follows~that 
\begin{equation*}
\begin{split}
& \bP\bigpar{\bigabs{\Xs\bigpar{\Hpodn{m_0}}-\mu} \le \alpha \mu/2 \text{ and $\Hpodn{m}$ completes}}\\
& \qquad = \sum_{H}\bP\bigpar{\text{$\Hpodn{m}$ completes} \: \bigl| \: \Hpodn{m_0} = H }\bP\bigpar{\Hpodn{m_0} = H}\\
& \qquad \ge c \cdot \sum_{H}\bP\bigpar{\Hpodn{m_0} = H} = c \cdot \bP\bigpar{\bigabs{\Xs\bigpar{\Hpodn{m_0}}-\mu} \le \alpha \mu/2}.
\end{split}
\end{equation*} 
Recall that during the final ${m-m_0 \le \alpha\mu/4}$ steps the number of small edges can only change by at most most~$\alpha\mu/4$. 
Since the graph~$\Hpodn{m}$ conditioned on completing has the same distribution as~$\Gpodn$, 
by invoking~\refT{main} 
it follows that 
\begin{equation}\label{eq:transferred:anchor2}
\begin{split}
\bP\bigpar{|\Xs\bigpar{\Hpodn{m_0}}-\mu| \le \alpha \mu/2}
& \le \frac{\bP\bigpar{\bigabs{\Xs\bigpar{\Hpodn{m_0}}-\mu} \le \alpha \mu/2 \text{ and $\Hpodn{m}$ completes}}}{c}\\
& \le \frac{\bP\bigpar{\bigabs{\Xs\bigpar{\Hpodn{m}}-\mu} \le \alpha \mu \: \bigl| \: \text{$\Hpodn{m}$ completes}}}{c}\\
& \le O(1) \cdot \bP\bigpar{\bigabs{\Xs\bigpar{\Gpodn}-\mu} \le \alpha \mu} \le e^{-\Theta(n)} .
\end{split}
\end{equation}

In view of~\eqref{eq:transferred:anchor1}--\eqref{eq:transferred:anchor2} 
it remains to compare the probabilities in both unconditional processes after~$m_0$ steps, 
for which we shall use a standard step-by-step comparison argument (similar to~\cite{CRWhg,RWapbsr,JSdeg,WW}). 
To this end, for mathematical convenience we henceforth~write 
\begin{equation}\label{eq:projected}
\Hppdn{i} := \pi\bigpar{\Hpodn{i}}
\end{equation}
for the multigraph representation of the relaxed $\bdn$-process after~step~$i$, 
where the natural mapping~$\pi$ is defined as in \refS{sec:configurationgraph} (it simply contracts corresponding points of the relaxed $\bdn$-process to vertices). 
We now fix a graph sequence $(G_i)_{0 \le i \le m_0}$ satisfying $|\Xs(G_{m_0})-\mu| \le \alpha \mu/2$ 
that can be attained by~$(\Hpdn{i})_{0 \le i \le m_0}$. 
Comparing the probability with which the next edge is added in each process, for~$0 \le i < m_0$ it follows~that 
\begin{equation}\label{eq:transferred:step}
\bP\Bigpar{\Hpdn{i+1}=G_{i+1} \: \Big| \: \bigcap_{0\le j\le i}\bigcpar{\Hpdn{j}=G_j}} 
\: = \: 
\frac{\binom{\Gamma_i}{2} }{|\cQ_i|} \cdot \bP\Bigpar{\Hppdn{i+1}=G_{i+1} \: \Big| \: \bigcap_{0\le j\le i}\bigcpar{\Hppdn{j}=G_j}},
 \end{equation}
where~$\cQ_i=\binom{\Gamma_i}{2}\backslash E(G_i)$ denotes the set of edges that can be added to~$\Hpdn{i}$, 
and~$\Gamma_i$ denotes the number of unsaturated vertices in~$G_i$ (as before).  
Note that $|E(G_i)\cap \binom{\Gamma_i}{2}|\le \Gamma_i \cdot \Delta$. 
The argument of Observation~\ref{obs:unsaturated} gives $\Gamma_i\ge 2(m-i)/\Delta \ge 2(m-m_0)/\Delta \gg 1$. 
Using $1-x\ge e^{-2x}$ for~$x\in [0,1/2]$, it follows~that 
\[|\cQ_i|\ge \binom{\Gamma_i}{2}-\Gamma_i\Delta=\binom{\Gamma_i}{2} \cdot \left(1-\frac{2\Delta}{\Gamma_i-1}\right)
\ge \binom{\Gamma_i}{2} \cdot \exp\left(\frac{-4\Delta^2}{m-m_0}\right).
\]
Since initially~$\bP(\Hpdn{0}=G_{0})=1=\bP(\Hpodn{0}=G_{0})$, 
by multiplying~\eqref{eq:transferred:step} it follows~that 
\begin{equation}\label{eq:transferred:sequence}
\bP\Bigpar{\bigcap_{0\le i\le m_0}\bigcpar{\Hpdn{i}=G_i}} 
 \le 
\exp\left(\frac{4\Delta^2m_0}{m-m_0}\right) \cdot \bP\Bigpar{\bigcap_{0\le i\le m_0}\bigcpar{\Hppdn{i}=G_i}}.
\end{equation}
Recalling~$\Delta=O(1)$, by definition of~$m_0$ we have~$\Delta^2m_0/(m-m_0)=O(1)$. 
Summing~\eqref{eq:transferred:sequence} over all graph sequences $(G_i)_{0 \le i \le m_0}$ satisfying $|\Xs(G_{m_0})-\mu| \le \alpha \mu/2$ 
that can be attained by~$(\Hpdn{i})_{0 \le i \le m_0}$, it follows~that
\begin{equation}\label{eq:transferred:transfer}
\begin{split}
\bP\bigpar{\bigabs{\Xs(\Hpdn{m_0})-\mu} \le \alpha \mu/2} 
& \le O(1) \cdot \bP\bigpar{\bigabs{\Xs(\Hppdn{m_0})-\mu} \le \alpha \mu/2} \\
& \le   
O(1) \cdot \bP\bigpar{\bigabs{\Xs(\Hpodn{m_0})-\mu} \le \alpha \mu/2} ,
\end{split}
\end{equation}
where for the last step we used~$\Xs(\Hppdn{m_0})=\Xs(\Hpodn{m_0})$, i.e., that~$\Hpodn{m_0}$ and its multigraph representation~$\Hppdn{m_0}=\pi(\Hpodn{m_0})$ contain the same number of small edges.

Finally, combining~\eqref{eq:transferred:transfer} with~\eqref{eq:transferred:anchor1}--\eqref{eq:transferred:anchor2} 
completes the proof of inequality~\eqref{eq:transferred:process} with~$\beta := \alpha/4$. 
\end{proof}

We conclude by giving the deferred proof of \refL{completionprob}, which is based on counting arguments. 
\begin{proof}[Proof of \refL{completionprob}]	
We first consider the standard $\bdn$-process. 
When this process fails to complete, we are left with an edge-sequence $\hsigma$ with less than~$m$ edges, to which no more edges can be added without violating the degree constraint (not allowing multiedges). 
We denote the set of such `stuck' edge-sequences be~$S=S(\bdn)$, and write~$\bP(\sigma)$ for the probability that the standard $\bdn$-process produces the edge-sequence~$\sigma$. 

We claim that there is a function~$g$ such that the following holds for any $\hsigma\in S$:
the edge-sequence~$g(\hsigma)$ completes, is identical with~$\hsigma$ in the first $m_\Delta := m-3\Delta^3$ steps, 
and we also~have 
\begin{equation*}
 \bP(\hsigma) \le D \cdot \bP(g(\hsigma)) \quad \text{and}\quad \bigabs{g^{-1}(g(\hsigma))} \le D 
\qquad \text{for} \qquad  D:=(6\Delta^3)^{6\Delta^3} .
\end{equation*}
Writing~$\im(g)$ for the image of~$g$, using this claim it follows that $\sum_{\hsigma\in S} \bP(\hsigma)\le D^2\sum_{\sigma\in\im(g)} \bP(\sigma)$ and  
\begin{equation}\label{deltalowerbound}
\bP\bigpar{\text{$\bdn$-process completes}} \ge 
\frac{\sum_{\sigma\in\im(g)} \bP(\sigma)}{\sum_{\sigma\in\im(g)} \bP(\sigma)+\sum_{\hsigma\in S} \bP(\hsigma)}	 
\ge \frac{1}{1+D^2},
\end{equation}
which establishes the desired lower bound on the probability of completing, with~$c:=1/(1+D^2)$.

In the remainder we show the existence of such $g$ with the claimed properties 
(to ensure that all described steps are feasible, we shall always tacitly use that the number of edges~$m \ge m_0$ is sufficiently large, whenever necessary). 
Consider~$\hsigma\in S$ and the graph~$H$ containing all edges in~$\hsigma$. 
Note that the remaining set~$U$ of unsaturated vertices of $H$ are pairwise adjacent (otherwise we can add more edges). 
Since the maximum degree in~$\bdn$ is $\Delta$, it follows that there are at most~$|U| \le \Delta$ unsaturated vertices at the end. 
Hence the number~$\ell$ of missing edges is at most~$\ell \le |U| \cdot \Delta \le \Delta^2$. 
We shall next pick $\ell$ pairwise disjoint edges whose endpoints are not adjacent to any vertices of~$U$. 
To this end we initialize $F=N(U)$ as the set of forbidden vertices. 
At each step $i=1,\dots, \ell$, we then pick from~$\hsigma$ the latest edge~$u_iv_i$ whose endpoints are disjoint from~$F$, and add~$u_iv_i$ to~$F$. 
Let $(a_1,\dots, a_{2\ell})$ be an arbitrary but deterministic enumeration of the multiset of unsaturated vertices of $\hsigma$ where unsaturated vertex $v$ appears $d_v^{(n)}-\deg_H(v)$ times.
We now construct the edge-sequence $g(\hsigma)$ by replacing the edge $u_{i}v_{i}$ in~$\hsigma$ with the sequence of edges $u_{i}a_{2i-1}, v_{i}a_{2i}$.
By construction, the resulting new edge-sequence $g(\hsigma)$ completes.
Furthermore, since each edge $u_iv_i$ appears after step $m-(\ell\Delta + \ell\cdot 2\Delta) \ge m-3\Delta^3=m_{\Delta}$ in~$\hsigma$, 
the edge-sequence~$g(\hsigma)$ is identical with~$\hsigma$ in the first~$m_{\Delta}$ steps.
Since after~$m_{\Delta}$ steps there at most most~$2(m-m_{\Delta})=6\Delta^3$ unsaturated vertices left, 
it follows that 
\begin{equation}\label{smallpreimage}
\bigabs{g^{-1}(g(\hsigma))} \le \binom{6\Delta^3}{2}^{m-m_{\Delta}} \le (6\Delta^3)^{6\Delta^3} = D.
\end{equation}
We write~$\cQ_i(\sigma)$ for the set of edges that can be added in step~$i+1$ of the standard $\bdn$-process after adding the first~$i$ edges in $\sigma$. 
By construction, it follows that 
\begin{equation}\label{failurecontrol}
\frac{\bP(\hsigma)}{\bP(g(\hsigma))} = \frac{\prod_{m_{\Delta} \le i < m-\ell}|\cQ_i(\hsigma)|^{-1}}{\prod_{m_{\Delta} \le i < m}|\cQ_i(g(\hsigma))|^{-1}}
\le \prod_{m_{\Delta} \le i < m}|\cQ_i(g(\hsigma))| \le \binom{6\Delta^3}{2}^{m-m_{\Delta}} \le D, 
\end{equation}
which completes the proof of the lower bound~\eqref{deltalowerbound}, as discussed.

Finally, we sketch the similar argument for the relaxed $\bdn$-process, 
where it again suffices to show that any stuck edge-sequence can be changed into a completing edge-sequence without changing the choices made in the first $m_\Delta = m-O(1)$ steps.
Note that in the relaxed $\bdn$-process only one vertex remains unsaturated, so the number of missing edges is~$\ell \le \Delta$.
Then we can again find the last~$\ell$ edges in~$\hsigma$ not incident to the unsaturated vertex, and use a similar edge-replacement procedure to construct suitable~$g(\hsigma)$, 
which then establishes the desired constant lower bound on the probability of completing. 
\end{proof}

\subsection{Proof of Remark~\ref{rem:lowertail}: lower tail of~$\Xs$ in degree restricted process}\label{sec:lowertail}
In this appendix we establish Remark~\ref{rem:lowertail} 
by complementing \refT{main} with a lower tail result for the number $\Xs(\Gpodn)$ of small edges, 
for which a simpler variant of the proof of \refT{main} suffices.
In particular, we can avoid the notion of good clusters and only rely on the basic cluster switching result~\refL{switch},   
by exploiting that configuration-graphs~$G$ which at most~$\Xs(G) \le (1-\eps)\mu$ small edges are contained in considerably more lower than upper clusters 
(in contrast to the proof of \refT{main}, where the relevant configuration-graphs~$G$ with~$\Xs(G) \approx \mu$ are contained in approximately the same number of lower and upper~clusters). 
\begin{theorem}[Number of small edges: lower tail]\label{location}
Suppose that assumptions of \refT{transferred} hold.
Then, for any~$\eps \in (0,1)$, 
the relaxed random graph process~$\Gpodn$ satisfies 
${\bP\bigpar{\Xs\bigpar{\Gpodn} \le (1-\eps)\mu} \le e^{-\Omega(n)}}$.  
\end{theorem}
\begin{proof}
We proceed similarly to \refT{main}, though some details are simpler here. 
We henceforth write~$\Xs(G)$ for the number of small edges in~$G$, and assume that~${\ell \le (1-\eps/2)\mu}$. 
Inspired by \refS{sec:doublecounting}, let~$\cL$ denote the set of all lower clusters~$C^-$ that contain some configuration-graph~$G$ with ${\Xs(G) \le \ell}$ small edges, 
and let~$\cU$ contain all upper clusters~$C^+$ that are switching-partners of the lower clusters~$C^-$ in~$\cL$.
For any~$G$ with ${\Xs(G) \le \ell}$, note that inequality~\eqref{eq:LG} remains valid, which in turn implies~${L_G \ge 1}$ and thus ${G \in \bigcup_{C^- \in \cL}C^-}$. 
The proof of Observation~\ref{obs4D} shows that any ${G \in \bigcup_{C^+ \in \cU}C^+}$ has ${\Xs(G) \le \ell+4\Delta}$ small edges. 
Proceeding similarly to~\eqref{bound:ZLZU} and~\eqref{eq:ZL:ZU:relevance}, by invoking \refL{switch} to all switching-partners it follows~that 
\begin{equation*}
\sum_{G:\Xs(G)\le \ell} L_G Z(G) \le Z(\cL) \le Z(\cU) \le \sum_{G:\Xs(G)\le \ell+4\Delta} U_G Z(G).
\end{equation*}
Since $\eps \mu =\Theta(n)$, we have~$\ell+4\Delta \le (1-2\eps/3)\mu$, say. 
For any configuration-graph~$G$ with~$\Xs(G) \le \ell+4\Delta$ small edges, 
by carefully inspecting equations~\eqref{eq:UG}--\eqref{LUcoefficient} for~$U_G$ and~$L_G$ it is routine to see 
(by combining that~\eqref{LUcoefficient} implies~$L_G/U_G=1+o(1)$ for~$\ell=\mu$, 
with the observation that the estimates~\eqref{eq:UG} for~$U_G$ and~\eqref{eq:LG} for~$L_G$ are increasing and decreasing in~$\ell$, respectively)  
that there is~$\tau=\tau(\eps) \in (0,1)$ such that 
\begin{equation*}
L_G \ge (1-\tau) \bigpar{{\textstyle\stubs}-2\mu}^2 \quad \text{ and } \quad U_G \le (1-\tau)^2 \bigpar{{\textstyle\stubs}-2\mu}^2 .
\end{equation*}
Putting things together, for any~${\ell \le (1-\eps/2)\mu}$ it follows that
\[
\sum_{G:\Xs(G)\le \ell} Z(G)  \le (1-\tau) \cdot \sum_{G:\Xs(G)\le \ell+4\Delta} Z(G) .
\]
By iterating this inequality, it then follows similarly to equation~\eqref{eq:main:goal:cor} that 
\begin{align*}
\bP\bigpar{\Xs\bigpar{\Gpodn} \le (1-\eps)\mu}  &\le  
\frac{\sum_{G:\Xs(G)\le (1-\eps)\mu} Z(G)}{ \sum_{G:\Xs(G)\le (1-\eps)\mu + 4\Delta \cdot \floor{\eps \mu/(8\Delta)}} Z(G)} \: \le \: (1-\tau)^{\floor{\eps \mu/(8\Delta)}}  \le e^{-\Theta(n)},
\end{align*}
completing the proof. 
\end{proof}
Finally, the lower tail result \refT{location} for~$\Gpodn$ of course 
also transfers to the standard \mbox{$\bdn$-process}~$\Gpdn$ 
(the comparison arguments from~\refS{sec:transfer} go through mutatis~mutandis).

\end{appendix}
\end{document}